\documentclass[a4paper,11pt]{article}
\usepackage[english]{babel}
\usepackage{nccmath}
\usepackage{amsmath}
\usepackage{amsfonts}
\usepackage{amssymb}
\usepackage{mathrsfs}
\usepackage{amsthm}
\usepackage[all]{xy}
\usepackage{graphicx}
\usepackage{subfigure}

\newcommand{\R}{{{\mathbb R}}}
\newcommand{\re}{{{\mathfrak{R}}}}
\newcommand{\im}{{{\mathfrak{I}}}}

\newcommand{\T}{{{\mathbb T}}}
\newcommand{\C}{{{\mathbb C}}}
\newcommand{\z}{{{z_1}}}
\newcommand{\s}{{{L^{2}(S)}}}
\newcommand{\zz}{{{z_2}}}
\newcommand{\F}{{{\mathcal{F}(z)}}}
\providecommand{\abs}[1]{\lvert#1\rvert}
\providecommand{\norm}[1]{\lVert#1\rVert}
\providecommand{\de}[1]{\partial^{4}_{\alpha}z#1}
\providecommand{\detil}[1]{\partial^{4}_{\alpha}\tilde{z}#1}
\providecommand{\dee}[1]{\partial^{5}_{\alpha}z#1}
\providecommand{\est}[1]{\exp C(\norm{\F}^{2}_{L^{\infty}(S)}+\norm{z}^{2}_{H^{#1}(S)})}

\providecommand{\esttil}[1]{\exp C(\norm{\mathcal{F}(\tilde{z})}^{2}_{L^{\infty}(S)}+\norm{\tilde{z}}^{2}_{H^{#1}(S)})}
\theoremstyle{plain}
\newtheorem{lem}{Lemma}[section]

\newtheorem{thm}{Theorem}[section]
\newtheorem{nota}{Remark}[section]

\theoremstyle{definition}

\title{Non-splat singularity for the one-phase Muskat problem}
\author{Diego C\'ordoba and Tania Pernas-Casta\~{n}o}

\begin{document}
\maketitle

\begin{abstract}
For the water waves equations, the existence of splat singularities has been shown in \cite{finite}, i.e., the interface self-intersects along an arc in finite time. The aim of this paper is to show the absence of splat singularities for the incompressible fluid dynamics in porous media.
\end{abstract}

 \section{Introduction}

	The Muskat problem \cite{muskat} models the evolution of the interface between two fluids of different characteristics in porous media, where the velocity of the fluid is given by Darcy's law:
	
	\begin{displaymath}
	\frac{\mu}{\kappa}u=-\nabla p-(0,g\rho)
	\end{displaymath}
	where $(x,t)\in\R^{2}\times\R^{+}$, $u=(u_{1}(x,t),u_{2}(x,t))$ is the incompressible velocity (i.e. $\nabla\cdot u=0$), $p=p(x,t)$ is the pressure, $\mu=\mu(x,t)$ is the dynamic viscosity, $\kappa$ is the permeability of the isotropic medium, $\rho=\rho(x,t)$ is the liquid density and $g$ is the acceleration due to gravity. The free boundary is caused by the discontinuity between the densities and viscosities of the fluids; the quantities $(\mu,\rho)$ are defined by
	\begin{displaymath}
(\mu,\rho)(x_{1},x_{2},t) := \left\{ \begin{array}{ll}
(\mu^{1},\rho^{1}) & \textrm{$x\in\Omega^{1}(t)$}\\
(\mu^{2},\rho^{2}) & \textrm{$x\in\Omega^{2}(t)=\R^{2}-\Omega^{1}(t)$}
\end{array} \right.
\end{displaymath}
where $\mu^{1}$, $\rho^{1}$, $\mu^{2}$ and $\rho^{2}$ are constants.

\begin{figure}[htb]
\centering
\subfigure[Splash singularity]{\includegraphics[width=62mm]{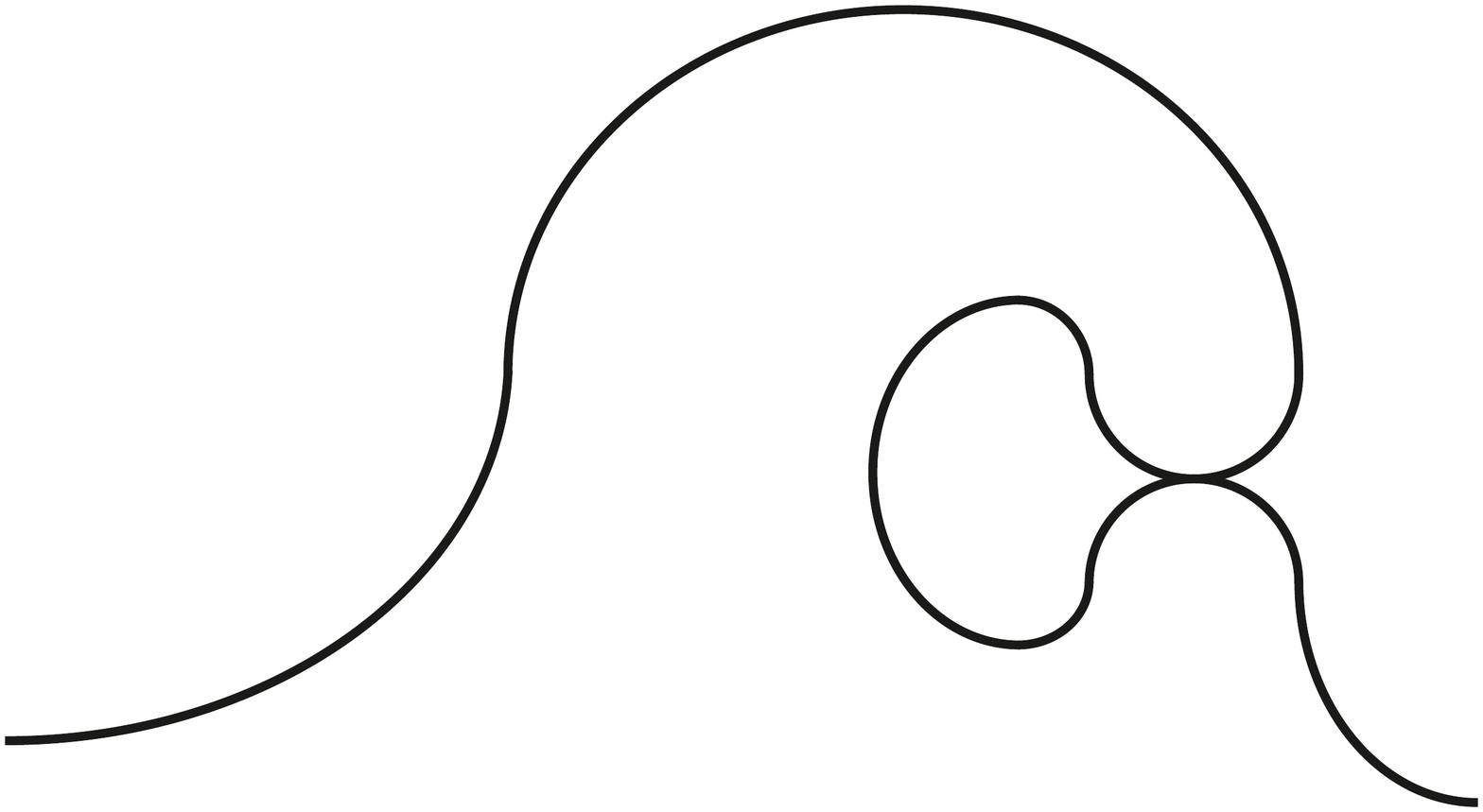}}
\subfigure[Splat singularity]{\includegraphics[width=62mm]{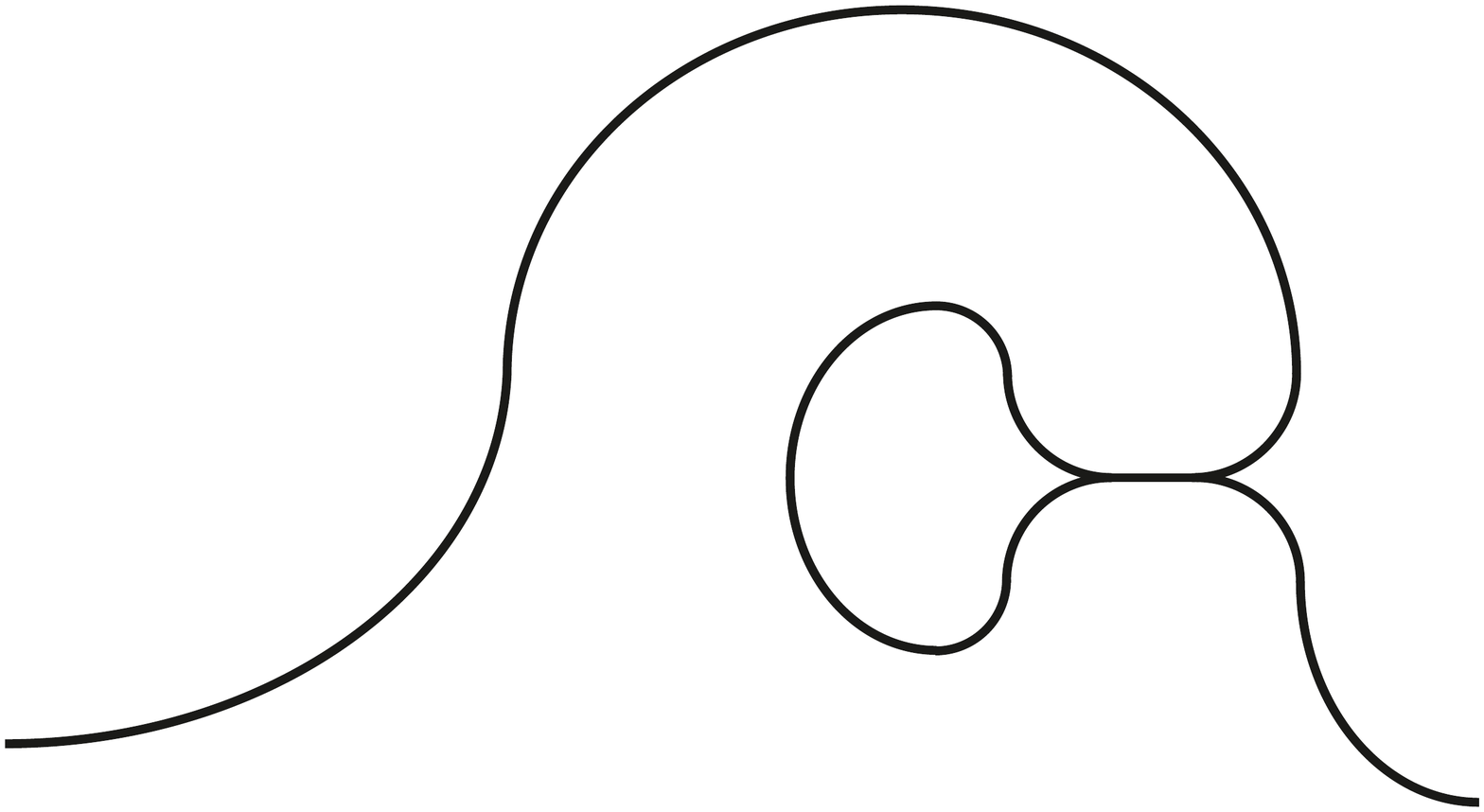}}
\caption{Finite time singularities} \label{fig:singu}
\end{figure}

	We will only study one of the two types of finite time singularities shown for water waves in \cite{finite}, the splash and splat singularities.
	The splash-type singularity (Figure \ref{fig:singu}(a)) corresponds to the case where the fluid interface self-intersects at a single point. This kind of singularity also occurs for the Muskat problem as proved in \cite{splash}.
	
	In this paper, we will focus on the splat-type singularity (Figure \ref{fig:singu}(b)). This singularity is a variation of the former in which the fluid interface self-intersects along an arc. This scenario has been shown to arise for the incompressible Euler equations in the water waves form, see \cite{finite}, which considers the evolution of the free boundary of a water region in vacuum and irrotational velocity. In \cite{3d}, these singularities have also been shown to exist for the case with vorticity.

For the Muskat problem, splash singularity cannot be developed in the case in which $\mu^{1}=\mu^{2}$ and $\rho^{1}\neq\rho^{2}$, for more details see \cite{absence}. For similar results about two-fluids interfaces see \cite{nosplash}, \cite{impossi}. However, the splash can be achieved with $\mu^{1}=\rho^{1}=0$ where $\R^{2}-\Omega(t)$ corresponds to the dry region (see \cite{splash}). 

The aim of this work is to show the absence of splat singularities in the case of an interface between an incompressible irrotational fluid and a dry region in porous media. Thus, $\mu^{1}=\rho^{1}=0$, i.e.,
\begin{displaymath}
(\mu,\rho)(x_{1},x_{2},t) := \left\{ \begin{array}{ll}
(\mu^{2},\rho^{2}) & \textrm{$x\in\Omega(t)$}\\
(0,0) & \textrm{$x\in\R^{2}-\Omega(t)$.}
\end{array} \right.
\end{displaymath}

Let the free boundary be parametrized by
\begin{displaymath}
\partial\Omega=\lbrace z(\alpha,t)=(\z(\alpha,t),\zz(\alpha,t)):\alpha\in\R\rbrace
\end{displaymath}
so that the periodic condition 
\begin{displaymath}
(\z(\alpha+2k\pi,t),\zz(\alpha+2k\pi,t))=(\z(\alpha,t)+2k\pi,\zz(\alpha,t))
\end{displaymath}
holds with initial data $z(\alpha,0)=z_{0}(\alpha)$.

From Darcy's law, we deduce that the fluid is irrotational, i.e. $\omega=\nabla\times u=0$, in the interior of the domain $\Omega$. Therefore, the vorticity is concentrated on the free boundary $z(\alpha,t)$ by a Dirac distribution as follows:
\begin{displaymath}
\omega(x,t)=\nabla^{\bot}\cdot u(x,t)=\varpi(\alpha,t)\delta(x-z(\alpha,t))
\end{displaymath}
where $\varpi(\alpha,t)$ represents the vorticity strength.

The interface $z(\alpha,t)$ evolves with an incompressible velocity field satisfying the Biot-Savart law, which can be explicitly computed and is given by the Birkhoff-Rott integral of the amplitude $\varpi$ along the interface $z(\alpha,t)$:
\begin{equation}
\label{br}
BR(z,\varpi)(\alpha,t)=\frac{1}{2\pi}PV\int_{\R}{\frac{(z(\alpha,t)-z(\beta,t))^{\bot}}{\abs{z(\alpha,t)-z(\beta,t)}^{2}}\varpi(\beta,t)d\beta}.
\end{equation}
We can subtract any term in the tangential direction to the curve in the velocity field without modifying the geometric evolution of the curve
\begin{equation}
\label{zt}
z_{t}(\alpha,t)=BR(z,\varpi)(\alpha,t)+c(\alpha,t)\partial_{\alpha}z(\alpha,t).
\end{equation}

A wise choice of $c(\alpha,t)$, namely
\begin{align}
\label{c}
c(\alpha,t)=&\frac{\alpha+\pi}{2\pi}\int_{\T}{\frac{\partial_{\beta}z(\beta,t)}{\abs{\partial_{\beta}z(\beta,t)}^{2}}\cdot\partial_{\beta}BR(z,\varpi)(\beta,t)d\beta}\nonumber\\
&-\int_{-\pi}^{\alpha}\frac{\partial_{\beta}z(\beta,t)}{\abs{\partial_{\beta}z(\beta,t)}^{2}}\cdot\partial_{\beta}BR(z,\varpi)(\beta,t)d\beta
\end{align}
allows us to remove the dependence on $\alpha$ from the length of the tangent vector to $z(\alpha,t)$ (for more details see \cite{hele}):
\begin{displaymath}
\abs{\partial_{\alpha}z(\alpha,t)}^{2}=A(t).
\end{displaymath}

We can close the system using Darcy's law and taking the dot product with $\partial_{\alpha}z(\alpha,t)$. It is easy to relate $\varpi$ and the free boundary by (see \cite{hele}):
\begin{equation}
\label{vor}
\varpi(\alpha,t)=-2BR(z,\varpi)(\alpha,t)\cdot\partial_{\alpha}z(\alpha,t)-2\kappa g\frac{\rho^{2}}{\mu^{2}}\partial_{\alpha}\zz(\alpha,t).
\end{equation}

For the stability of the problem we consider the Rayleigh-Taylor condition. Rayleigh \cite{ray} and Saffman-Taylor \cite{tay} gave a condition that must be satisfied for the linearized model in order to have a solution locally in time, namely that the normal component of the pressure gradient jump at the interface has to have a distinguished sign. This condition can be written as
\begin{displaymath}
\sigma(\alpha,t)=\frac{\mu^{2}}{\kappa}BR(z,\varpi)(\alpha,t)\cdot\partial^{\bot}_{\alpha}z(\alpha,t)+g\rho^{2}\partial_{\alpha}\z(\alpha,t)>0.
\end{displaymath}
 Using Hopf's lemma, the Rayleigh-Taylor condition is satisfied for $\mu^{1}=\rho^{1}=0$ (see \cite{splash}). For the case of equal viscosities ($\mu^{1}=\mu^{2}$), this condition holds when the more dense fluid lies below the interface \cite{raytay}.
 
 This stability has been used to prove local existence in Sobolev spaces, when $\mu^{1}\neq\mu^{2}$ and $\rho^{1}\neq\rho^{2}$, in \cite{hele}. For improvements for local existence results in the case $\mu^{1}=\rho^{1}=0$, see \cite{graner}. When $\mu^{1}=\mu^{2}$ there is local existence and instant analyticity in the stable case, see \cite{raytay} and \cite{contour}. For small data, the fact that $\sigma>0$ has been used to prove global existence as we can check in \cite{global}, \cite{global2}, \cite{para} and \cite{Granero}. Furthermore, there exists initial data with $\sigma>0$ that in finite time turns to $\sigma<0$ (see \cite{raytay} and \cite{Granero-Gomez}) and later in finite time the interface breaks down \cite{Breakdown}.

 Finally we introduce the function that measures the arc-chord condition
\begin{displaymath}
\F(\alpha,\beta,t)=\frac{\beta^{2}}{\abs{z(\alpha)-z(\alpha-\beta)}^{2}},\quad\alpha,\beta\in\R
\end{displaymath}
with 
\begin{displaymath}
\F(\alpha,0,t)=\frac{1}{\abs{\partial_{\alpha}z(\alpha,t)}^{2}}.
\end{displaymath}
The main theorem of this paper is the following:
\begin{thm}
\label{main}
Let $z_{0}(\alpha)\in H^{k}(\T)$ for $k\ge 4$ and $\mathcal{F}(z_{0})(\alpha,\beta)\in L^{\infty}$ . Then the Muskat problem (\ref{br}-\ref{vor}) will not break down in a splat singularity, i.e., there is no time where there exist disjoint intervals $I_{1},I_{2}\in\R$ such that $z(I_{1},t)=z(I_{2},t)$.  
\end{thm}

In order to prove this theorem we have organized the paper as follows.

 In sections \ref{esti}, \ref{min} and \ref{inst} we present several a priori estimates that provide instant analyticity for a single curve that initially satisfies the arc-chord and Rayleigh-Taylor conditions. Section \ref{decai} is devoted to prove that the region of analyticity does not collapse to the real axis as long as the curve remains smooth and the arc-chord condition remains bounded.
 
 Instant analyticity and exponential decay of the strip of analyticity is shown in \cite{raytay} for the case where both fluids have equal viscosities ($\mu^{1}=\mu^{2}$). In such case, the formula for the strength of the vorticity is simpler 
 
\begin{displaymath}
\varpi(\alpha,t)=-(\rho^{2}-\rho^{1})\partial_{\alpha}\zz(\alpha,t).
\end{displaymath}

In our scenario, the one-phase Muskat problem, the expression (\ref{vor}) of the strength of the vorticity involves the Birkhoff-Rott integral.

Finally in section \ref{nosplat}, we prove the main theorem using a contradiction argument. The idea of the proof is the following:

Suppose that there exists a splat singularity at time $T$. If the solution $z(\alpha,t)$ is real-analytic at time $T$, the formation of a splat singularity would be impossible. This follows from the fact that we would get a real-analytic curve $z(\alpha,t)$ self-intersecting along an arc, therefore $z(\alpha,t)$ should self-intersect at all points.

Since the curve self-intersects, the arc-chord condition fails in our domain $\Omega$, and thus we have no control on the decay of the strip of analyticity. In order to get around this issue it is necessary to apply a transformation defined by $\tilde{z}(\alpha,t))=P(z(\alpha,t))$ where $P$ is a conformal map (see \cite{finite}):
\begin{displaymath}
P(w)=(\tan(\frac{w}{2}))^{\frac{1}{2}}.
\end{displaymath}

This conformal map transforms our domain $\Omega$ in $\tilde{\Omega}$ as we can see in Figure \ref{fig:trans}. The branch of the root will be taken in such a way that it separates the self-intersecting points of the interface.

\begin{figure}[htb]
\centering
\subfigure[$\Omega(T)$ domain]{\includegraphics[width=62mm]{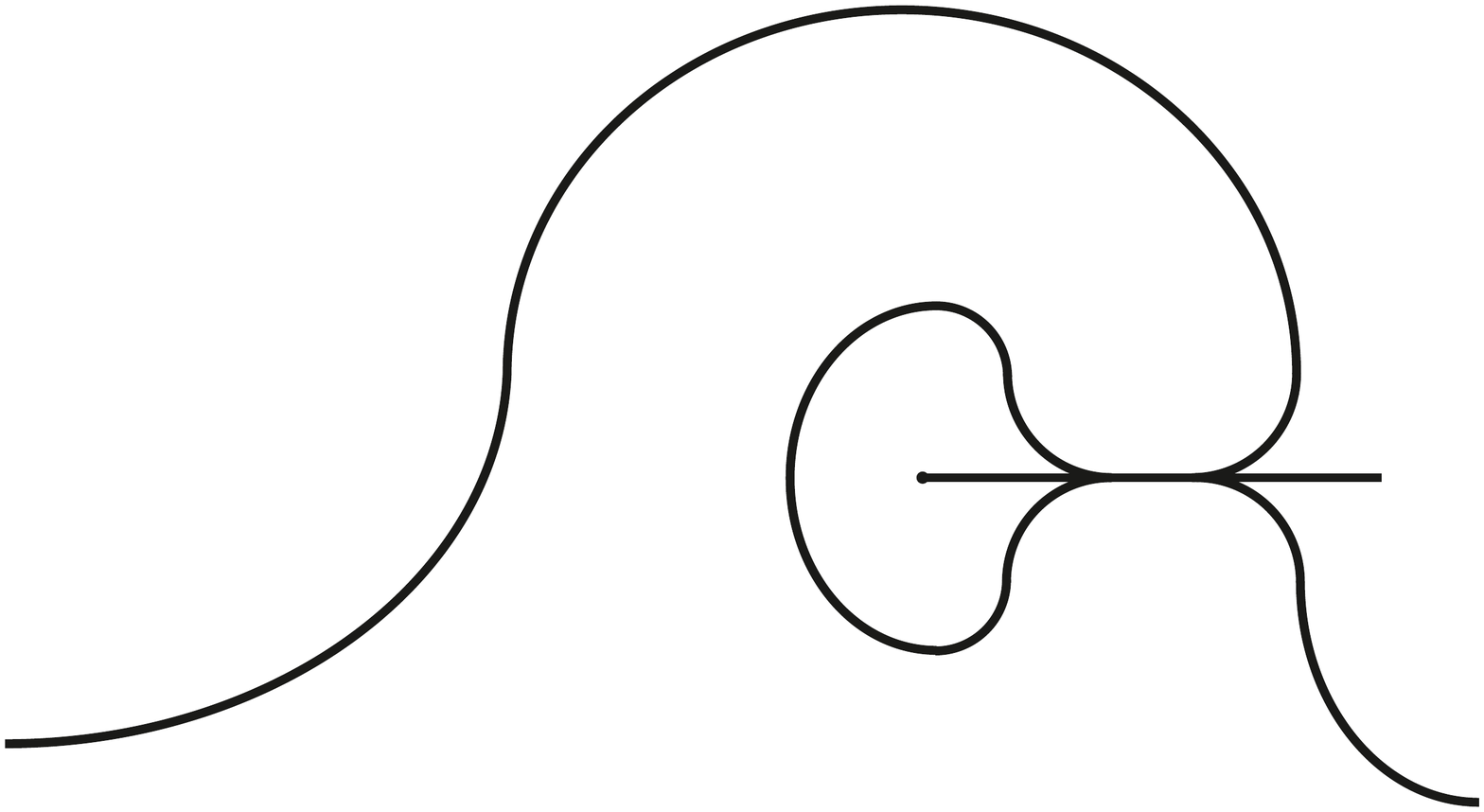}}
\subfigure[$\tilde{\Omega}(T)$ domain]{\includegraphics[width=62mm]{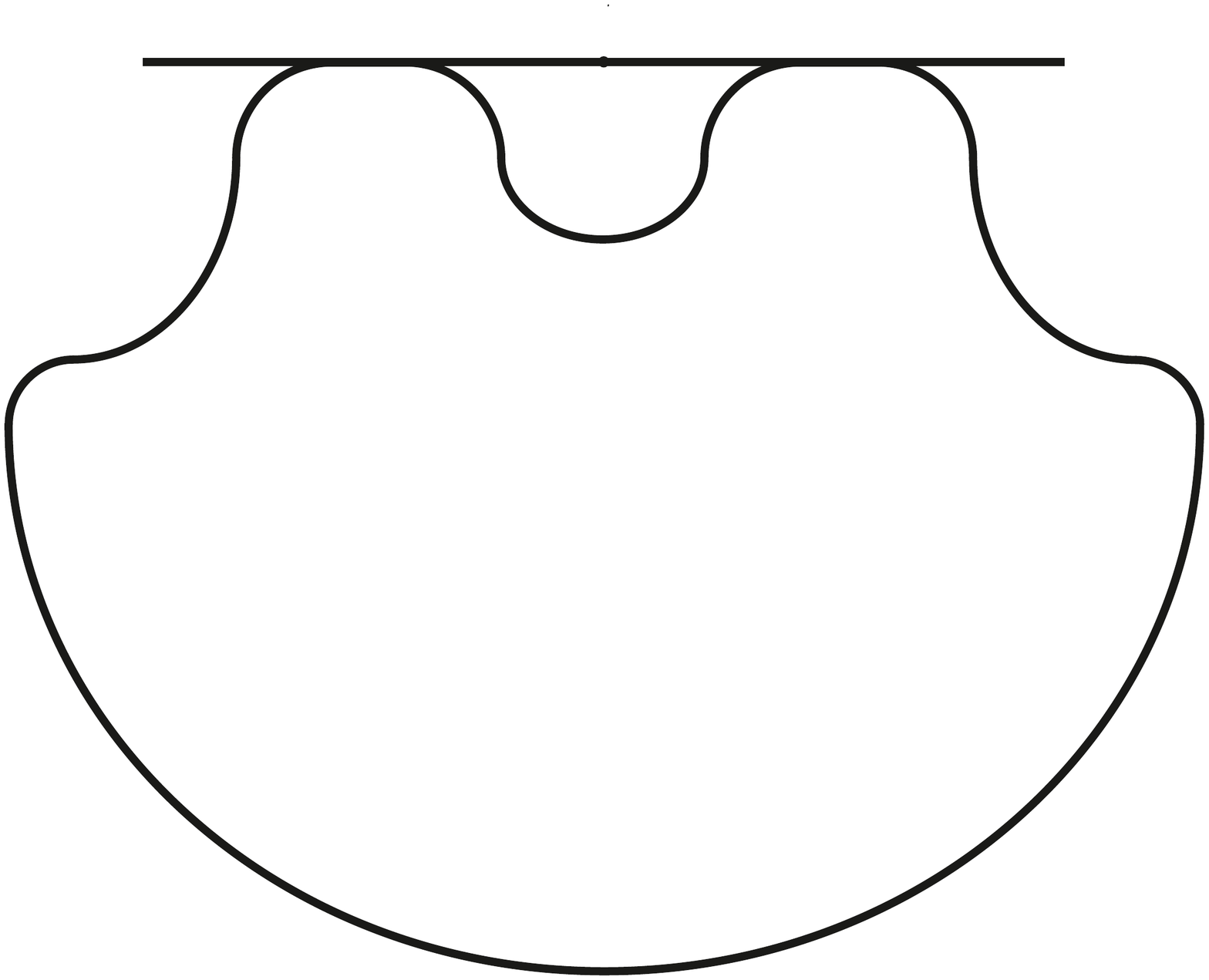}}
\caption{Finite time singularities} \label{fig:trans}
\end{figure}

The new contour evolution equation where we handle the splat singularity is (see \cite{splash} for more details):

\begin{equation*}
\tilde{z}_{t}(\alpha,t)=Q^{2}(\alpha,t)BR(\tilde{z},\tilde{\varpi})(\alpha,t)+\tilde{c}(\alpha)\partial_{\alpha}\tilde{z}(\alpha,t)
\end{equation*}
where
\begin{equation*}
Q^{2}(\alpha,t)=\abs{\frac{dP}{dw}(z(\alpha,t))}^{2}=\abs{\frac{dP}{dw}(P^{-1}(\tilde{z}(\alpha,t)))}^{2},
\end{equation*}
\begin{equation*}
\tilde{\varpi}(\alpha,t)=-2BR(\tilde{z},\tilde{\varpi})(\alpha,t)\cdot\partial_{\alpha}\tilde{z}(\alpha,t)-2\frac{\rho^{2}}{\mu^{2}}\partial_{\alpha}(P_{2}^{-1}(\tilde{z}(\alpha,t)))
\end{equation*}
and
\begin{align*}
&\tilde{c}(\alpha,t)=\frac{\alpha+\pi}{2\pi}\int_{\T}{\frac{\partial_{\beta}\tilde{z}(\beta,t)}{\abs{\partial_{\beta}\tilde{z}(\beta,t)}^{2}}\cdot\partial_{\beta}BR(\tilde{z},\tilde{\varpi})(\beta,t)d\beta}\nonumber\\
&-\int_{-\pi}^{\alpha}\frac{\partial_{\beta}\tilde{z}(\beta,t)}{\abs{\partial_{\beta}\tilde{z}(\beta,t)}^{2}}\cdot\partial_{\beta}BR(\tilde{z},\tilde{\varpi})(\beta,t)d\beta.
\end{align*}

Finally we find the Rayleigh-Taylor condition in terms of $\tilde{z}$
\begin{equation}
\label{rttild}
\tilde{\sigma}(\alpha,t)=\frac{\mu^{2}}{\kappa}BR(\tilde{z},\tilde{\varpi})(\alpha,t)\cdot\partial^{\bot}_{\alpha}\tilde{z}(\alpha,t)+\rho^{2}g\nabla P_{2}^{-1}(\tilde{z}(\alpha,t))\cdot\partial^{\bot}_{\alpha}\tilde{z}(\alpha,t).
\end{equation}

Our final goal in section \ref{nosplat} is to prove instant analyticity and decay of the strip of analyticity for the Muskat problem in the new domain, which allows us to apply our argument of non-splat, i.e., to prove Theorem \ref{main}.
  \section{Estimates on $z(\alpha,t)$}
  \label{esti}
  
  Here we show the main estimates that provide instant analyticity into the strip $S(t)=\{\alpha+i\zeta:\abs{\zeta}<\lambda t\}$ for each t. To do that we will need the following estimates from \cite{hele}:
  \begin{equation}
 \label{ampli}
 \norm{\varpi}_{H^{k}}\le \exp C(\norm{\F}^{2}_{L^{\infty}}+\norm{z}^{2}_{H^{k+1}}),
 \end{equation}
 for $k\ge 2$.
 \begin{equation}
 \label{bir}
 \norm{BR(z,\varpi)}_{H^{k}}\le \exp C(\norm{\F}^{2}_{L^{\infty}}+\norm{z}^{2}_{H^{k+1}}),
 \end{equation}
 for $k\ge 2$.
 
 These estimates follows also into the complex strip $S$, since the time derivative plays no role and hence any extra term appears in relation with the terms in \cite{hele}.
 
 \begin{nota}
  Inequalities (\ref{ampli}) and (\ref{bir}) can also be bounded by a polynomial function, see \cite{porous}. In our case, to prove instant analyticity and the decay of the strip, both estimates are valid.
 \end{nota}
  Let $\lambda$\footnote{At the end of the proof of the Theorem \ref{inst}, we can take any $\lambda<\frac{\min_{\alpha}(\sigma(\alpha,0))}{2}$} be given in the definition of $L^{2}(S)$ and $H^{k}(S)$,
  \begin{displaymath}
  \norm{z}^{2}_{L^{2}(S)}(t)=\sum_{\pm}\int_{\T}\abs{z(\alpha\pm i\lambda t,t)-(\alpha\pm i\lambda t,0)}^{2}d\alpha,
  \end{displaymath}
  \begin{displaymath}
  \norm{z}^{2}_{H^{k}(S)}(t)=\norm{z}^{2}_{L^{2}(S)}(t)+\sum_{\pm}\int_{\T}\abs{\partial_{\alpha}^{k}z(\alpha\pm i\lambda t,t)}^{2}d\alpha.
  \end{displaymath}
 \begin{nota}
 Above $\abs{\cdot}$ is the modulus of a vector in $\C^{2}$.
 \end{nota}

\subsection{Estimates for the $H^{4}(S)$ norm}\label{estbr}

We shall analyze the evolution of $\norm{z}_{H^{4}(S)}(t)$.

In order to simplify the exposition we write $z(\alpha,t)=z(\alpha)$ for a fixed $t$, and we denote $\alpha\pm i\lambda t\equiv\gamma$.

It is easy to find that
\begin{equation}
\label{l2}
\frac{1}{2}\frac{d}{dt}\int_{\T}\abs{z(\alpha\pm i\lambda t)-(\alpha\pm i\lambda t,0)}^{2}d\alpha\le\est{k},
\end{equation}
for $k\ge 3$.

Next, we check that
$$\frac{1}{2}\frac{d}{dt}\int_{\T}\abs{\partial^{4}_{\alpha}z(\gamma)}^{2}d\alpha=\sum_{j=1,2}\frac{1}{2}\frac{d}{dt}\int_{\T}\abs{\partial^{4}_{\alpha}z_{j}(\gamma)}^{2}d\alpha$$

where,
$$\frac{1}{2}\frac{d}{dt}\int_{\T}\abs{\partial^{4}_{\alpha}z_{j}(\gamma)}^{2}d\alpha=\re\int_{\T}\overline{\partial^{4}_{\alpha}z_{j}(\gamma)}(\partial_{t}(\partial^{4}_{\alpha}z_{j})(\gamma)\pm i\lambda\partial^{5}_{\alpha}z_{j}(\gamma))d\alpha$$

then,
$$\frac{1}{2}\frac{d}{dt}\int_{\T}\abs{\partial^{4}_{\alpha}z(\gamma)}^{2}d\alpha=\re\int_{\T}\overline{\partial^{4}_{\alpha}z(\gamma)}\cdot\partial^{4}_{\alpha}z_{t}(\gamma)d\alpha\pm\re\int_{\T}\overline{\partial^{4}_{\alpha}z(\gamma)}\cdot i\lambda\partial^{5}_{\alpha}z(\gamma)d\alpha$$$\equiv I_{1}+I_{2}.$

Let us study $I_{2}$:
\begin{align*}
&I_{2}=\re\int_{\T}\overline{\partial^{4}_{\alpha}z(\gamma)}\cdot\partial^{5}_{\alpha}z(\gamma)i\lambda d\alpha=\int_{\T} \lambda(-\re(\de{})\im(\partial^{5}_{\alpha}z)+\im(\partial^{4}_{\alpha}z)\re(\dee{}))d\alpha\\
&=2\lambda\int_{\T}\im(\de{})\re(\dee{})d\alpha=-2\lambda\int_{\T}\im(\de{})\re(\Lambda(H(\de{})))d\alpha\\
&=-2\lambda\int_{\T}\Lambda^{\frac{1}{2}}(\im(\de))\re(\Lambda^{\frac{1}{2}}H(\de))d\alpha\le 2\lambda\norm{\Lambda^{\frac{1}{2}}\im(\de)}_{\s}\norm{\Lambda^{\frac{1}{2}}\de}_{\s}\\
&\le 2\lambda\norm{\Lambda^{\frac{1}{2}}\de}^{2}_{\s},
\end{align*}
where $\Lambda$ is defined by the Fourier transform $\widehat{\Lambda f}(\xi)=\abs{\xi}\widehat{f}(\xi)$ and $H$ is the Hilbert transform:
\begin{align*}
&\Lambda(f)(x)=\frac{1}{2\pi}PV\int\frac{f(x)-f(y)}{\abs{x-y}^{2}}dy,\\
&H(f)(x)=\frac{1}{\pi}PV\int\frac{f(y)}{x-y}dy.
\end{align*} 
Therefore,
$$I_{2}\le 2\lambda\norm{\Lambda^{\frac{1}{2}}\de}^{2}_{\s}.$$

Since we have $z_{t}(\gamma)=BR(z,\varpi)(\gamma)+c(\gamma)\partial_{\alpha}z(\gamma)$, then:
$$I_{1}=\re\int_{\T}\overline{\de{(\gamma)}}\cdot\partial^{4}_{\alpha}BR(z,\varpi)(\gamma)d\alpha+\re\int_{\T}\overline{\de{(\gamma)}}\cdot\partial^{4}_{\alpha}(c(\gamma)\cdot\partial_{\alpha}z(\gamma))d\alpha\equiv J_{1}+J_{2}.$$

We will estimate $J_{1}$ in the subsections $\ref{estbr}$ and $\ref{look}$ and $J_{2}$ in $\ref{cest}.$

\subsubsection{Integrable terms in $\partial^{4}_{\alpha}BR(z,\varpi)$}\label{estbr}
We descompose $J_{1}=I_{3}+I_{4}+I_{5}+I_{6}+I_{7}$, where:
\begin{align*}
&I_{3}=\frac{1}{2\pi}\re\int_{\T}\int_{\R}\overline{\de{(\gamma)}}\cdot\partial^{4}_{\alpha}(\frac{(z(\gamma)-z(\gamma-\beta))^{\bot}}{\abs{z(\gamma)-z(\gamma-\beta)}^{2}})\varpi(\gamma-\beta)d\alpha d\beta,\\
&I_{4}=\frac{2}{\pi}\re\int_{\T}\int_{\R}\overline{\de{(\gamma)}}\cdot\partial^{3}_{\alpha}(\frac{(z(\gamma)-z(\gamma-\beta))^{\bot}}{\abs{z(\gamma)-z(\gamma-\beta)}^{2}})\partial_{\alpha}\varpi(\gamma-\beta)d\alpha d\beta,\\
&I_{5}=\frac{3}{\pi}\re\int_{\T}\int_{\R}\overline{\de{(\gamma)}}\cdot\partial^{2}_{\alpha}(\frac{(z(\gamma)-z(\gamma-\beta))^{\bot}}{\abs{z(\gamma)-z(\gamma-\beta)}^{2}})\partial^{2}_{\alpha}\varpi(\gamma-\beta)d\alpha d\beta,\\
&I_{6}=\frac{2}{\pi}\re\int_{\T}\int_{\R}\overline{\de{(\gamma)}}\cdot\partial_{\alpha}(\frac{(z(\gamma)-z(\gamma-\beta))^{\bot}}{\abs{z(\gamma)-z(\gamma-\beta)}^{2}})\partial^{3}_{\alpha}\varpi(\gamma-\beta)d\alpha d\beta,\\
&I_{7}=\frac{1}{2\pi}\re\int_{\T}\int_{\R}\overline{\de{(\gamma)}}\cdot\frac{(z(\gamma)-z(\gamma-\beta))^{\bot}}{\abs{z(\gamma)-z(\gamma-\beta)}^{2}}\partial^{4}_{\alpha}\varpi(\gamma-\beta)d\alpha d\beta.
\end{align*}

Below we estimate the highest order term of each $I_{1}$. In order to estimate $I_{j}$ for $j=4,5,6$, we refer the reader to the paper \cite{hele}. We get,
\begin{displaymath}
I_{4}+I_{5}+I_{6}\le\est{4}
\end{displaymath}

The most singular terms for $I_{3}$ are those in which four derivatives appear. In order to simplify we write $\Delta\partial^{k}_{\alpha}z\equiv\partial^{k}_{\alpha}z(\gamma)-\partial_{\alpha}^{k}z(\gamma-\beta)$.

 One of the two singular terms of $I_{3}$ is given by
$$K_{1}=\frac{1}{2\pi}\re\int_{\T}\int_{\R}\overline{\de{(\gamma)}}\cdot\frac{(\Delta\de)^{\bot}}{\abs{\Delta z}^{2}}\varpi(\gamma-\beta)d\alpha d\beta,$$

which we decompose in $K_{1}=L_{1}+L_{2}$, where

\begin{align*}
&L_{1}=\frac{1}{2\pi}\re\int_{\T}\int_{\R}\overline{\de{(\gamma)}}\cdot(\Delta\de)^{\bot}\varpi(\gamma-\beta)(\frac{1}{\abs{z(\gamma)-z(\gamma-\beta)}^{2}}-\frac{1}{\abs{\partial_{\alpha}z(\gamma)}^{2}\beta^{2}})d\alpha d\beta,\\
&L_{2}=\frac{1}{2\pi}\re\int_{\T}\int_{\R}\overline{\de{(\gamma)}}\cdot(\Delta\de)^{\bot}\frac{\varpi(\gamma-\beta)}{\abs{\partial_{\alpha}z(\gamma)}^{2}\beta^{2}}d\alpha d\beta.
\end{align*}

Let us study $L_{1}$, if $\psi=\gamma-\beta+s\beta+t\beta-st\beta$, $\phi=\gamma-\beta+s\beta$

and

 \begin{align}
 \label{b}
 &B(\gamma,\beta)\equiv\frac{1}{\abs{z(\gamma)-z(\gamma-\beta)}^{2}}-\frac{1}{\abs{\partial_{\alpha}z(\gamma)}^{2}\beta^{2}}\nonumber\\
 &=\frac{\beta\int _
{0}^{1}\int _
{0}^{1}\partial^{2}_{\alpha}z(\psi)(1-s)dtds\cdot\int _
{0}^{1}[\partial_{\alpha}z(\gamma)+\partial_{\alpha}z(\phi)]ds}{\abs{z(\gamma)-z(\gamma-\beta)}^{2}\abs{\partial_{\alpha}z(\gamma)}^{2}}\\
 &=\frac{\beta\int _
{0}^{1}\int _
{0}^{1}\frac{\partial^{2}_{\alpha}z(\psi)-\partial^{2}_{\alpha}z(\gamma)}{\abs{\psi-\gamma}^{\delta}}\beta^{\delta}(-1+s+t-st)^{\delta}(1-s)dtds\cdot\int _
{0}^{1}[\partial_{\alpha}z(\gamma)+\partial_{\alpha}z(\phi)]ds}{\abs{z(\gamma)-z(\gamma-\beta)}^{2}\abs{\partial_{\alpha}z(\gamma)}^{2}}\nonumber\\
 &+\frac{\beta\partial^{2}_{\alpha}z(\gamma)\cdot\int _
{0}^{1}[\partial_{\alpha}z(\gamma)+\partial_{\alpha}z(\phi)]ds}{\abs{z(\gamma)-z(\gamma-\beta)}^{2}\abs{\partial_{\alpha}z(\gamma)}^{2}}\equiv B_1(\gamma,\beta)+B_2(\gamma,\beta)\nonumber
 \end{align}
 
we have,
\begin{align*}
&L_{1}=\frac{1}{2\pi}\re\int_{\T}\int_{\R}\overline{\de{(\gamma)}}\cdot(\Delta\de)^{\bot}\varpi(\gamma-\beta)B_1(\gamma,\beta)d\alpha d\beta\\
&+\frac{1}{2\pi}\re\int_{\T}\int_{\R}\overline{\de{(\gamma)}}\cdot(\Delta\de)^{\bot}\varpi(\gamma-\beta)B_2(\gamma,\delta)d\alpha d\beta\equiv M_{1}+M_{2}.
\end{align*}

It is easy to check,
$$M_{1}\le C\norm{\F}^{\frac{3}{2}}_{L^{\infty}(S)}\norm{z}_{\mathcal{C}^{2,\delta}(S)}\norm{\varpi}_{L^{\infty}(S)}\norm{\de}^{2}_{\s}.$$

Furthermore,
\begin{align*}
&B_2(\gamma,\beta)=\frac{\beta^{2}\partial^{2}_{\alpha}z(\gamma)\int _
{0}^{1}\int _
{0}^{1}\partial^{2}_{\alpha}z(\eta)(s-1)dtds}{\abs{z(\gamma)-z(\gamma-\beta)}^{2}\abs{\partial_{\alpha}z(\gamma)}^{2}}\\
&+\frac{\beta\partial^{2}_{\alpha}z(\gamma)2\partial_{\alpha}z(\gamma)}{\abs{z(\gamma)-z(\gamma-\beta)}^{2}\abs{\partial_{\alpha}z(\gamma)}^{2}}\equiv B_3(\gamma,\beta)+B_4(\gamma,\beta).
 \end{align*}

 In the same way, we deal with $M_{2}$ and we have:
  \begin{align*}
 &M_{2}=\frac{1}{2\pi}\re\int_{\T}\int_{\R}\overline{\de{(\gamma)}}\cdot(\Delta\de)^{\bot}\varpi(\gamma-\beta)B_3(\gamma,\beta)d\alpha d\beta\\
 &+\frac{1}{2\pi}\re\int_{\T}\int_{\R}\overline{\de{(\gamma)}}\cdot(\Delta\de)^{\bot}\varpi(\gamma-\beta)B_4(\gamma,\beta)d\alpha d\beta\equiv N_{1}+N_{2}.
 \end{align*}
 It is clear that,
 $$N_{1}\le C\norm{\F}^{2}_{L^{\infty}}\norm{z}^{2}_{\mathcal{C}^{2}}\norm{\varpi}_{L^{\infty}}\norm{\de}^{2}_{\s}$$
and
\begin{align*}
 &N_{2}=\frac{1}{2\pi}\re\int_{\T}\int_{\R}\overline{\de{(\gamma)}}\cdot(\Delta\de)^{\bot}\varpi(\gamma-\beta)\frac{\beta\partial^{2}_{\alpha}z(\gamma)2\partial_{\alpha}z(\gamma)}{\abs{\partial_{\alpha}z(\gamma)}^{2}}B(\gamma,\beta)d\alpha d\beta\\
 &+\frac{1}{2\pi}\re\int_{\T}\int_{\R}\overline{\de{(\gamma)}}\cdot(\Delta\de)^{\bot}\varpi(\gamma-\beta)\frac{\partial^{2}_{\alpha}z(\gamma)2\partial_{\alpha}z(\gamma)}{\abs{\partial_{\alpha}z(\gamma)}^{4}\beta}d\alpha d\beta\equiv N_{2}^{1}+N_{2}^{2}.
 \end{align*}
Directly,
\begin{align*}
N_{2}^{1}\le C\norm{\F}^{2}_{L^{\infty}}\norm{z}^{2}_{\mathcal{C}^{2}}\norm{z}^{2}_{\mathcal{C}^{1}}\norm{\varpi}_{L^{\infty}}\norm{\de}^{2}_{\s}
\end{align*}
and we decompose,
\begin{align*}
&N_{2}^{2}=\frac{1}{2\pi}\re\int_{\T}\int_{\R}\overline{\de{(\gamma)}}\cdot(\de{(\gamma)})^{\bot}\varpi(\gamma-\beta)\frac{\partial^{2}_{\alpha}z(\gamma)2\partial_{\alpha}z(\gamma)}{\abs{\partial_{\alpha}z(\gamma)}^{2}\beta}d\alpha d\beta\\
&-\frac{1}{2\pi}\re\int_{\T}\int_{\R}\overline{\de{(\gamma)}}\cdot(\de{(\gamma-\beta)})^{\bot}\varpi(\gamma-\beta)\frac{\partial^{2}_{\alpha}z(\gamma)2\partial_{\alpha}z(\gamma)}{\abs{\partial_{\alpha}z(\gamma)}^{2}\beta}d\alpha d\beta\\
&\equiv N_{2}^{21}+N_{2}^{22}
\end{align*}
where
\begin{align*}
&N_{2}^{21}=\frac{1}{2}\re\int_{\T}\overline{\de{(\gamma)}}\cdot(\de{(\gamma)})^{\bot}\frac{\partial^{2}_{\alpha}z(\gamma)2\partial_{\alpha}z(\gamma)}{\abs{\partial_{\alpha}z(\gamma)}^{2}}H(\varpi)d\alpha\\
&\le C\norm{\F}^{\frac{1}{2}}_{L^{\infty}}\norm{z}_{\mathcal{C}^{2}}\norm{\de}^{2}_{\s}\norm{\varpi}_{\mathcal{C}^{\delta}}
\end{align*}
and
\begin{align*}
&N_{2}^{22}=-\frac{1}{2\pi}\re\int_{\T}\int_{\R}\overline{\de{(\gamma)}}\cdot(\de{(\gamma-\beta)})^{\bot}(\varpi(\gamma-\beta)-\varpi(\gamma))\frac{\partial^{2}_{\alpha}z(\gamma)2\partial_{\alpha}z(\gamma)}{\abs{\partial_{\alpha}z(\gamma)}^{2}\beta}d\alpha d\beta\\
&-\frac{1}{2\pi}\re\int_{\T}\int_{\R}\overline{\de{(\gamma)}}\cdot(\de{(\gamma-\beta)})^{\bot}\varpi(\gamma)\frac{\partial^{2}_{\alpha}z(\gamma)2\partial_{\alpha}z(\gamma)}{\abs{\partial_{\alpha}z(\gamma)}^{2}\beta}d\alpha d\beta\\
&\le C\norm{\F}^{2}_{L^{\infty}}\norm{z}_{\mathcal{C}^{2}}\norm{z}_{\mathcal{C}^{1}}\norm{\varpi}_{\mathcal{C}^{1}}\norm{\de}^{2}_{\s}\\
&-\frac{1}{2}\re\int_{\T}\overline{\de{(\gamma)}}\cdot\varpi(\gamma)\frac{\partial^{2}_{\alpha}z(\gamma)2\partial_{\alpha}z(\gamma)}{\abs{\partial_{\alpha}z(\gamma)}^{2}}H(\de)d\alpha\\
&\le C\norm{\F}^{2}_{L^{\infty}}\norm{z}_{\mathcal{C}^{2}}\norm{z}_{\mathcal{C}^{1}}\norm{\varpi}_{\mathcal{C}^{1}}\norm{\de}^{2}_{\s}
\end{align*}
Here we have used
\begin{align*}
&\norm{H(f)}_{L^{p}}\le C\norm{f}_{L^{p}}\quad\text{for }1<p<\infty,\\
&\norm{H(f)}_{L^{\infty}}\le\norm{f}_{\mathcal{C}^{\delta}}\quad\text{for $f\in\mathcal{C}^{\delta}$, and $0<\delta<1$.}
\end{align*}
Hence, using (\ref{ampli})
$$L_{1}\le\est{4}.$$

For $L_{2}$ we write $L_{2}=M_{3}+M_{4}$, with
\begin{align*}
&M_{3}=\frac{1}{2\pi}\re\int_{\T}\int_{\R}\overline{\de{(\gamma)}}\cdot(\Delta\de)^{\bot}\frac{\varpi(\gamma-\beta)-\varpi(\gamma)}{\abs{\partial_{\alpha}z(\gamma)}^{2}\beta^{2}}d\alpha d\beta,\\
&M_{4}=\frac{1}{2\pi}\re\int_{\T}\int_{\R}\overline{\de{(\gamma)}}\cdot(\Delta\de)^{\bot}\frac{\varpi(\gamma)}{\abs{\partial_{\alpha}z(\gamma)}^{2}\beta^{2}}d\alpha d\beta.
\end{align*}

Next we write
\begin{align*}
M_{3}&=\frac{1}{2\pi}\re\int_{\T}\int_{\R}\overline{\de{(\gamma)}}\cdot\de
(\gamma)\frac{\varpi(\gamma-\beta)-\varpi(\gamma)}{\abs{\partial_{\alpha}z(\gamma)}^{2}\beta^{2}}d\alpha d\beta\\
&-\frac{1}{2\pi}\re\int_{\T}\int_{\R}\overline{\de{(\gamma)}}\cdot\de(\gamma-\beta)\frac{\varpi(\gamma-\beta)-\varpi(\gamma)}{\abs{\partial_{\alpha}z(\gamma)}^{2}\beta^{2}}d\alpha d\beta\equiv N_{3}+N_{4}
\end{align*}
where 
\begin{align*}
N_{3}&=\re\int_{\T}\overline{\de{(\gamma)}}\cdot\de
(\gamma)\frac{\Lambda(\varpi)(\gamma)}{\abs{\partial_{\alpha}z(\gamma)}^{2}}d\alpha\\
&\le C\norm{\F}_{L^{\infty}(S)}\norm{\de}^{2}_{\s}\norm{\Lambda\varpi}_{L^{\infty}(S)}\le C\norm{\F}_{L^{\infty}(S)}\norm{\de}^{2}_{\s}\norm{\varpi}_{\mathcal{C}^{1,\delta}(S)}
\end{align*}
and
\begin{align*}
N_{4}&=-\frac{1}{2\pi}\re\int_{\T}\int_{\R}\overline{\de{(\gamma)}}\cdot\de(\gamma-\beta)\frac{\int _
{0}^{1}[\partial_{\alpha}\varpi(\gamma-s\beta)-\partial_{\alpha}\varpi(\gamma)]ds}{\abs{\partial_{\alpha}z(\gamma)}^{2}\beta}d\alpha d\beta\\
&-\frac{1}{2\pi}\re\int_{\T}\int_{\R}\overline{\de{(\gamma)}}\cdot\de(\gamma-\beta)\frac{\partial_{\alpha}\varpi(\gamma)}{\abs{\partial_{\alpha}z(\gamma)}^{2}\beta}d\alpha d\beta\\
&\le C\norm{\F}_{L^{\infty}(S)}\norm{\de}^{2}_{\s}\norm{\varpi}_{\mathcal{C}^{2}(S)}-\frac{1}{2}\re\int_{\T}\overline{\de{(\gamma)}}\cdot H(\de)(\gamma)\frac{\partial_{\alpha}\varpi(\gamma)}{\abs{\partial_{\alpha}z(\gamma)}^{2}}d\alpha\\
&\le\est{4}.
\end{align*}

For $M_{4}$,
\begin{align*}
M_{4}&=\re\int_{\T}\overline{\de{(\gamma)}}\cdot\Lambda(\de^{\bot})(\gamma)\frac{\varpi(\gamma)}{\abs{\partial_{\alpha}z(\gamma)}^{2}}d\alpha\\
&=\int_{\T}\im(\frac{\varpi}{A(t)})(-\re(\de)\cdot\im(\Lambda(\de^{\bot}))+\im(\de)\cdot\re(\Lambda(\de^{\bot})))d\alpha\\
&+\int_{\T}\re(\frac{\varpi}{A(t)})(\re(\de)\cdot\re(\Lambda(\de^{\bot}))+\im(\de)\cdot\im(\Lambda(\de^{\bot})))d\alpha\\
&\equiv N_{5}+N_{6}.
\end{align*}
 Now we take,
 \begin{align*}
 N_{6}&=\int_{\T}\re(\frac{\varpi}{A(t)})(-\re(\de_{1})\re(\Lambda(\de_{2}))+\re(\de_{2})\re(\Lambda(\de_{1})))d\alpha\\
 &+\int_{\T}\re(\frac{\varpi}{A(t)})(-\im(\de_{1})\im(\Lambda(\de_{2}))+\im(\de_{2})\im(\Lambda(\de_{1})))d\alpha\\
 &\equiv N_{6}^{1}+N_{6}^{2}
 \end{align*}
  where it is easy to find a commutator formula such that, using (see \cite{commutator})
 \begin{equation}
 \label{conmutador}
 \norm{\Lambda(fg)-g\Lambda(f)}_{L^{2}}\le\norm{g}_{\mathcal{C}^{1,\delta}}\norm{f}_{L^{2}},
 \end{equation}
 we get
 \begin{align*}
 N_{6}^{1}&=\int_{\T}(-\Lambda(\re(\frac{\varpi}{A(t)})\re(\de_{1}))+\re(\frac{\varpi}{A(t)})\re(\Lambda(\de_{1})))\re(\de_{2})d\alpha\\
 &\le C\norm{\re(\frac{\varpi}{A(t)})}_{\mathcal{C}^{1,\delta}}\norm{\de}^{2}_{\s}.
 \end{align*}

In the same way,

 \begin{align*}
 N_{6}^{2}\le C\norm{\re(\frac{\varpi}{A(t)})}_{\mathcal{C}^{1,\delta}(S)}\norm{\de}^{2}_{\s}.
 \end{align*}
 
 Thus,
 $$N_{6}\le\est{4}.$$

For $N_{5}$ we have,
\begin{align*}
N_{5}&=\int_{\T}\im(\frac{\varpi}{A(t)})(\re(\de^{\bot})\cdot\im(\Lambda(\de))+\im(\de)\cdot\re(\Lambda(\de^{\bot})))d\alpha\\
&=\int_{\T}(\Lambda(\im(\frac{\varpi}{A(t)})\re(\de^{\bot}))-\im(\frac{\varpi}{A(t)})\re(\Lambda(\de^{\bot})))\cdot\im(\de)d\alpha\\
&+2\int_{\T}\im(\frac{\varpi}{A(t)})\re(\Lambda(\de^{\bot}))\cdot\im(\de)d\alpha\\
&\equiv N_{5}^{1}+N_{5}^{2}.
\end{align*}
Then,
\begin{align*}
N_{5}^{1}\le C\norm{\im(\frac{\varpi}{A(t)})}_{\mathcal{C}^{1,\delta}(S)}\norm{\de}^{2}_{\s}\le\est{4}
\end{align*}
and
\begin{align*}
N_{5}^{2}&=2\int_{\T}\Lambda^{\frac{1}{2}}(\im(\frac{\varpi}{A(t)})\im(\de))\cdot\re(\Lambda^{\frac{1}{2}}(\de^{\bot}))d\alpha\le 2\norm{\Lambda^{\frac{1}{2}}(\im(\frac{\varpi}{A(t)})\im(\de))}_{\s}\norm{\Lambda^{\frac{1}{2}}\de{}}_{\s}\\
&\le C\norm{\im(\frac{\varpi}{A(t)})}_{H^{2}(S)}(\norm{\de}_{\s}+\norm{\Lambda^{\frac{1}{2}}\de{}}_{\s})\norm{\Lambda^{\frac{1}{2}}\de{}}_{\s}\\
&\le C\norm{\im(\frac{\varpi}{A(t)})}_{H^{2}(S)}(\frac{\norm{\de}_{\s}^{2}}{2}+\frac{\norm{\Lambda^{\frac{1}{2}}\de{}}_{\s}^{2}}{2})+C\norm{\im(\frac{\varpi}{A(t)})}_{H^{2}(S)}\norm{\Lambda^{\frac{1}{2}}\de{}}^{2}_{\s}\\
&\le\est{4}+C\norm{\im(\frac{\varpi}{A(t)})}_{H^{2}(S)}\norm{\Lambda^{\frac{1}{2}}\de{}}^{2}_{\s}.
\end{align*}
Concluding,
$$K_{1}\le\est{4}+C\norm{\im(\frac{\varpi}{A(t)})}_{H^{2}(S)}\norm{\Lambda^{\frac{1}{2}}\de{}}^{2}_{\s}.$$

The other singular term with four derivatives inside $I_{3}$ is given by
$$K_{2}=-\frac{1}{\pi}\re\int_{\T}\int_{\R}\overline{\de{(\gamma)}}\cdot\frac{(\Delta z)^{\bot}}{\abs{\Delta z}^{4}}(\Delta z\cdot\Delta\de{})\varpi(\gamma-\beta) d\alpha d\beta.$$

Here we take $K_{2}=L_{3}+L_{4}+L_{5}$ where
\begin{align*}
&L_{3}=-\frac{1}{\pi}\re\int_{\T}\int_{\R}\overline{\de{(\gamma)}}\cdot(\frac{(\Delta z)^{\bot}}{\abs{\Delta z}^{4}}-\frac{\partial_{\alpha}z(\gamma)^{\bot}}{\abs{\partial_{\alpha}z(\gamma)}^{4}\beta^{3}})(\Delta z-\beta\partial_{\alpha}z(\gamma))\cdot\varpi(\gamma-\beta)\Delta\de{} d\alpha d\beta,\\
&L_{4}=-\frac{1}{\pi}\re\int_{\T}\int_{\R}\overline{\de{(\gamma)}}\cdot(\frac{(\Delta z)^{\bot}}{\abs{\Delta z}^{4}}-\frac{\partial_{\alpha}z(\gamma)^{\bot}}{\abs{\partial_{\alpha}z(\gamma)}^{4}\beta^{3}})(\beta\partial_{\alpha}z(\gamma)\cdot\Delta\de{})\varpi(\gamma-\beta) d\alpha d\beta,\\
&L_{5}=-\frac{1}{\pi}\re\int_{\T}\int_{\R}\overline{\de{(\gamma)}}\cdot\frac{\partial_{\alpha}z(\gamma)^{\bot}}{\abs{\partial_{\alpha}z(\gamma)}^{4}\beta^{3}}(\Delta z\cdot\Delta\de{})\varpi(\gamma-\beta) d\alpha d\beta.
\end{align*}
We compute
 \begin{align}
 \label{c(gamma)}
 &C(\gamma,\beta)=\frac{(\Delta z)^{\bot}}{\abs{\Delta z}^{4}}-\frac{\partial_{\alpha}z(\gamma)^{\bot}}{\abs{\partial_{\alpha}z(\gamma)}^{4}\beta^{3}}=\frac{\beta^{2}\int _
{0}^{1}\int _
{0}^{1}\partial^{2}_{\alpha}z^{\bot}(\eta)(s-1)dsdt}{\abs{\Delta z}^{4}}\nonumber\\
 &+\frac{\beta^{2}\partial_{\alpha}z^{\bot}(\gamma)\int _
{0}^{1}\int _
{0}^{1}\partial^{2}_{\alpha}z(\psi)(1-s)dsdt\cdot\int _
{0}^{1}[\partial_{\alpha}z(\gamma)+\partial_{\alpha}z(\phi)]ds\int _
{0}^{1}[\abs{\partial_{\alpha}z(\gamma)}^{2}+\abs{\partial_{\alpha}z(\phi)}^{2}]ds}{\abs{\Delta z}^{4}\abs{\partial_{\alpha}z(\gamma)}^{4}}\\
 &\equiv C_1(\gamma,\beta)+C_2(\gamma,\beta)\nonumber
 \end{align}
 and
 \begin{align*}
 \Delta z-\beta\partial_{\alpha}z(\gamma)=\beta^{2}\int _
{0}^{1}\int _
{0}^{1}\partial^{2}_{\alpha}z(\eta)(s-1)dtds
 \end{align*}
 where $\eta=\gamma-t\beta+st\beta$, allowing us to obtain the desired estimate for the term $L_{3}$.
 
 Next we split $L_{4}=M_{5}+M_{6}$ since $\Delta\de{}=\de(\gamma)-\de(\gamma-\beta)$:
 \begin{align*}
 &M_{5}=-\frac{1}{\pi}\re\int_{\T}\int_{\R}\overline{\de{(\gamma)}}\cdot C(\gamma,\beta)(\beta\partial_{\alpha} z(\gamma)\cdot\de{(\gamma)})\varpi(\gamma-\beta) d\alpha d\beta,\\
 &M_{6}=\frac{1}{\pi}\re\int_{\T}\int_{\R}\overline{\de{(\gamma)}}\cdot C(\gamma,\beta)(\beta\partial_{\alpha} z(\gamma)\cdot\de{(\gamma-\beta)})\varpi(\gamma-\beta) d\alpha d\beta.\\
 \end{align*}
 
 By following the same approach for $L_{1}$ we have,
 
\begin{align*}
\abs{M_{5}}&\le C\norm{\F}_{L^{\infty}(S)}^{2}\norm{z}_{\mathcal{C}^{2}(S)}^{2}\norm{\varpi}_{L^{\infty}(S)}\norm{\de}^{2}_{\s}\\
&+\abs{\re\int_{\T}\overline{\de{(\gamma)}}\cdot \partial^{2}_{\alpha}z^{\bot}(\gamma)(\partial_{\alpha}z(\gamma)\cdot\de{(\gamma)})\frac{H(\varpi)(\gamma)}{\abs{\partial_{\alpha}z(\gamma)}^{4}} d\alpha}
\end{align*}
and
\begin{align*}
\abs{M_{6}}&\le C\norm{\F}_{L^{\infty}(S)}^{2}\norm{z}_{\mathcal{C}^{2}(S)}^{2}\norm{\varpi}_{L^{\infty}(S)}\norm{\de}^{2}_{\s}\\
&+\abs{\re\int_{\T}\overline{\de{(\gamma)}}\cdot \partial^{2}_{\alpha}z^{\bot}(\gamma)(\partial_{\alpha}z(\gamma)\cdot H(\de{})(\gamma))\frac{\varpi(\gamma)}{\abs{\partial_{\alpha}z(\gamma)}^{4}} d\alpha}.
\end{align*}
Then, the term $L_{4}$ is controlled.

To conclude the estimates of $K_{2}$, we need to see what happens with the term $L_{5}$

\begin{align*}
L_{5}&=-\frac{1}{\pi}\re\int_{\T}\int_{\R}\overline{\de{(\gamma)}}\cdot\frac{\partial^{\bot}_{\alpha}z(\gamma)}{\abs{\partial_{\alpha}z(\gamma)}^{4}}(\int _
{0}^{1}[\partial_{\alpha}z(\phi)-\partial_{\alpha}z(\gamma)]ds\cdot\Delta\de{})\frac{\varpi(\gamma-\beta)}{\beta^{2}} d\alpha d\beta\\
&-\frac{1}{\pi}\re\int_{\T}\int_{\R}\overline{\de{(\gamma)}}\cdot\frac{\partial^{\bot}_{\alpha}z(\gamma)}{\abs{\partial_{\alpha}z(\gamma)}^{4}}(\partial_{\alpha}z(\gamma)\cdot\Delta\de{})\frac{\varpi(\gamma-\beta)}{\beta^{2}} d\alpha d\beta\\
&\equiv M_{7}+M_{8}.
\end{align*}
For $M_{7}$ we proceed in the same way as in $L_{4}$ and we get:
\begin{align*}
M_{7}&\le\est{4}\\
&-\re\int_{\T}\overline{\de{(\gamma)}}\cdot\frac{\partial^{\bot}_{\alpha}z(\gamma)}{\abs{\partial_{\alpha}z(\gamma)}^{4}}(\partial^{2}_{\alpha}z(\gamma)\cdot\de{(\gamma)})H(\varpi)(\gamma) d\alpha\\
&+\re\int_{\T}\overline{\de{(\gamma)}}\cdot\frac{\partial^{\bot}_{\alpha}z(\gamma)}{\abs{\partial_{\alpha}z(\gamma)}^{4}}(\partial^{2}_{\alpha}z(\gamma)\cdot H(\de)(\gamma))\varpi(\gamma) d\alpha.
\end{align*}

Then, 
$$M_{7}\le\est{4}.$$

To control the term $M_{8}$, we decompose it as follows,
\begin{align*}
M_{8}&=-\frac{1}{\pi}\re\int_{\T}\int_{\R}\overline{\de{(\gamma)}}\cdot\frac{\partial^{\bot}_{\alpha}z(\gamma)}{\abs{\partial_{\alpha}z(\gamma)}^{4}}(\partial_{\alpha}z(\gamma)\cdot\Delta\de{})\frac{\varpi(\gamma-\beta)-\varpi(\gamma)}{\beta^{2}} d\alpha d\beta\\
&-\frac{1}{\pi}\re\int_{\T}\int_{\R}\overline{\de{(\gamma)}}\cdot\frac{\partial^{\bot}_{\alpha}z(\gamma)}{\abs{\partial_{\alpha}z(\gamma)}^{4}}(\partial_{\alpha}z(\gamma)\cdot\Delta\de{})\frac{\varpi(\gamma)}{\beta^{2}} d\alpha d\beta\\
&\equiv N_{7}+N_{8}.
\end{align*}

Since $\Delta\de{}=\de(\gamma)-\de(\gamma-\beta)$ we have,
\begin{align*}
N_{7}&=-\frac{1}{\pi}\re\int_{\T}\int_{\R}\overline{\de{(\gamma)}}\cdot\frac{\partial^{\bot}_{\alpha}z(\gamma)}{\abs{\partial_{\alpha}z(\gamma)}^{4}}(\partial_{\alpha}z(\gamma)\cdot\de{(\gamma)})\frac{\varpi(\gamma-\beta)-\varpi(\gamma)}{\beta^{2}} d\alpha d\beta\\
&+\frac{1}{\pi}\re\int_{\T}\int_{\R}\overline{\de{(\gamma)}}\cdot\frac{\partial^{\bot}_{\alpha}z(\gamma)}{\abs{\partial_{\alpha}z(\gamma)}^{4}}(\partial_{\alpha}z(\gamma)\cdot\de{(\gamma-\beta)})\frac{\varpi(\gamma-\beta)-\varpi(\gamma)}{\beta^{2}} d\alpha d\beta\\
&\equiv O_{1}+O_{2}
\end{align*}
where,
\begin{align*}
O_{1}&=-2\re\int_{\T}\overline{\de{(\gamma)}}\cdot\frac{\partial^{\bot}_{\alpha}z(\gamma)}{\abs{\partial_{\alpha}z(\gamma)}^{4}}(\partial_{\alpha}z(\gamma)\cdot\de{(\gamma)})\Lambda(\varpi)(\gamma) d\alpha d\beta\\
&\le C\norm{\F}_{L^{\infty}(S)}\norm{\de}^{2}_{\s}\norm{\Lambda\varpi}_{L^{\infty}(S)}\le C\norm{\F}_{L^{\infty}(S)}\norm{\de}^{2}_{\s}\norm{\varpi}_{\mathcal{C}^{1,\delta}(S)}\\
&\le\est{4}
\end{align*}
and
\begin{align*}
O_{2}&\le\est{4}\\
&+\re\int_{\T}\overline{\de{(\gamma)}}\cdot\frac{\partial^{\bot}_{\alpha}z(\gamma)}{\abs{\partial_{\alpha}z(\gamma)}^{4}}(\partial_{\alpha}z(\gamma)\cdot H(\de)(\gamma))\partial_{\alpha}\varpi(\gamma) d\alpha\\
&\le\est{4}.
\end{align*}

Using integration by parts for $\Lambda$,
\begin{align*}
N_{8}&=-2\re\int_{\T}\Lambda(\overline{\de{}}\cdot\frac{\partial_{\alpha}z^{\bot}}{A^{2}(t)}\varpi\partial_{\alpha}z)(\gamma)\cdot\de{}(\gamma)d\alpha\\
&=-2\re\int_{\T}(\Lambda(\overline{\de{}}\cdot\frac{\partial_{\alpha}z^{\bot}}{A^{2}(t)}\varpi\partial_{\alpha}z)(\gamma)-\partial_{\alpha}z(\gamma)\varpi(\gamma)\frac{\partial^{\bot}_{\alpha}z(\gamma)}{A^{2}(t)}\cdot\Lambda(\overline{\de{}})(\gamma))\cdot\de(\gamma)d\alpha\\
&-2\re\int_{\T}\frac{\partial^{\bot}_{\alpha}z(\gamma)}{A^{2}(t)}\cdot\Lambda(\overline{\de{}})(\gamma)\partial_{\alpha}z(\gamma)\cdot\de(\gamma)\varpi(\gamma)d\alpha\\
&\equiv O_{3}+O_{4}.
\end{align*}

Using the commutator estimate (\ref{conmutador}),
\begin{align*}
O_{3}\le\est{4}.
\end{align*}
Taking three derivatives of $A(t)=\abs{\partial_{\alpha}z}^{2}$ we take
$$\partial_{\alpha}z(\alpha)\cdot\partial^{4}_{\alpha}z(\alpha)=-3\partial^{2}_{\alpha}z(\alpha)\cdot\partial^{3}_{\alpha}z(\alpha).$$
Together with $\Lambda=\partial_{\alpha}H$ and integrating by parts
\begin{align*}
O_{4}&=-6\re\int_{\T} H(\overline{\de{}})(\gamma)\cdot\frac{\partial^{2}_{\alpha}z^{\bot}(\gamma)}{A^{2}(t)}\partial^{2}_{\alpha}z(\gamma)\cdot\partial^{3}_{\alpha}z(\gamma)\varpi(\gamma)d\alpha\\
&-6\re\int_{\T} H(\overline{\de{}})(\gamma)\cdot\frac{\partial^{\bot}_{\alpha}z(\gamma)}{A^{2}(t)}\partial^{3}_{\alpha}z(\gamma)\cdot\partial^{3}_{\alpha}z(\gamma)\varpi(\gamma)d\alpha\\
&-6\re\int_{\T} H(\overline{\de{}})(\gamma)\cdot\frac{\partial^{\bot}_{\alpha}z(\gamma)}{A^{2}(t)}\partial^{2}_{\alpha}z(\gamma)\cdot\partial^{4}_{\alpha}z(\gamma)\varpi(\gamma)d\alpha\\
&-6\re\int_{\T} H(\overline{\de{}})(\gamma)\cdot\frac{\partial^{\bot}_{\alpha}z(\gamma)}{A^{2}(t)}\partial^{2}_{\alpha}z(\gamma)\cdot\partial^{3}_{\alpha}z(\gamma)\partial_{\alpha}\varpi(\gamma)d\alpha\\
&\le\est{4}.
\end{align*}

Then,
\begin{displaymath}
L_{5}\le\est{4}.
\end{displaymath}

All previous discussion shows that $I_{3}$ satisfies,
$$I_{3}\le\est{4}+C\norm{\im(\frac{\varpi}{A(t)})}_{H^{2}(S)}\norm{\Lambda^{\frac{1}{2}}(\de)}^{2}_{\s}.$$

\subsubsection{Searching for the Raylegh-Taylor condition in $I_{7}$}\label{look}

Let us recall the formula for the Raylegh-Taylor condition
$$\sigma(\alpha,t)=\frac{\mu^{2}}{\kappa}BR(z,\varpi)(\alpha,t)\cdot\partial^{\bot}_{\alpha}z(\alpha,t)+g\rho^{2}\partial_{\alpha}\z(\alpha,t).$$
We write $I_{7}$ in the form $I_{7}=K_{8}+K_{9}$ where
\begin{align*}
&K_{8}=\frac{1}{2\pi}\re\int_{\T}\int_{\R}\overline{\de(\gamma)}\cdot(\frac{(z(\gamma)-z(\gamma-\beta))^{\bot}}{\abs{z(\gamma)-z(\gamma-\beta)}^{2}}-\frac{\partial^{\bot}_{\alpha}z(\gamma)}{\abs{\partial_{\alpha}z(\gamma)}^{2}\beta})\partial^{4}_{\alpha}\varpi(\gamma-\beta)d\alpha d\beta,\\
&K_{9}=\frac{1}{2\pi}\re\int_{\T}\int_{\R}\overline{\de(\gamma)}\cdot\frac{\partial^{\bot}_{\alpha}z(\gamma)}{\abs{\partial_{\alpha}z(\gamma)}^{2}\beta}\partial^{4}_{\alpha}\varpi(\gamma-\beta)d\alpha d\beta.
\end{align*}

After an integration by parts we obtain:
\begin{align*}
K_{8}&=-\frac{1}{2\pi}\re\int_{\T}\int_{\R}\overline{\de(\gamma)}\cdot(\frac{(\Delta z)^{\bot}}{\abs{\Delta z}^{2}}-\frac{\partial^{\bot}_{\alpha}z(\gamma)}{\abs{\partial_{\alpha}z(\gamma)}^{2}\beta})\partial_{\beta}(\partial^{3}_{\alpha}\varpi(\gamma-\beta))d\alpha d\beta\\
&=\frac{1}{2\pi}\re\int_{\T}\int_{\R}\overline{\de(\gamma)}\cdot\partial_{\beta}(\frac{(\Delta z)^{\bot}}{\abs{\Delta z}^{2}}-\frac{\partial^{\bot}_{\alpha}z(\gamma)}{\abs{\partial_{\alpha}z(\gamma)}^{2}\beta})\partial^{3}_{\alpha}\varpi(\gamma-\beta)d\alpha d\beta.
\end{align*}

We decompose
\begin{align}
&\partial_{\beta}(\frac{(\Delta z)^{\bot}}{\abs{\Delta z}^{2}}-\frac{\partial^{\bot}_{\alpha}z(\gamma)}{\abs{\partial_{\alpha}z(\gamma)}^{2}\beta})\nonumber\\
&=\frac{(\Delta\partial_{\alpha}z)^{\bot}}{\abs{\Delta z}^{2}}+\partial^{\bot}_{\alpha}z(\gamma)(\frac{1}{\abs{\Delta z}^{2}}-\frac{1}{\abs{\partial_{\alpha}z(\gamma)}^{2}\beta^{2}})-2\frac{(\Delta z)^{\bot}\Delta z\cdot\Delta\partial_{\alpha}z}{\abs{\Delta z}^{4}}\nonumber\\
&-2\frac{(\Delta z)^{\bot}(\Delta z-\beta\partial_{\alpha}z(\gamma))\cdot\partial_{\alpha}z(\gamma)}{\abs{\Delta z}^{4}}-2\frac{(\Delta z-\beta\partial_{\alpha}z(\gamma))^{\bot}\beta\abs{\partial_{\alpha}z(\gamma)}^{2}}{\abs{\Delta z}^{4}}\nonumber\\
&+(\frac{2\partial^{\bot}_{\alpha}z(\gamma)}{\abs{\partial_{\alpha}z(\gamma)}^{2}\beta^{2}}-\frac{2\beta^{2}\partial^{\bot}_{\alpha}z(\gamma)\abs{\partial_{\alpha}z(\gamma)}^{2}}{\abs{\Delta z}^{4}})\nonumber\\
&\equiv F_{1}(\gamma,\beta)+F_{2}(\gamma,\beta)+F_{3}(\gamma,\beta)+F_{4}(\gamma,\beta)+F_{5}(\gamma,\beta)+F_{6}(\gamma,\beta)\label{descomposicion}
\end{align}

Then, we have
\begin{align*}
K_{8}&=\frac{1}{2\pi}\re\int_{\T}\int_{\R}\overline{\de(\gamma)}\cdot F_{1}(\gamma,\beta)\partial^{3}_{\alpha}\varpi(\gamma-\beta)d\alpha d\beta\\
&+\frac{1}{2\pi}\re\int_{\T}\int_{\R}\overline{\de(\gamma)}\cdot F_{2}(\gamma,\beta)\partial^{3}_{\alpha}\varpi(\gamma-\beta)d\alpha d\beta\\
&+\frac{1}{2\pi}\re\int_{\T}\int_{\R}\overline{\de(\gamma)}\cdot F_{3}(\gamma,\beta)\partial^{3}_{\alpha}\varpi(\gamma-\beta)d\alpha d\beta\\
&+\frac{1}{2\pi}\re\int_{\T}\int_{\R}\overline{\de(\gamma)}\cdot F_{4}(\gamma,\beta)\partial^{3}_{\alpha}\varpi(\gamma-\beta)d\alpha d\beta\\
&+\frac{1}{2\pi}\re\int_{\T}\int_{\R}\overline{\de(\gamma)}\cdot F_{5}(\gamma,\beta)\partial^{3}_{\alpha}\varpi(\gamma-\beta)d\alpha d\beta\\
&+\frac{1}{2\pi}\re\int_{\T}\int_{\R}\overline{\de(\gamma)}\cdot F_{6}(\gamma,\beta)\partial^{3}_{\alpha}\varpi(\gamma-\beta)d\alpha d\beta\\
&\equiv P_{1}+P_{2}+P_{3}+P_{4}+P_{5}+P_{6}.
\end{align*}

For $P_{1}, P_{2}, P_{3}, P_{4}$ and $P_{5}$ we can estimates with the same approach as before, and we easily get
\begin{displaymath}
P_{1}+P_{2}+P_{3}+P_{4}+P_{5}\le\est{4}.
\end{displaymath} 

 
 Since,
 \begin{align*}
&-\frac{1}{2}F_{6}(\gamma,\beta)=\partial^{\bot}_{\alpha}z(\gamma)\frac{\beta^{4}\abs{\partial_{\alpha}z(\gamma)}^{4}-\abs{\Delta z}^{4}}{\abs{\Delta z}^{4}\abs{\partial_{\alpha}z(\gamma)}^{2}\beta^{2}}=U_{1}(\gamma,\beta)\\
 &+\frac{\partial^{\bot}_{\alpha}z(\gamma)}{2}\frac{\beta^{4}\partial^{2}_{\alpha}z(\gamma)\int _
{0}^{1}\int _
{0}^{1}\partial^{2}_{\alpha}z(\eta)(s-1)dtds\int _
{0}^{1}[\abs{\partial_{\alpha}z(\gamma)}^{2}+\abs{\partial_{\alpha}z(\phi)}^{2}]ds}{\abs{\Delta z}^{4}\abs{\partial_{\alpha}z(\gamma)}^{2}}\\
 &+\partial^{\bot}_{\alpha}z(\gamma)\frac{\beta^{4}\partial^{2}_{\alpha}z(\gamma)\partial_{\alpha}z(\gamma)\int _
{0}^{1}\int _
{0}^{1}\partial_{\alpha}z(\eta)\cdot\partial^{2}_{\alpha}z(\eta)(s-1)dtds}{\abs{\Delta z}^{4}\abs{\partial_{\alpha}z(\gamma)}^{2}}+\partial^{\bot}_{\alpha}z(\gamma)\frac{\beta^{3}\partial^{2}_{\alpha}z(\gamma)\partial_{\alpha}z(\gamma)}{\abs{\Delta z}^{4}}\\
 &\equiv U_{1}(\gamma,\beta)+U_{2}(\gamma,\beta)+U_{3}(\gamma,\beta)+U_{4}(\gamma,\beta)
 \end{align*}
 where $U_{1}(\gamma,\beta)$ is the remainder term that does not cause any trouble.
 we get,
 \begin{align*}
 P_{6}=&-\frac{1}{\pi}\re\int_{\T}\int_{\R}\overline{\de(\gamma)}\cdot U_{1}(\gamma,\beta)\partial^{3}_{\alpha}\varpi(\gamma-\beta)d\alpha d\beta\\
 &-\frac{1}{\pi}\re\int_{\T}\int_{\R}\overline{\de(\gamma)}\cdot U_{2}(\gamma,\beta)\partial^{3}_{\alpha}\varpi(\gamma-\beta)d\alpha d\beta\\
 &-\frac{1}{\pi}\re\int_{\T}\int_{\R}\overline{\de(\gamma)}\cdot U_{3}(\gamma,\beta)\partial^{3}_{\alpha}\varpi(\gamma-\beta)d\alpha d\beta\\
 &-\frac{1}{\pi}\re\int_{\T}\int_{\R}\overline{\de(\gamma)}\cdot U_{4}(\gamma,\beta)\partial^{3}_{\alpha}\varpi(\gamma-\beta)d\alpha d\beta\\
 &\equiv Q_{1}+Q_{2}+Q_{3}+Q_{4}
 \end{align*}
 where
 \begin{align*}
 &Q_{1}\le C\norm{\F}^{2}_{L^{\infty}(S)}\norm{z}^{2}_{\mathcal{C}^{1}(S)}\norm{z}_{\mathcal{C}^{3}(S)}\norm{\de}_ {\s}\norm{\partial^{3}_{\alpha}\varpi}_{\s},\\
 &Q_{2}\le C\norm{\F}^{2}_{L^{\infty}(S)}\norm{z}^{2}_{\mathcal{C}^{2}(S)}\norm{z}_{\mathcal{C}^{1}(S)}\norm{\de}_ {\s}\norm{\partial^{3}_{\alpha}\varpi}_{\s},\\
 &Q_{3}\le C\norm{\F}^{2}_{L^{\infty}(S)}\norm{z}_{\mathcal{C}^{1}(S)}\norm{z}_{\mathcal{C}^{2}(S)}\norm{\de}_ {\s}\norm{\partial^{3}_{\alpha}\varpi}_{\s}
 \end{align*}
 
 and if we split,
 \begin{align*}
 Q_{4}&=-\frac{1}{\pi}\re\int_{\T}\int_{\R}\overline{\de(\gamma)}\cdot\partial^{\bot}_{\alpha}z(\gamma)\beta^{3}\partial^{2}_{\alpha}z(\gamma)\partial_{\alpha}z(\gamma)\partial^{3}_{\alpha}\varpi(\gamma-\beta)(\frac{1}{\abs{\Delta z}^{4}}-\frac{1}{\abs{\partial_{\alpha}z(\gamma)}^{4}\beta^{4}})d\alpha d\beta\\
 &-\frac{1}{\pi}\re\int_{\T}\overline{\de(\gamma)}\cdot\frac{\partial^{\bot}_{\alpha}z(\gamma)\partial^{2}_{\alpha}z(\gamma)\partial_{\alpha}z(\gamma)}{\abs{\partial_{\alpha}z(\gamma)}^{4}}H(\partial^{3}_{\alpha}\varpi)(\gamma)d\alpha.
 \end{align*}
 
 It is clear that
 $$Q_{4}\le\est{4}.$$
 
Thus,
$$P_{6}\le\est{4}.$$

Therefore,
$$K_{8}\le\est{4}.$$

We consider now the $K_{9}$ term which can be written as follows 
\begin{align*}
K_{9}&=\frac{1}{2}\re\int_{\T}\overline{\de(\gamma)}\cdot\frac{\partial^{\bot}_{\alpha}z(\gamma)}{\abs{\partial_{\alpha}z(\gamma)}^{2}}H(\partial^{4}_{\alpha}\varpi)(\gamma)d\alpha=\frac{1}{2}\re\int_{\T}\overline{\de(\gamma)}\cdot\frac{\partial^{\bot}_{\alpha}z(\gamma)}{\abs{\partial_{\alpha}z(\gamma)}^{2}}\Lambda(\partial^{3}_{\alpha}\varpi)(\gamma)d\alpha\\
&=\frac{1}{2}\re\int_{\T}\frac{\Lambda(\overline{\de{}}\cdot\partial_{\alpha}z^{\bot})(\gamma)}{A(t)}\partial^{3}_{\alpha}\varpi(\gamma)d\alpha
\end{align*}

using the formula $$\varpi(\alpha)=-2BR(z,\varpi)(\alpha,t)\cdot\partial_{\alpha}z(\alpha,t)-2\kappa g\frac{\rho^{2}}{\mu^{2}}\partial_{\alpha}\zz(\alpha,t)=-T(\varpi)(\alpha)-2g\kappa\frac{\rho^{2}}{\mu^{2}}\partial_{\alpha}z_{2}(\alpha)$$
we separate $K_{9}$ as a sum of two parts, $P_{7}$ and $P_{8}$, where
\begin{align*}
&P_{7}=-g\kappa\frac{\rho^{2}}{\mu^{2}}\re\int_{\T}\frac{\Lambda(\overline{\de{}}\cdot\partial_{\alpha}z^{\bot})(\gamma)}{A(t)}\partial^{4}_{\alpha}\zz(\gamma)d\alpha,\\
&P_{8}=-\frac{1}{2}\re\int_{\T}\frac{\Lambda(\overline{\de{}}\cdot\partial_{\alpha}z^{\bot})(\gamma)}{A(t)}\partial^{3}_{\alpha}T(\varpi)(\gamma)d\alpha.
\end{align*}

For $P_{7}$ we decompose further $P_{7}=Q_{5}+Q_{6}$ where

\begin{align*}
&Q_{5}=g\kappa\frac{\rho^{2}}{\mu^{2}}\re\int_{\T}\frac{\Lambda(\overline{\de_{1}}\partial_{\alpha}\zz)(\gamma)}{A(t)}\partial^{4}_{\alpha}\zz(\gamma)d\alpha,\\
&Q_{6}=-g\kappa\frac{\rho^{2}}{\mu^{2}}\re\int_{\T}\frac{\Lambda(\overline{\de_{2}}\partial_{\alpha}\z)(\gamma)}{A(t)}\partial^{4}_{\alpha}\zz(\gamma)d\alpha.
\end{align*}

Then $Q_{5}$ is written as $Q_{5}=R_{1}+R_{2}$ with
\begin{align*}
&R_{1}=g\kappa\frac{\rho^{2}}{\mu^{2}}\re\int_{\T}\frac{\Lambda(\overline{\de_{1}}\partial_{\alpha}\zz)(\gamma)-\partial_{\alpha}\zz(\gamma)\Lambda(\overline{\de_{1}})(\gamma)}{A(t)}\partial^{4}_{\alpha}\zz(\gamma)d\alpha,\\
&R_{2}=g\kappa\frac{\rho^{2}}{\mu^{2}}\re\int_{\T}\frac{\partial_{\alpha}\zz(\gamma)\Lambda(\overline{\de_{1}})(\gamma)}{A(t)}\partial^{4}_{\alpha}\zz(\gamma)d\alpha.
\end{align*}

Using the commutator estimate (\ref{conmutador}), we get
$$R_{1}\le C\norm{\F}_{L^{\infty}(S)}\norm{z}_{\mathcal{C}^{2,\delta}(S)}\norm{\de}^{2}_{\s}\le\est{4}.$$

The identity
$$\partial_{\alpha}\zz(\gamma)\de_{2}(\gamma)=\partial_{\alpha}z(\gamma)\cdot\de{(\gamma)}-\partial_{\alpha}\z(\gamma)\de_{1}(\gamma)$$

let us write $R_{2}$ as the sum of $S_{1}$ and $S_{2}$ where
\begin{align*}
&S_{1}=g\kappa\frac{\rho^{2}}{\mu^{2}}\re\int_{\T}\frac{\Lambda(\overline{\de_{1}})(\gamma)}{A(t)}\partial_{\alpha}z(\gamma)\cdot\de{(\gamma)}d\alpha,\\
&S_{2}=-g\kappa\frac{\rho^{2}}{\mu^{2}}\re\int_{\T}\frac{\Lambda(\overline{\de_{1}})(\gamma)}{A(t)}\partial_{\alpha}\z(\gamma)\de_{1}(\gamma)d\alpha.
\end{align*}
For $S_{1}$ we use an integration by parts and
$$\partial_{\alpha}z(\gamma)\cdot\de(\gamma)=-3\partial^{2}_{\alpha}z(\gamma)\cdot\partial^{3}_{\alpha}z(\gamma)$$
we get,
\begin{align*}
S_{1}&=-3g\kappa\frac{\rho^{2}}{\mu^{2}}\re\int_{\T}\frac{\Lambda(\overline{\de_{1}})(\gamma)}{A(t)}\partial^{2}_{\alpha}z(\gamma)\cdot\partial^{3}_{\alpha}z(\gamma)d\alpha\\
&=3g\kappa\frac{\rho^{2}}{\mu^{2}}\re\int_{\T}\frac{H(\overline{\de_{1}})(\gamma)}{A(t)}\partial^{3}_{\alpha}z(\gamma)\cdot\partial^{3}_{\alpha}z(\gamma)d\alpha\\
&+3g\kappa\frac{\rho^{2}}{\mu^{2}}\re\int_{\T}\frac{H(\overline{\de_{1}})(\gamma)}{A(t)}\partial^{2}_{\alpha}z(\gamma)\cdot\partial^{4}_{\alpha}z(\gamma)d\alpha\\
&\le\est{4}.
\end{align*}

	And writing $Q_{6}$ in the form
\begin{align*}
Q_{6}&=-g\kappa\frac{\rho^{2}}{\mu^{2}}\re\int_{\T}\frac{\Lambda(\overline{\de_{2}}\cdot\partial_{\alpha}\z)(\gamma)-\partial_{\alpha}\z(\gamma)\Lambda(\overline{\de_{2}})(\gamma)}{A(t)}\partial^{4}_{\alpha}\zz(\gamma)d\alpha\\
&-g\kappa\frac{\rho^{2}}{\mu^{2}}\re\int_{\T}\frac{\partial_{\alpha}\z(\gamma)}{A(t)}\Lambda(\overline{\de_{2}})(\gamma)\partial^{4}_{\alpha}\zz(\gamma)d\alpha\equiv R_{3}+R_{4},
\end{align*}

by the commutator estimate, we have
\begin{align*}
R_{3}\le C\norm{\F}_{L^{\infty}(S)}\norm{z}_{\mathcal{C}^{2,\delta}(S)}\norm{\de}^{2}_{\s}\le\est{4}.
\end{align*}

Since,
\begin{align*}
S_{2}+R_{4}=-g\kappa\frac{\rho^{2}}{\mu^{2}}\re\int_{\T}\frac{\partial_{\alpha}\z(\gamma)}{A(t)}\de{(\gamma)}\cdot\Lambda(\overline{\de{}})(\gamma)d\alpha
\end{align*}
we obtain finally
\begin{align*}
P_{7}&\le\est{4}-g\kappa\frac{\rho^{2}}{\mu^{2}}\re\int_{\T}\frac{\partial_{\alpha}\z(\gamma)}{A(t)}\de{(\gamma)}\cdot\Lambda(\overline{\de{}})(\gamma)d\alpha.
\end{align*}

In the estimate above we can observe how part of $\sigma(\gamma)$ appears in the non-integrable terms.

Let us return to $P_{8}=Q_{7}+Q_{8}+Q_{9}+Q_{10}$ where
\begin{align*}
&Q_{7}=-\re\int_{\T}\frac{\Lambda(\overline{\de{}}\cdot\partial_{\alpha}z^{\bot})(\gamma)}{A(t)}\partial^{3}_{\alpha}BR(z,\varpi)(\gamma)\cdot\partial_{\alpha}z(\gamma)d\alpha,\\
&Q_{8}=-3\re\int_{\T}\frac{\Lambda(\overline{\de{}}\cdot\partial_{\alpha}z^{\bot})(\gamma)}{A(t)}\partial^{2}_{\alpha}BR(z,\varpi)(\gamma)\cdot\partial^{2}_{\alpha}z(\gamma)d\alpha,\\
&Q_{9}=-3\re\int_{\T}\frac{\Lambda(\overline{\de{}}\cdot\partial_{\alpha}z^{\bot})(\gamma)}{A(t)}\partial_{\alpha}BR(z,\varpi)(\gamma)\cdot\partial^{3}_{\alpha}z(\gamma)d\alpha,\\
&Q_{10}=-\re\int_{\T}\frac{\Lambda(\overline{\de{}}\cdot\partial_{\alpha}z^{\bot})(\gamma)}{A(t)}BR(z,\varpi)(\gamma)\cdot\de(\gamma)d\alpha.\\
\end{align*}

We will control first the terms $Q_{8},Q_{7}$ and $Q_{9}$ and then we will show how the rest of $\sigma(\gamma)$ appears in $Q_{10}$.

Using $\Lambda=H\partial_{\alpha}$ and integrating by parts, we obtain
\begin{align*}
Q_{8}&=3\re\int_{\T}\frac{H(\overline{\de{}}\cdot\partial_{\alpha}z^{\bot})(\gamma)}{A(t)}\partial^{3}_{\alpha}BR(z,\varpi)(\gamma)\cdot\partial^{2}_{\alpha}z(\gamma)d\alpha\\
&+3\re\int_{\T}\frac{H(\overline{\de{}}\cdot\partial_{\alpha}z^{\bot})(\gamma)}{A(t)}\partial^{2}_{\alpha}BR(z,\varpi)(\gamma)\cdot\partial^{3}_{\alpha}z(\gamma)d\alpha\\
\equiv R_{5}+R_{6}.
\end{align*}

With (\ref{bir})
\begin{align*}
&R_{5}\le C\norm{\F}_{L^{\infty}(S)}\norm{z}_{\mathcal{C}^{2}(S)}\norm{\partial^{3}_{\alpha}BR}_{\s}\norm{\de\cdot\partial_{\alpha}z}_{\s}\\
&\le\est{4}
\end{align*}
and
\begin{align*}
&R_{6}\le C\norm{\F}_{L^{\infty}(S)}\norm{z}_{\mathcal{C}^{3}(S)}\norm{\partial^{2}_{\alpha}BR}_{\s}\norm{\de\cdot\partial_{\alpha}z}_{\s}\\
&\le\est{4}.
\end{align*}

With $Q_{7}$ we also integrate by parts to obtain $Q_{7}=R_{7}+R_{8}$ where
\begin{align*}
&R_{7}=\re\int_{\T}\frac{H(\overline{\de{}}\cdot\partial_{\alpha}z^{\bot})(\gamma)}{A(t)}\partial^{4}_{\alpha}BR(z,\varpi)(\gamma)\cdot\partial_{\alpha}z(\gamma)d\alpha,\\
&R_{8}=\re\int_{\T}\frac{H(\overline{\de{}}\cdot\partial_{\alpha}z^{\bot})(\gamma)}{A(t)}\partial^{3}_{\alpha}BR(z,\varpi)(\gamma)\cdot\partial^{2}_{\alpha}z(\gamma)d\alpha.
\end{align*}

Easily we have
\begin{align*}
R_{8}&\le C\norm{\F}_{L^{\infty}(S)}\norm{z}_{\mathcal{C}^{2}(S)}\norm{\de\cdot\partial_{\alpha}z}_{\s}\norm{\partial^{3}_{\alpha}BR}_{\s}\\
&\le\est{4}.
\end{align*}
In $R_{7}$ the application of Leibniz's rule to $\partial^{3}_{\alpha}BR(z,\varpi)$ produces many terms which can be estimated with the same tools used before. For the most singular terms we have the expressions:
\begin{align*}
&S_{3}=\re\int_{\T}\frac{H(\overline{\de{}}\cdot\partial_{\alpha}z^{\bot})(\gamma)}{A(t)}\partial_{\alpha}BR(z,\partial^{3}_{\alpha}\varpi)(\gamma)\cdot\partial_{\alpha}z(\gamma)d\alpha,\\
&S_{4}=\re\int_{\T}\frac{H(\overline{\de{}}\cdot\partial_{\alpha}z^{\bot})(\gamma)}{A(t)}\int_{\T}\frac{\Delta\de}{\abs{\Delta z}^{2}}\cdot\partial_{\alpha}z(\gamma)\varpi(\gamma-\beta)d\beta d\alpha,\\
&S_{5}=-2\re\int_{\T}\frac{H(\overline{\de{}}\cdot\partial_{\alpha}z^{\bot})(\gamma)}{A(t)}\int_{\T}\frac{\Delta z^{\bot}\cdot\partial_{\alpha}z(\gamma)}{\abs{\Delta z}^{4}}(\Delta z\cdot\Delta\de)\varpi(\gamma-\beta)d\alpha d\beta.
\end{align*}

Let us consider
\begin{align*}
\partial_{\alpha}BR(z,\varpi)(\gamma)\cdot\partial_{\alpha}z(\gamma)&=\partial_{\alpha}(BR(z,\varpi)(\gamma)\cdot\partial_{\alpha}z(\gamma))-BR(z,\varpi)(\gamma)\cdot\partial^{2}_{\alpha}z(\gamma)\\
&=\frac{1}{2}\partial_{\alpha}T(\varpi)(\gamma)-BR(z,\varpi)\cdot\partial^{2}_{\alpha}z(\gamma)
\end{align*}

which yields
\begin{align*}
S_{3}&\le C\norm{\F}_{L^{\infty}(S)}\norm{\de}_{\s}\norm{z}_{\mathcal{C}^{1}(S)}(\norm{T(\partial^{3}_{\alpha}\varpi)}_{H^{1}(S)}+\norm{BR(z,\partial^{3}_{\alpha}\varpi)}_{\s}\norm{z}_{\mathcal{C}^{2}(S)})\\
&\le\est{4}
\end{align*}

because $\norm{T}_{L^{2}\to H^{1}}\le\norm{\F}^{2}_{L^{\infty}}\norm{z}^{4}_{\mathcal{C}^{2,\delta}}$ for more details see Lemma $3.1$ in \cite{hele}.

Next we write $S_{4}=T_{1}+T_{2}$,
\begin{align*}
&T_{1}=\re\int_{\T}\int_{\R}\frac{H(\overline{\de{}}\cdot\partial_{\alpha}z^{\bot})(\gamma)}{A(t)}\Delta\de\cdot\partial_{\alpha}z(\gamma)\varpi(\gamma-\beta)(\frac{1}{\abs{\Delta z}^{2}}-\frac{1}{\abs{\partial_{\alpha}z(\gamma)}^{2}\beta^{2}})d\beta d\alpha,\\
&T_{2}=\re\int_{\T}\int_{\R}\frac{H(\overline{\de{}}\cdot\partial_{\alpha}z^{\bot})(\gamma)}{A(t)}\Delta\de\cdot\partial_{\alpha}z(\gamma)\frac{\varpi(\gamma-\beta)}{A(t)\beta^{2}}d\beta d\alpha.
\end{align*}

Using $B_2(\gamma,\beta)=B_3(\gamma,\beta)+B_4(\gamma,\beta)$, we split $T_{1}=U_{1}+U_{2}+U_{3}$,
\begin{align*}
&U_{1}=\re\int_{\T}\int_{\R}\frac{H(\overline{\de{}}\cdot\partial_{\alpha}z^{\bot})(\gamma)}{A(t)}\Delta\de\cdot\partial_{\alpha}z(\gamma)\varpi(\gamma-\beta)B_1(\gamma,\beta)d\beta d\alpha,\\
&U_{2}=\re\int_{\T}\int_{\R}\frac{H(\overline{\de{}}\cdot\partial_{\alpha}z^{\bot})(\gamma)}{A(t)}\Delta\de\cdot\partial_{\alpha}z(\gamma)\varpi(\gamma-\beta)B_3(\gamma,\beta)d\beta d\alpha,\\
&U_{3}=\re\int_{\T}\int_{\R}\frac{H(\overline{\de{}}\cdot\partial_{\alpha}z^{\bot})(\gamma)}{A(t)}\Delta\de\cdot\partial_{\alpha}z(\gamma)\varpi(\gamma-\beta)B_4(\gamma,\beta)d\beta d\alpha.
\end{align*}

being
\begin{align*}
&B_1(\gamma,\beta)=\frac{\beta\int _
{0}^{1}\int _
{0}^{1}\frac{\partial^{2}_{\alpha}z(\psi)-\partial^{2}_{\alpha}z(\gamma)}{\abs{\psi-\gamma}^{\delta}}\beta^{\delta}(1+s+t-st)^{\delta}(1-s)dtds\int _
{0}^{1}[\partial_{\alpha}z(\gamma)+\partial_{\alpha}z(\phi)]ds}{\abs{z(\gamma)-z(\gamma-\beta)}^{2}\abs{\partial_{\alpha}z(\gamma)}^{2}},\\
&B_3(\gamma,\beta)=\frac{\beta^{2}\partial^{2}_{\alpha}z(\gamma)\int _
{0}^{1}\int _
{0}^{1}\partial^{2}_{\alpha}z(\eta)(s-1)dtds}{\abs{z(\gamma)-z(\gamma-\beta)}^{2}\abs{\partial_{\alpha}z(\gamma)}^{2}},\\
&B_4(\gamma,\beta)=\frac{\beta\partial^{2}_{\alpha}z(\gamma)2\partial_{\alpha}z(\gamma)}{\abs{z(\gamma)-z(\gamma-\beta)}^{2}\abs{\partial_{\alpha}z(\gamma)}^{2}}
\end{align*}

therefore
\begin{align*}
&U_{1}\le C\norm{\F}^{2}_{L^{\infty}(S)}\norm{z}_{\mathcal{C}^{1}(S)}\norm{z}_{\mathcal{C}^{2,\delta}(S)}\norm{\varpi}_{L^{\infty}(S)}\norm{\de}^{2}_{\s},\\
&U_{2}\le C\norm{\F}^{2}_{L^{\infty}(S)}\norm{z}_{\mathcal{C}^{1}(S)}\norm{z}^{2}_{\mathcal{C}^{2}(S)}\norm{\varpi}_{L^{\infty}(S)}\norm{\de}^{2}_{\s}
\end{align*}

and
\begin{align*}
U_{3}&=2\re\int_{\T}\int_{\R}\frac{H(\overline{\de{}}\cdot\partial_{\alpha}z^{\bot})(\gamma)}{A^{2}(t)}\Delta\de\cdot\partial_{\alpha}z(\gamma)\varpi(\gamma-\beta)\beta\partial^{2}_{\alpha}z(\gamma)\partial_{\alpha}z(\gamma)B(\gamma,\beta)d\beta d\alpha\\
&+2\re\int_{\T}\int_{\R}\frac{H(\overline{\de{}}\cdot\partial_{\alpha}z^{\bot})(\gamma)}{A^{2}(t)}\Delta\de\cdot\partial_{\alpha}z(\gamma)\varpi(\gamma-\beta)\frac{\partial^{2}_{\alpha}z(\gamma)\partial_{\alpha}z(\gamma)}{\abs{\partial_{\alpha}z(\gamma)}^{2}\beta}d\beta d\alpha\equiv V_{1}+V_{2}.
\end{align*}
Recall that 
$$B(\gamma,\beta)\equiv\frac{\beta\int _
{0}^{1}\int _
{0}^{1}\partial^{2}_{\alpha}z(\psi)(1-s)dtds\int _
{0}^{1}\partial_{\alpha}z(\gamma)+\partial_{\alpha}z(\phi)ds}{\abs{z(\gamma)-z(\gamma-\beta)}^{2}\abs{\partial_{\alpha}z(\gamma)}^{2}},$$
then the term $V_{1}$ is controlled.

We split $V_{2}=W_{1}+W_{2}$ where
\begin{align*}
&W_{1}=2\re\int_{\T}\int_{\R}\frac{H(\overline{\de{}}\cdot\partial_{\alpha}z^{\bot})(\gamma)}{A^{2}(t)}\de{(\gamma)}\cdot\partial_{\alpha}z(\gamma)\varpi(\gamma-\beta)\frac{\partial^{2}_{\alpha}z(\gamma)\partial_{\alpha}z(\gamma)}{\abs{\partial_{\alpha}z(\gamma)}^{2}\beta}d\beta d\alpha,\\
&W_{2}=-2\re\int_{\T}\int_{\R}\frac{H(\overline{\de{}}\cdot\partial_{\alpha}z^{\bot})(\gamma)}{A^{2}(t)}\de{(\gamma-\beta)}\cdot\partial_{\alpha}z(\gamma)\varpi(\gamma-\beta)\frac{\partial^{2}_{\alpha}z(\gamma)\partial_{\alpha}z(\gamma)}{\abs{\partial_{\alpha}z(\gamma)}^{2}\beta}d\beta d\alpha.
\end{align*}

Easily
\begin{align*}
W_{1}&=2\re\int_{\T}\frac{H(\overline{\de{}}\cdot\partial_{\alpha}z^{\bot})(\gamma)}{A^{2}(t)}\de{(\gamma)}\cdot\partial_{\alpha}z(\gamma)H(\varpi)(\gamma)\frac{\partial^{2}_{\alpha}z(\gamma)\partial_{\alpha}z(\gamma)}{\abs{\partial_{\alpha}z(\gamma)}^{2}}d\alpha\\
&\le C\norm{\F}^{2}_{L^{\infty}(S)}\norm{\de}^{2}_{\s}\norm{z}_{\mathcal{C}^{1}(S)}
\norm{z}_{\mathcal{C}^{2}(S)}\norm{H\varpi}_{L^{\infty}(S)}\\
&\le\est{4}
\end{align*}

and 
\begin{align*}
W_{2}&\le C\norm{\F}^{2}_{L^{\infty}(S)}\norm{\de}^{2}_{\s}\norm{z}_{\mathcal{C}^{1}(S)}
\norm{z}_{\mathcal{C}^{2}(S)}\norm{\varpi}_{\mathcal{C}^{1}(S)}\\
&-2\re\int_{\T}\frac{H(\overline{\de{}}\cdot\partial_{\alpha}z^{\bot})(\gamma)}{A^{2}(t)}H(\de)(\gamma)\cdot\partial_{\alpha}z(\gamma)\varpi(\gamma)\frac{\partial^{2}_{\alpha}z(\gamma)\partial_{\alpha}z(\gamma)}{\abs{\partial_{\alpha}z(\gamma)}^{2}}d\alpha\\
&\le\est{4}
\end{align*}

Hence,
$$T_{1}\le\est{4}.$$

We decompose $T_{2}=U_{4}+U_{5}$,
\begin{align*}
&U_{4}=\re\int_{\T}\int_{\R}\frac{H(\overline{\de{}}\cdot\partial_{\alpha}z^{\bot})(\gamma)}{A^{2}(t)}\Delta\de\cdot\partial_{\alpha}z(\gamma)\frac{\varpi(\gamma-\beta)-\varpi(\gamma)}{\beta^{2}}d\beta d\alpha,\\
&U_{5}=\re\int_{\T}\int_{\R}\frac{H(\overline{\de{}}\cdot\partial_{\alpha}z^{\bot})(\gamma)}{A^{2}(t)}\Delta\de\cdot\partial_{\alpha}z(\gamma)\frac{\varpi(\gamma)}{\beta^{2}}d\beta d\alpha.
\end{align*}

Then we split
\begin{align*}
U_{4}&=\re\int_{\T}\int_{\R}\frac{H(\overline{\de{}}\cdot\partial_{\alpha}z^{\bot})(\gamma)}{A^{2}(t)}\de(\gamma)\cdot\partial_{\alpha}z(\gamma)\frac{\varpi(\gamma-\beta)-\varpi(\gamma)}{\beta^{2}}d\beta d\alpha\\
&-\re\int_{\T}\int_{\R}\frac{H(\overline{\de{}}\cdot\partial_{\alpha}z^{\bot})(\gamma)}{A^{2}(t)}\de(\gamma-\beta)\cdot\partial_{\alpha}z(\gamma)\frac{\varpi(\gamma-\beta)-\varpi(\gamma)}{\beta^{2}}d\beta d\alpha\\
&\equiv V_{3}+V_{4}
\end{align*}

where
\begin{displaymath}
V_{3}\le\est{4}
\end{displaymath}
and, 
\begin{align*}
V_{4}&\le C\norm{\F}^{\frac{3}{2}}_{L^{\infty}(S)}\norm{\de}^{2}_{\s}\norm{z}_{\mathcal{C}^{1}(S)}\norm{\varpi}_{\mathcal{C}^{2}(S)}\\
&-\pi\re\int_{\T}\frac{H(\overline{\de{}}\cdot\partial_{\alpha}z^{\bot})(\gamma)}{A^{2}(t)}H(\de)(\gamma)\cdot\partial_{\alpha}z(\gamma)\partial_{\alpha}\varpi(\gamma)d\alpha\\
&\le\est{4}.
\end{align*}
Then,
$$U_{4}\le\est{4}.$$

For $U_{5}$ integrating by parts for $\Lambda$ we have,
\begin{align*}
U_{5}&=\re\int_{\T}\Lambda(\frac{H(\overline{\de{}}\cdot\partial_{\alpha}z^{\bot})}{A^{2}(t)}\partial_{\alpha}z\varpi)\cdot\de(\gamma)d\alpha\\
&=\re\int_{\T}(\Lambda(\frac{H(\overline{\de{}}\cdot\partial_{\alpha}z^{\bot})}{A^{2}(t)}\partial_{\alpha}z\varpi)(\gamma)-\partial_{\alpha}z(\gamma)\Lambda(\frac{H(\overline{\de{}}\cdot\partial_{\alpha}z^{\bot})}{A^{2}(t)})(\gamma))\cdot\de(\gamma)d\alpha\\
&+\re\int_{\T}\partial_{\alpha}z(\gamma)\Lambda(\frac{H(\overline{\de{}}\cdot\partial_{\alpha}z^{\bot})}{A^{2}(t)})(\gamma)\cdot\de(\gamma)d\alpha\\
&\le\est{4}-\re\int_{\T}\partial_{\alpha}z(\gamma)\frac{\partial_{\alpha}(\overline{\de{}}\cdot\partial_{\alpha}z^{\bot})(\gamma)}{A^{2}(t)}\cdot\de(\gamma)d\alpha\\
\end{align*}
now, using $\partial_{\alpha}z(\alpha)\cdot\partial^{4}_{\alpha}z(\alpha)=-3\partial^{2}_{\alpha}z(\alpha)\cdot\partial_{\alpha}^{3}z(\alpha)$
\begin{align*}
&-\re\int_{\T}\frac{\partial_{\alpha}(\overline{\de{}}\cdot\partial_{\alpha}z^{\bot})(\gamma)}{A^{2}(t)}\partial_{\alpha}z(\gamma)\cdot\de(\gamma)d\alpha=3\re\int_{\T}\frac{\partial_{\alpha}(\overline{\de{}}\cdot\partial_{\alpha}z^{\bot})(\gamma)}{A^{2}(t)}\partial^{2}_{\alpha}z(\gamma)\cdot\partial_{\alpha}^{3}z(\gamma)d\alpha\\
&=-3\re\int_{\T}\frac{\overline{\de{}}\cdot\partial_{\alpha}z^{\bot}(\gamma)}{A^{2}(t)}\partial^{3}_{\alpha}z(\gamma)\cdot\partial_{\alpha}^{3}z(\gamma)d\alpha-3\re\int_{\T}\frac{\overline{\de{}}\cdot\partial_{\alpha}z^{\bot}(\gamma)}{A^{2}(t)}\partial^{2}_{\alpha}z(\gamma)\cdot\partial_{\alpha}^{4}z(\gamma)d\alpha\\
&\le\est{4}.
\end{align*}

Therefore,
$$T_{2}\le\est{4}.$$
Thus, $S_{4}$ satisfies identical estimates than $T_{2}$.

To conclude with $R_{7}$, let us estimate $S_{5}$.

We split $S_{5}=T_{3}+T_{4}$

\begin{align*}
&T_{3}=-2\re\int_{\T}\int_{\R}\frac{H(\overline{\de{}}\cdot\partial_{\alpha}z^{\bot})(\gamma)}{A(t)}C(\gamma,\beta)\cdot\partial_{\alpha}z(\gamma)(\Delta z\cdot\Delta\de)\varpi(\gamma-\beta)d\alpha d\beta,\\
&T_{4}=-2\re\int_{\T}\int_{\R}\frac{H(\overline{\de{}}\cdot\partial_{\alpha}z^{\bot})(\gamma)}{A(t)}\frac{\partial^{\bot}_{\alpha}z(\gamma)}{A^{2}(t)\beta^{3}}\cdot\partial_{\alpha}z(\gamma)(\Delta z\cdot\Delta\de)\varpi(\gamma-\beta)d\alpha d\beta.
\end{align*}

Since $\partial^{\bot}_{\alpha}z(\gamma)\cdot\partial_{\alpha}z(\gamma)=0$, for (\ref{c(gamma)}) we have $C(\gamma,\beta)\cdot\partial_{\alpha}z(\gamma)=C_1(\gamma,\beta)\cdot\partial_{\alpha}z(\gamma)$ and $T_{4}=0$.

Recall that,
\begin{align*}
&C_1(\gamma,\beta)=\frac{\beta^{2}\int _
{0}^{1}\int _
{0}^{1}\partial^{2}_{\alpha}z^{\bot}(\eta)(s-1)dsdt}{\abs{\Delta z}^{4}}
\end{align*}
with $\eta=\gamma-t\beta+st\beta.$

Using 
\begin{displaymath}
\Delta\partial^{k}_{\alpha}z=\beta\int _
{0}^{1}\partial^{k+1}_{\alpha}z(\phi)ds
\end{displaymath}
and 
\begin{align*}
&C_1(\gamma,\beta)\cdot\partial_{\alpha}z(\gamma)-\frac{\beta^{2}\partial^{2}_{\alpha}z^{\bot}(\gamma)\cdot\partial_{\alpha}z(\gamma)}{\abs{\Delta z}^{4}}=\frac{\beta^{2}\int _
{0}^{1}\int _
{0}^{1}[\partial^{2}_{\alpha}z^{\bot}(\eta)-\partial^{2}_{\alpha}z(\gamma)](s-1)dtds\cdot\partial_{\alpha}z(\gamma)}{\abs{\Delta z}^{4}}
\end{align*}
we get
\begin{align*}
S_{5}&\le\est{4}\\
&-2\re\int_{\T}\int_{\R}\frac{H(\overline{\de{}}\cdot\partial_{\alpha}z^{\bot})(\gamma)}{A^{3}(t)}\partial^{2}_{\alpha}z^{\bot}(\gamma)\cdot\partial_{\alpha}z(\gamma)\partial_{\alpha}z(\gamma)\cdot\Delta\de\frac{\varpi(\gamma-\beta)}{\beta}d\alpha d\beta\\
&\le\est{4}\\
&-2\pi\re\int_{\T}\frac{H(\overline{\de{}}\cdot\partial_{\alpha}z^{\bot})(\gamma)}{A^{3}(t)}\partial^{2}_{\alpha}z^{\bot}(\gamma)\cdot\partial_{\alpha}z(\gamma)\partial_{\alpha}z(\gamma)\cdot\de(\gamma)H(\varpi)(\gamma)d\alpha\\
&-4\pi\re\int_{\T}\frac{H(\overline{\de{}}\cdot\partial_{\alpha}z^{\bot})(\gamma)}{A^{3}(t)}\partial^{2}_{\alpha}z^{\bot}(\gamma)\cdot\partial_{\alpha}z(\gamma)\partial_{\alpha}z(\gamma)\cdot H(\de)(\gamma)\varpi(\gamma)d\alpha.\\
\end{align*}

Therefore we can control $S_{5}$.  

Let us decompose
\begin{align*}
Q_{9}=&3\re\int_{\T}\frac{H(\overline{\de{}}\cdot\partial_{\alpha}z^{\bot})(\gamma)}{A(t)}\partial^{2}_{\alpha}BR(z,\varpi)(\gamma)\cdot\partial^{3}_{\alpha}z(\gamma)d\alpha\\
&+3\re\int_{\T}\frac{H(\overline{\de{}}\cdot\partial_{\alpha}z^{\bot})(\gamma)}{A(t)}\partial_{\alpha}BR(z,\varpi)(\gamma)\cdot\partial^{4}_{\alpha}z(\gamma)d\alpha\equiv R_{9}+R_{10},
\end{align*}
using (\ref{bir}) 
\begin{align*}
&R_{9}\le C\norm{\F}_{L^{\infty}(S)}\norm{z}_{\mathcal{C}^{3}(S)}\norm{BR}_{H^{2}(S)}\norm{z}_{\mathcal{C}^{1}(S)}\norm{\de}_{\s}\le\est{4},\\
&R_{10}\le C\norm{\F}_{L^{\infty}(S)}\norm{\partial_{\alpha}BR}_{L^{\infty}(S)}\norm{z}_{\mathcal{C}^{1}(S)}\norm{\de}^{2}_{\s}\le\est{4}.
\end{align*}

Then $Q_{9}\le\est{4}$.

Finally we have to find the rest of $\sigma(\gamma)$ in $Q_{10}$. To do that let us split $Q_{10}=R_{11}+R_{12}+R_{13}+R_{14}$ where
\begin{align*}
&R_{11}=\re\int_{\T}\frac{\Lambda(\overline{\de_{1}}\partial_{\alpha}\zz)(\gamma)}{A(t)}BR_{1}(z,\varpi)(\gamma)\de_{1}(\gamma)d\alpha,\\
&R_{12}=\re\int_{\T}\frac{\Lambda(\overline{\de_{1}}\partial_{\alpha}\zz)(\gamma)}{A(t)}BR_{2}(z,\varpi)(\gamma)\de_{2}(\gamma)d\alpha,\\
&R_{13}=-\re\int_{\T}\frac{\Lambda(\overline{\de_{2}}\partial_{\alpha}\z)(\gamma)}{A(t)}BR_{1}(z,\varpi)(\gamma)\de_{1}(\gamma)d\alpha,\\
&R_{14}=-\re\int_{\T}\frac{\Lambda(\overline{\de_{2}}\partial_{\alpha}\z)(\gamma)}{A(t)}BR_{2}(z,\varpi)(\gamma)\de_{2}(\gamma)d\alpha.\\
\end{align*}

Then
\begin{align*}
R_{11}&=\re\int_{\T}\frac{\Lambda(\overline{\de_{1}}\partial_{\alpha}\zz)(\gamma)-\partial_{\alpha}\zz(\gamma)\Lambda(\overline{\de_{1}})(\gamma)}{A(t)}BR_{1}(z,\varpi)(\gamma)\de_{1}(\gamma)d\alpha\\
&+\re\int_{\T}\frac{\partial_{\alpha}\zz(\gamma)\Lambda(\overline{\de_{1}})(\gamma)}{A(t)}BR_{1}(z,\varpi)(\gamma)\de_{1}(\gamma)d\alpha,
\end{align*}

and the commutator estimates yields
\begin{displaymath}
R_{11}\le\est{4}+\re\int_{\T}\frac{\partial_{\alpha}\zz(\gamma)BR_{1}(z,\varpi)(\gamma)}{A(t)}\de_{1}(\gamma)\Lambda(\overline{\de_{1}})(\gamma)d\alpha.
\end{displaymath}

In a similar way we have
\begin{align*}
&R_{12}\le\est{4}+\re\int_{\T}\frac{\partial_{\alpha}\zz(\gamma)BR_{2}(z,\varpi)(\gamma)}{A(t)}\de_{2}(\gamma)\Lambda(\overline{\de_{1}})(\gamma)d\alpha,\\
&R_{13}\le\est{4}-\re\int_{\T}\frac{\partial_{\alpha}\z(\gamma)BR_{1}(z,\varpi)(\gamma)}{A(t)}\de_{1}(\gamma)\Lambda(\overline{\de_{2}})(\gamma)d\alpha,\\
&R_{14}\le\est{4}-\re\int_{\T}\frac{\partial_{\alpha}\z(\gamma)BR_{2}(z,\varpi)(\gamma)}{A(t)}\de_{2}(\gamma)\Lambda(\overline{\de_{2}})(\gamma)d\alpha.
\end{align*}

Since, $$\partial_{\alpha}\zz\partial_{\alpha}^{4}\zz=\partial_{\alpha}z\cdot\de{}-\partial_{\alpha}\z\de_{1}=-3\partial^{2}_{\alpha}z\cdot\partial^{3}_{\alpha}z-\partial_{\alpha}\z\de_{1}$$
and $H\partial_{\alpha}=\Lambda$, using integration by parts
\begin{align*}
&R_{12}\le\est{4}-3\re\int_{\T}\partial_{\alpha}(\frac{BR_{2}(z,\varpi)(\gamma)}{A(t)}\partial^{2}_{\alpha}z\cdot\partial^{3}_{\alpha}z)H(\overline{\de_{1}})(\gamma)d\alpha\\
&-\re\int_{\T}\frac{\partial_{\alpha}\z(\gamma)BR_{2}(z,\varpi)(\gamma)}{A(t)}\de_{1}(\gamma)\Lambda(\overline{\de_{1}})(\gamma)d\alpha\\
&\le\est{4}-\re\int_{\T}\frac{\partial_{\alpha}\z(\gamma)BR_{2}(z,\varpi)(\gamma)}{A(t)}\de_{1}(\gamma)\Lambda(\overline{\de_{1}})(\gamma)d\alpha.\\
\end{align*}
And in the same way,
\begin{align*}
&R_{13}\le\est{4}+\re\int_{\T}\frac{\partial_{\alpha}\zz(\gamma)BR_{1}(z,\varpi)(\gamma)}{A(t)}\de_{2}(\gamma)\Lambda(\overline{\de_{2}})(\gamma)d\alpha.
\end{align*}

Therefore,
\begin{align*}
&R_{11}+R_{12}+R_{13}+R_{14}\le\est{4}\\
&-\re\int_{\T}\frac{BR(z,\varpi)(\gamma)\cdot\partial^{\bot}_{\alpha}z(\gamma)}{A(t)}\de(\gamma)\cdot\Lambda(\overline{\de{}})(\gamma)d\alpha.
\end{align*}
Then,
\begin{align*}
&P_{7}+P_{8}\le\est{4}\\
&-\re\int_{\T}\frac{BR(z,\varpi)(\gamma)\cdot\partial_{\alpha}z^{\bot}(\gamma)}{A(t)}\de(\gamma)\cdot\Lambda(\overline{\de{}})(\gamma)d\alpha-g\kappa\frac{\rho^{2}}{\mu^{2}}\re\int_{\T}\frac{\partial_{\alpha}\z(\gamma)}{A(t)}\de{(\gamma)}\cdot\Lambda(\overline{\de{}})(\gamma)d\alpha.
\end{align*}
Let us look at these last two terms,
\begin{align*}
&-\re\int_{\T}\frac{BR(z,\varpi)(\gamma)\cdot\partial_{\alpha}z^{\bot}(\gamma)}{A(t)}\de(\gamma)\cdot\Lambda(\overline{\de{}})(\gamma)d\alpha-\kappa\frac{\rho^{2}}{\mu^{2}}\re\int_{\T}\frac{\partial_{\alpha}\z(\gamma)}{A(t)}\de{(\gamma)}\cdot\Lambda(\overline{\de{}})(\gamma)d\alpha\\
&=-\re\int_{\T}\frac{\sigma(\gamma,t)}{A(t)}\de{(\gamma)}\cdot\Lambda(\overline{\de{}})(\gamma)d\alpha\\
&=\int_{\T}\im(\frac{\sigma}{A(t)})(-\re(\de)\cdot\im(\Lambda(\de))+\im(\de)\cdot\re(\Lambda(\de)))d\alpha\\
&-\int_{\T}\re(\frac{\sigma}{A(t)})(\re(\de)\cdot\re(\Lambda(\de))+\im(\de)\cdot\im(\Lambda(\de)))d\alpha\equiv Y_{1}+Y_{2},
\end{align*}
we get
\begin{align*}
Y_{1}&=\int_{\T}(-\Lambda(\im(\frac{\sigma}{A(t)})\re(\de))+\im(\frac{\sigma}{A(t)})\re(\Lambda(\de)))\cdot\im(\de)d\alpha\\
&\le C\norm{\frac{\sigma}{A(t)}}_{\mathcal{C}^{1,\delta}(S)}\norm{\de}^{2}_{\s}\le\est{4}
\end{align*}
and
\begin{align*}
Y_{2}&=-\int_{\T}(\re(\frac{\sigma}{A(t)})-m(t))(\re(\de)\cdot\re(\Lambda(\de))+\im(\de)\cdot\im(\Lambda(\de)))d\alpha\\
&-\int_{\T} m(t)(\re(\de)\cdot\re(\Lambda(\de))+\im(\de)\cdot\im(\Lambda(\de)))d\alpha\\
&\equiv Y_{3}+Y_{4}
\end{align*}
where $$m(t)=\min_{\gamma}\sigma(\gamma,t).$$

Since $\re(\frac{\sigma}{A(t)})-m(t)>0$ using $2g\Lambda(g)-\Lambda(g^{2})\ge 0$, see \cite{point}
\begin{align*}
&Y_{3}\le\frac{1}{2}\norm{\Lambda(\re(\frac{\sigma}{A(t)}))}_{L^{\infty}(S)}\norm{\de}^{2}_{\s}\le C\norm{\frac{\sigma}{A(t)}}_{\mathcal{C}^{1,\delta}(S)}\norm{\de}^{2}_{\s}\\
&\le\est{4},\\
&Y_{4}=-m(t)\norm{\Lambda^{\frac{1}{2}}\de}^{2}_{\s}
\end{align*}
Combining all previous estimates
\begin{align*}
I_{7}\le\est{4}-m(t)\norm{\Lambda^{\frac{1}{2}}\de}^{2}_{\s}.
\end{align*}

\subsubsection{Estimates on $\partial^{4}_{\alpha}(c(\gamma,t)\cdot\partial_{\alpha}z(\gamma,t))$ for $J_{2}$}
\label{cest}

In the evolution of the norm of $\de(\gamma)$ it remains to control the term 
\begin{align*}
J_{2}&=\re\int_{\T}\overline{\de(\gamma)}\cdot\partial^{4}_{\alpha}c(\gamma)\partial_{\alpha}z(\gamma)d\alpha+4\re\int_{\T}\overline{\de(\gamma)}\cdot\partial^{3}_{\alpha}c(\gamma)\partial^{2}_{\alpha}z(\gamma)d\alpha\\
&+6\re\int_{\T}\overline{\de(\gamma)}\cdot\partial^{2}_{\alpha}c(\gamma)\partial^{3}_{\alpha}z(\gamma)d\alpha+4\re\int_{\T}\overline{\de(\gamma)}\cdot\partial_{\alpha}c(\gamma)\partial^{4}_{\alpha}z(\gamma)d\alpha\\
&+\re\int_{\T}\overline{\de(\gamma)}\cdot c(\gamma)\partial^{5}_{\alpha}z(\gamma)d\alpha\equiv Q_{1}+Q_{2}+Q_{3}+Q_{4}+Q_{5}.
\end{align*}

Let us recall the formula
\begin{align*}
c(\alpha,t)=\frac{\alpha+\pi}{2\pi}\int_{\T}{\frac{\partial_{\beta}z(\beta,t)}{\abs{\partial_{\beta}z(\beta,t)}^{2}}\cdot\partial_{\beta}BR(z,\varpi)(\beta,t)d\beta}\\-\int_{-\pi}^{\alpha}\frac{\partial_{\beta}z(\beta,t)}{\abs{\partial_{\beta}z(\beta,t)}^{2}}\cdot\partial_{\beta}BR(z,\varpi)(\beta,t)d\beta,
\end{align*}
then
\begin{align*}
Q_{2}&=4\re\int_{\T}\frac{\overline{\de(\gamma)}}{A(t)}\cdot\partial^{2}_{\alpha}z(\gamma)\partial^{3}_{\alpha}z(\gamma)\cdot\partial_{\alpha}BR(z,\varpi)(\gamma)d\alpha\\
&+8\re\int_{\T}\frac{\overline{\de(\gamma)}}{A(t)}\cdot\partial^{2}_{\alpha}z(\gamma)\partial^{2}_{\alpha}z(\gamma)\cdot\partial^{2}_{\alpha}BR(z,\varpi)(\gamma)d\alpha\\
&+4\re\int_{\T}\frac{\overline{\de(\gamma)}}{A(t)}\cdot\partial^{2}_{\alpha}z(\gamma)\partial_{\alpha}z(\gamma)\cdot\partial^{3}_{\alpha}BR(z,\varpi)(\gamma)d\alpha\equiv N_{1}+N_{2}+N_{3}
\end{align*}
and
\begin{align*}
&N_{1}\le C\norm{\F}_{L^{\infty}(S)}\norm{z}_{\mathcal{C}^{2}(S)}\norm{z}_{\mathcal{C}^{3}(S)}\norm{BR(z,\varpi)}_{H^{1}(S)}\norm{\de}_{\s},\\
&N_{2}\le C\norm{\F}_{L^{\infty}(S)}\norm{z}^{2}_{\mathcal{C}^{2}(S)}\norm{\partial^{2}_{\alpha}BR(z,\varpi)}_{\s}\norm{\de}_{\s},\\
&N_{3}\le C\norm{\F}_{L^{\infty}(S)}\norm{z}_{\mathcal{C}^{2}(S)}\norm{z}_{\mathcal{C}^{1}(S)}\norm{\partial^{3}_{\alpha}BR(z,\varpi)}_{\s}\norm{\de}_{\s}.
\end{align*}
Thus $$Q_{2}\le\est{4}.$$

In the same way,
\begin{align*}
Q_{3}&=-6\re\int_{\T}\overline{\de(\gamma)}\cdot\partial^{3}_{\alpha}z(\gamma)\frac{\partial^{2}_{\alpha}z(\gamma)}{\abs{\partial_{\alpha}z(\gamma)}^{2}}\cdot\partial_{\alpha}BR(z,\varpi)(\gamma)d\alpha\\
&-6\re\int_{\T}\overline{\de(\gamma)}\cdot\partial^{3}_{\alpha}z(\gamma)\frac{\partial_{\alpha}z(\gamma)}{\abs{\partial_{\alpha}z(\gamma)}^{2}}\cdot\partial^{2}_{\alpha}BR(z,\varpi)(\gamma)d\alpha\equiv N_{4}+N_{5}
\end{align*}
where
\begin{align*}
&N_{4}\le C\norm{\F}_{L^{\infty}(S)}\norm{z}_{\mathcal{C}^{2}(S)}\norm{z}_{\mathcal{C}^{3}(S)}\norm{\de}_{\s}\norm{\partial_{\alpha}BR(z,\varpi)}_{\s},\\
&N_{5}\le C\norm{\F}_{L^{\infty}(S)}\norm{z}_{\mathcal{C}^{3}(S)}\norm{z}_{\mathcal{C}^{1}(S)}\norm{\de}_{\s}\norm{\partial^{2}_{\alpha}BR(z,\varpi)}_{\s},
\end{align*}
thus $$Q_{3}\le\est{4}.$$

The term $Q_{4}$ satisfies
\begin{align*}
Q_{4}&\le C\norm{\partial_{\alpha}c}_{L^{\infty}(S)}\norm{\de}^{2}_{\s}\le C\norm{\F}^{\frac{1}{2}}_{L^{\infty}(S)}\norm{\partial_{\alpha}BR}_{L^{\infty}(S)}\norm{\de}^{2}_{\s}\\
&\le\est{4}
\end{align*}

and for $Q_{5}$
\begin{align*}
Q_{5}&=\re\int_{\T} c(\gamma)\overline{\de(\gamma)}\cdot \dee(\gamma)d\alpha\\
&=\int_{\T}\re(c)(\re(\de)\re(\dee)+\im(\de)\im(\dee))d\alpha\\
&+\int_{\T}\im(c)(-\re(\de)\im(\dee)+\im(\de)\re(\dee))d\alpha\\
&\equiv Q_{5}^{1}+Q_{5}^{2}
\end{align*}
where,
\begin{align*}
&Q_{5}^{1}=-\frac{1}{2}\int_{\T}\re(\partial_{\alpha}c)\abs{\de}^{2}d\alpha\le\norm{\partial_{\alpha}z}_{L^{\infty}}\norm{\de}^{2}_{\s}\\
&\le C\norm{\F}^{\frac{1}{2}}_{L^{\infty}(S)}\norm{\partial_{\alpha}BR(z,\varpi)}_{L^{\infty}(S)}\norm{\de}^{2}_{\s}\le\est{4}
\end{align*}
and
\begin{align*}
Q_{5}^{2}&=\int_{\T}\im(\partial_{\alpha}c)\re(\de)\im(\de)d\alpha+2\int_{\T}\im(c)\im(\de)\re(\dee)d\alpha\\
&\le\norm{\partial_{\alpha}c}_{L^{\infty}(S)}\norm{\de}^{2}_{\s}-2\int_{\T}\im(c)\im(\de)\re(\Lambda(H(\de)))d\alpha\\
&\le\est{4}-2\int_{\T}\Lambda^{\frac{1}{2}}(\im(c)\im(\de))\re(\Lambda^{\frac{1}{2}}(H(\de)))d\alpha\\
&\le\est{4}+K\norm{\im(c)}_{H^{2}(S)}\norm{\Lambda^{\frac{1}{2}}\de}^{2}_{\s}.
\end{align*}
Finally,
\begin{align*}
Q_{1}&=\re\int_{\T}\frac{\overline{\de(\gamma)}}{A(t)}\cdot\partial_{\alpha}z(\gamma)\de(\gamma)\cdot\partial_{\alpha}BR(z,\varpi)(\gamma)d\alpha\\
&-3\re\int_{\T}\frac{\overline{\de(\gamma)}}{A(t)}\cdot\partial_{\alpha}z(\gamma)\partial^{3}_{\alpha}z(\gamma)\cdot\partial^{2}_{\alpha}BR(z,\varpi)(\gamma)d\alpha\\
&-3\re\int_{\T}\frac{\overline{\de(\gamma)}}{A(t)}\cdot\partial_{\alpha}z(\gamma)\partial^{2}_{\alpha}z(\gamma)\cdot\partial^{3}_{\alpha}BR(z,\varpi)(\gamma)d\alpha\\
&-\re\int_{\T}\frac{\overline{\de(\gamma)}}{A(t)}\cdot\partial_{\alpha}z(\gamma)\partial_{\alpha}z(\gamma)\cdot\partial^{4}_{\alpha}BR(z,\varpi)(\gamma)d\alpha\\
&\equiv N_{6}+N_{7}+N_{8}+N_{9}
\end{align*}
where
\begin{align*}
&N_{6}\le C\norm{\F}_{L^{\infty}(S)}\norm{z}_{\mathcal{C}^{1}(S)}\norm{\partial_{\alpha}BR}_{L^{\infty}(S)}\norm{\de}^{2}_{\s},\\
&N_{7}\le C\norm{\F}_{L^{\infty}(S)}\norm{z}_{\mathcal{C}^{1}(S)}\norm{z}_{\mathcal{C}^{3}(S)}\norm{\partial^{2}_{\alpha}BR}_{\s}\norm{\de}_{\s},\\
&N_{8}\le C\norm{\F}_{L^{\infty}(S)}\norm{z}_{\mathcal{C}^{1}(S)}\norm{z}_{\mathcal{C}^{2}(S)}\norm{\partial^{3}_{\alpha}BR}_{\s}\norm{\de}_{\s}.
\end{align*}

To estimate $N_{9}$, we must proceed in the same way we did with $J_{1}$. We split $N_{9}=I'_{3}+I'_{4}+I'_{5}+I'_{6}+I'_{7}$
\begin{align*}
&I'_{3}=-\re\int_{\T}\int_{\R}\frac{\overline{\de(\gamma)}}{A(t)}\cdot\partial_{\alpha}z(\gamma)\partial_{\alpha}z(\gamma)\cdot\partial^{4}_{\alpha}(\frac{(z(\gamma)-z(\gamma-\beta))^{\bot}}{\abs{z(\gamma)-z(\gamma-\beta)}^{2}})\varpi(\gamma-\beta)d\alpha d\beta,\\
&I'_{4}=-4\re\int_{\T}\int_{\R}\frac{\overline{\de(\gamma)}}{A(t)}\cdot\partial_{\alpha}z(\gamma)\partial_{\alpha}z(\gamma)\cdot\partial^{3}_{\alpha}(\frac{(z(\gamma)-z(\gamma-\beta))^{\bot}}{\abs{z(\gamma)-z(\gamma-\beta)}^{2}})\partial_{\alpha}\varpi(\gamma-\beta)d\alpha d\beta,\\
&I'_{5}=-6\re\int_{\T}\int_{\R}\frac{\overline{\de(\gamma)}}{A(t)}\cdot\partial_{\alpha}z(\gamma)\partial_{\alpha}z(\gamma)\cdot\partial^{2}_{\alpha}(\frac{(z(\gamma)-z(\gamma-\beta))^{\bot}}{\abs{z(\gamma)-z(\gamma-\beta)}^{2}})\partial^{2}_{\alpha}\varpi(\gamma-\beta)d\alpha d\beta,\\
&I'_{6}=-4\re\int_{\T}\int_{\R}\frac{\overline{\de(\gamma)}}{A(t)}\cdot\partial_{\alpha}z(\gamma)\partial_{\alpha}z(\gamma)\cdot\partial_{\alpha}(\frac{(z(\gamma)-z(\gamma-\beta))^{\bot}}{\abs{z(\gamma)-z(\gamma-\beta)}^{2}})\partial^{3}_{\alpha}\varpi(\gamma-\beta)d\alpha d\beta,\\
&I'_{7}=-\re\int_{\T}\int_{\R}\frac{\overline{\de(\gamma)}}{A(t)}\cdot\partial_{\alpha}z(\gamma)\partial_{\alpha}z(\gamma)\cdot(\frac{(z(\gamma)-z(\gamma-\beta))^{\bot}}{\abs{z(\gamma)-z(\gamma-\beta)}^{2}})\partial^{4}_{\alpha}\varpi(\gamma-\beta)d\alpha d\beta.\\
\end{align*}

To study this terms we have to repeat all estimates as in section \ref{estbr}. We select only the terms with different decompositions and we leave to the reader the remainder easy cases. 

If we consider the term corresponding to $Q_{4}$ in section \ref{estbr} we have since
\begin{align*}
&\re(\overline{\de{}}\cdot\partial_{\alpha}z)=\re(\de)\cdot\re(\partial_{\alpha}z)+\im(\de)\cdot\im(\partial_{\alpha}z),\\
&\im(\overline{\de{}}\cdot\partial_{\alpha}z)=-\re(\de)\cdot\im(\partial_{\alpha}z)+\im(\de)\cdot\re(\partial_{\alpha}z),\\
&\re(\de\cdot\partial_{\alpha}z)=\re(\de)\cdot\re(\partial_{\alpha}z)-\im(\de)\cdot\im(\partial_{\alpha}z),\\
&\im(\de\cdot\partial_{\alpha}z)=\im(\de)\cdot\re(\partial_{\alpha}z)+\re(\de)\cdot\im(\partial_{\alpha}z),\\
\end{align*}
and
$$\partial_{\alpha}z\cdot\de=-3\partial^{2}_{\alpha}z\cdot\partial_{\alpha}^{3}z.$$
we can write
\begin{align*}
&\re(\overline{\de{}}\cdot\partial_{\alpha}z)=\re(-3\partial^{2}_{\alpha}z\cdot\partial^{3}_{\alpha}z)+2\im(\de)\cdot\im(\partial_{\alpha}z),\\
&\im(\overline{\de{}}\cdot\partial_{\alpha}z)=\im(-3\partial^{2}_{\alpha}z\cdot\partial^{3}_{\alpha}z)-2\re(\de)\cdot\im(\partial_{\alpha}z).\\
\end{align*}
Thus
\begin{align*}
Q'_{4}&=-2\pi\re\int_{\T}\frac{\overline{\de(\gamma)}}{A^{2}(t)}\cdot\partial_{\alpha}z(\gamma)\partial_{\alpha}z(\gamma)\cdot\Lambda(\de^{\bot})(\gamma)\varpi(\gamma)d\alpha\\
&=-2\pi\int_{\T}\re(\overline{\de{}}\cdot\partial_{\alpha}z)\re(\partial_{\alpha}z\cdot\Lambda(\de^{\bot})\frac{\varpi}{A^{2}(t)})-\im(\overline{\de{}}\cdot\partial_{\alpha}z)\im(\partial_{\alpha}z\cdot\Lambda(\de^{\bot})\frac{\varpi}{A^{2}(t)})d\alpha\\
&=-2\pi\int_{\T}(\re(-3\partial^{2}_{\alpha}z\cdot\partial^{3}_{\alpha}z)+2\im(\de)\cdot\im(\partial_{\alpha}z))\re(\partial_{\alpha}z\cdot\Lambda(\de^{\bot})\frac{\varpi}{A^{2}(t)})d\alpha\\
&+2\pi\int_{\T}(\im(-3\partial^{2}_{\alpha}z\cdot\partial^{3}_{\alpha}z)-2\re(\de)\cdot\im(\partial_{\alpha}z))\im(\partial_{\alpha}z\cdot\Lambda(\de^{\bot})\frac{\varpi}{A^{2}(t)})d\alpha\\
&\equiv Q^{'1}_{4}+Q^{'2}_{4}
\end{align*}
we have,
\begin{align*}
Q^{'1}_{4}&=-2\pi\int_{\T}\re(-3\partial^{2}_{\alpha}z\cdot\partial^{3}_{\alpha}z)\re(\partial_{\alpha}z\cdot\Lambda(\de^{\bot})\frac{\varpi}{A^{2}(t)})d\alpha\\
&-4\pi\int_{\T}\im(\de)\cdot\im(\partial_{\alpha}z)\re(\partial_{\alpha}z\cdot\Lambda(\de^{\bot})\frac{\varpi}{A^{2}(t)})d\alpha\\
&\equiv Q^{'11}_{4}+Q^{'12}_{4}.
\end{align*}
Clearly,
\begin{align*}
Q^{'11}_{4}\le\est{4}.
\end{align*}
Since
\begin{align*}
&\re(\partial_{\alpha}z\cdot\Lambda(\de^{\bot})\frac{\varpi}{A^{2}(t)})=\re(\partial_{\alpha}z\frac{\varpi}{A^{2}(t)})\cdot\re(\Lambda(\de^{\bot}))-\im(\partial_{\alpha}z\frac{\varpi}{A^{2}(t)})\cdot\im(\Lambda(\de^{\bot}))
\end{align*}
we take,
\begin{align*}
Q^{'12}_{4}&=-4\pi\int_{\T}\im(\de)\cdot\im(\partial_{\alpha}z)\re(\partial_{\alpha}z\frac{\varpi}{A^{2}(t)})\cdot\re(\Lambda(\de^{\bot}))d\alpha\\
&+4\pi\int_{\T}\im(\de)\cdot\im(\partial_{\alpha}z)\im(\partial_{\alpha}z\frac{\varpi}{A^{2}(t)})\cdot\im(\Lambda(\de^{\bot}))d\alpha\\
&\le\est{4}+k\norm{\im(\partial_{\alpha}z)\re(\partial_{\alpha}z\frac{\varpi}{A^{2}(t)})}_{H^{2}(S)}\norm{\Lambda^{\frac{1}{2}}\de}^{2}_{\s}\\
&+\est{4}+c\norm{\im(\partial_{\alpha}z)\im(\partial_{\alpha}z\frac{\varpi}{A^{2}(t)})}_{H^{2}(S)}\norm{\Lambda^{\frac{1}{2}}\de}^{2}_{\s}.
\end{align*}

For $Q^{'2}_{4}$
\begin{align*}
Q^{'2}_{4}&=2\pi\int_{\T}\im(-3\partial^{2}_{\alpha}z\cdot\partial^{3}_{\alpha}z)\im(\partial_{\alpha}z\cdot\Lambda(\de^{\bot})\frac{\varpi}{A^{2}(t)})d\alpha\\
&-4\pi\int_{\T}\re(\de)\cdot\im(\partial_{\alpha}z)\im(\partial_{\alpha}z\cdot\Lambda(\de^{\bot})\frac{\varpi}{A^{2}(t)})d\alpha\\
&\equiv Q^{'21}_{4}+Q^{'22}_{4}
\end{align*}

Clearly,
\begin{align*}
Q^{'21}_{4}\le\est{4}.
\end{align*}
Since,
\begin{align*}
&\im(\partial_{\alpha}z\cdot\Lambda(\de^{\bot})\frac{\varpi}{A^{2}(t)})=\re(\partial_{\alpha}z\frac{\varpi}{A^{2}(t)})\cdot\im(\Lambda(\de^{\bot}))+\im(\partial_{\alpha}z\frac{\varpi}{A^{2}(t)})\cdot\re(\Lambda(\de^{\bot}))
\end{align*}
we have,
\begin{align*}
Q^{'22}_{4}&=-4\pi\int_{\T}\re(\de)\cdot\im(\partial_{\alpha}z)\re(\partial_{\alpha}z\frac{\varpi}{A^{2}(t)})\cdot\im(\Lambda(\de^{\bot}))d\alpha\\
&-4\pi\int_{\T}\re(\de)\cdot\im(\partial_{\alpha}z)\im(\partial_{\alpha}z\frac{\varpi}{A^{2}(t)})\cdot\re(\Lambda(\de^{\bot}))d\alpha\\
&\le\est{4}+k\norm{\im(\partial_{\alpha}z)\re(\partial_{\alpha}z\frac{\varpi}{A^{2}(t)})}_{H^{2}(S)}\norm{\Lambda^{\frac{1}{2}}\de}^{2}_{\s}\\
&+\est{4}+c\norm{\im(\partial_{\alpha}z)\im(\partial_{\alpha}z\frac{\varpi}{A^{2}(t)})}_{H^{2}(S)}\norm{\Lambda^{\frac{1}{2}}\de}^{2}_{\s}.
\end{align*}

Using a similar method for the rest of non-integrable terms we obtain 

\begin{align*}
&J_{2}\le\est{4}\\
&+C(\norm{\im(\partial_{\alpha}z)\re(\partial_{\alpha}z\frac{\varpi}{A^{2}(t)})}_{H^{2}(S)}+\norm{\im(\partial_{\alpha}z)\im(\partial_{\alpha}z\frac{\varpi}{A^{2}(t)})}_{H^{2}(S)}+\norm{\im(c)}_{H^{2}(S)})\norm{\Lambda^{\frac{1}{2}}\de}^{2}_{\s}.
\end{align*}

In conclusion,
\begin{align*}
&I_{1}\le\est{4}+C[\norm{\im(\frac{\varpi}{A(t)})}_{H^{2}(S)}\norm{\im(\partial_{\alpha}z)\re(\partial_{\alpha}z\frac{\varpi}{A^{2}(t)})}_{H^{2}(S)}\\
&+\norm{\im(\partial_{\alpha}z)\im(\partial_{\alpha}z\frac{\varpi}{A^{2}(t)})}_{H^{2}(S)}+\norm{\im(c)}_{H^{2}(S)}-m(t)]\norm{\Lambda^{\frac{1}{2}}\de}^{2}_{\s}
\end{align*}
and therefore
\begin{align}
\label{energy}
&\frac{1}{2}\frac{d}{dt}\int_{\T}\abs{\de(\gamma)}^{2}d\alpha=I_{1}+I_{2}\nonumber\\
&\le\est{4}+C[\norm{\im(\frac{\varpi}{A(t)})}_{H^{2}(S)}\norm{\im(\partial_{\alpha}z)\re(\partial_{\alpha}z\frac{\varpi}{A^{2}(t)})}_{H^{2}(S)}\nonumber\\
&+\norm{\im(\partial_{\alpha}z)\im(\partial_{\alpha}z\frac{\varpi}{A^{2}(t)})}_{H^{2}(S)}+\norm{\im(c)}_{H^{2}(S)}-m(t)+2\lambda]\norm{\Lambda^{\frac{1}{2}}\de}^{2}_{\s}.
\end{align}

\section{The evolution of the minimum of $\sigma(\gamma,t)$}
\label{min}

Taking the divergence in Darcy's law we obtain
\begin{displaymath}
\Delta p=0.
\end{displaymath}
Since the pressure is zero on the interface and recalling that the flow is irrotational in the interior of the domain $\Omega$ by the Hopf's lemma we have
\begin{displaymath}
\sigma(\alpha,t)=-\frac{\partial p}{\partial\eta}|_{z(\alpha,t)}>0.
\end{displaymath}

In spite of this property, we need to get an a priori estimate for the evolution of the minimum of $\sigma$ in the strip $S$ in order to absorb the high order terms in (\ref{energy}).

 Recall that
\begin{equation}
\label{sig}
\sigma(\alpha,t)=\frac{\mu^{2}}{\kappa}BR(z,\varpi)(\alpha,t)\cdot\partial^{\bot}_{\alpha}z(\alpha,t)+g\rho^{2}\partial_{\alpha}\z(\alpha,t).
\end{equation}
\begin{lem}
Let $z(\gamma,t)$ be a solution of the system with $z(\gamma,t)\in\mathcal{C}([0,T];H^{4}(S))\cap\mathcal{C}^{1}([0,T];H^{3}(S))$, and 
$$m(t)=\min_{\gamma}\sigma(\gamma,t).$$
Then
$$m(t)\ge m(0)-\int^{t}_{0}\est{4}ds.$$
\end{lem}

\begin{proof}
We may consider $\gamma_{t}\in\C$ such that
$$m(t)=\min_{\gamma}\sigma(\gamma,t)=\sigma(\gamma_{t},t).$$

We may calculate the derivative of $m(t)$, to obtain
$$m'(t)=\sigma_{t}(\gamma_{t},t).$$

The identity (\ref{sig}) yields,
\begin{align*}
\sigma_{t}(\gamma,t)&=\frac{\mu^{2}}{\kappa}\partial_{t}BR(z,\varpi)(\gamma,t)\cdot\partial^{\bot}_{\alpha}z(\gamma,t)+i\lambda\frac{\mu^{2}}{\kappa}\partial_{\alpha}BR(z,\varpi)(\gamma,t)\cdot\partial^{\bot}_{\alpha}z(\gamma,t)\\
&+\frac{\mu^{2}}{\kappa}BR(z,\varpi)(\gamma,t)\cdot\partial^{\bot}_{\alpha}z_{t}(\gamma,t)+\frac{\mu^{2}}{\kappa}BR(z,\varpi)(\gamma,t)\cdot i\lambda\partial^{2}_{\alpha}z(\gamma,t)\\
&+g\rho^{2}\partial_{\alpha}\z_{t}(\gamma,t)+g\rho^{2}\partial^{2}_{\alpha}\z(\gamma,t)\equiv R_{1}+R_{2}+R_{3}+R_{4}+R_{5}+R_{6}.
\end{align*}

We can easily estimate,
\begin{align*}
&\abs{R_{2}}\le \lambda\frac{\mu^{2}}{\kappa}\norm{\partial_{\alpha}BR(z,\varpi)}_{L^{\infty}(S)}\norm{\partial_{\alpha}z}_{L^{\infty}(S)}\le\est{4},\\
&\abs{R_{4}}\le \lambda\frac{\mu^{2}}{\kappa}\norm{BR(z,\varpi)}_{L^{\infty}(S)}\norm{z}_{\mathcal{C}^{2}(S)}\le\est{4},\\
&\abs{R_{6}}\le g\rho^{2}\norm{z}_{\mathcal{C}^{2}(S)}\le\est{4},
\end{align*}
and we have,
\begin{displaymath}
\abs{R_{3}}+\abs{R_{5}}\le C(\norm{BR(z,\varpi)}_{L^{\infty}(S)}+1)\norm{\partial_{\alpha}z_{t}}_{L^{\infty}(S)}.
\end{displaymath}
Since $z_{t}(\gamma)=BR(z,\varpi)(\gamma)+c(\gamma)\partial_{\alpha}z(\gamma)$,
\begin{align*}
&\norm{\partial_{\alpha}z_{t}}_{L^{\infty}(S)}\le\norm{\partial_{\alpha}BR(z,\varpi)}_{L^{\infty}(S)}+\norm{\partial_{\alpha}c}_{L^{\infty}(S)}\norm{\partial_{\alpha}z}_{L^{\infty}(S)}+\norm{c}_{L^{\infty}(S)}\norm{\partial^{2}_{\alpha}z}_{L^{\infty}(S)}\\
&\le C\norm{\partial_{\alpha}BR(z,\varpi)}_{L^{\infty}(S)}(1+\norm{\F}^{\frac{1}{2}}_{L^{\infty}(S)}\norm{z}_{\mathcal{C}^{2}(S)})\le\est{4}.
\end{align*}
Then,
\begin{displaymath}
\abs{R_{3}+R_{5}}\le\est{4}.
\end{displaymath}

Recall that
\begin{displaymath}
BR(z,\varpi)(\gamma)=\frac{1}{2\pi}\int_{\T}\frac{\Delta z^{\bot}}{\abs{\Delta z}^{2}}\varpi(\gamma-\beta)d\beta,
\end{displaymath}
then
\begin{align*}
BR_{t}(z,\varpi)(\gamma)&=\frac{1}{2\pi}\int_{\T}\frac{\Delta z_{t}^{\bot}}{\abs{\Delta z}^{2}}\varpi(\gamma-\beta)d\beta-\frac{1}{\pi}\int_{\T}\frac{\Delta z^{\bot}(\Delta z\cdot\Delta z_{t})}{\abs{\Delta z}^{4}}\varpi(\gamma-\beta)d\beta\\
&+\frac{1}{2\pi}\int_{\T}\frac{\Delta z^{\bot}}{\abs{\Delta z}^{2}}\varpi_{t}(\gamma-\beta)d\beta\equiv J_{1}+J_{2}+J_{3}.
\end{align*}

We get
\begin{align*}
J_{1}&=\frac{1}{2\pi}\int_{\T}\Delta z_{t}^{\bot}\varpi(\gamma-\beta)(\frac{1}{\abs{\Delta z}^{2}}-\frac{1}{A(t)\beta^{2}})d\beta\\
&+\frac{1}{2\pi}\int_{\T}\frac{\Delta z_{t}^{\bot}}{A(t)\beta^{2}}\varpi(\gamma-\beta)d\beta\equiv K_{1}+K_{2}.
\end{align*}
Using that $\Delta z_{t}^{\bot}=\beta\int _
{0}^{1}\partial_{\alpha}z_{t}(\phi)ds,$
\begin{align*}
K_{1}&=\frac{1}{2\pi}\int_{\T}\varpi(\gamma-\beta)\beta\int _
{0}^{1}\partial_{\alpha}z_{t}^{\bot}(\phi)dsB(\gamma,\beta)d\beta\\
&\le C\norm{\F}^{\frac{3}{2}}_{L^{\infty}(S)}\norm{z}_{\mathcal{C}^{2}(S)}\norm{\varpi}_{L^{\infty}(S)}\norm{\partial_{\alpha}z_{t}}_{L^{\infty}(S)}\le\est{4}.
\end{align*}
Since
$$\partial^{2}_{\alpha}z_{t}=\partial^{2}BR(z,\varpi)+\partial^{2}_{\alpha}c\partial_{\alpha}z+2\partial_{\alpha}c\partial^{2}_{\alpha}z+c\partial^{3}z$$
and
\begin{align*}
&\norm{\partial^{2}_{\alpha}BR(z,\varpi)}_{L^{\infty}(S)}\le\est{4},\\
&\norm{\partial^{2}_{\alpha}c\partial_{\alpha}z}_{L^{\infty}(S)}=\norm{\frac{\partial^{2}_{\alpha}z}{\abs{\partial_{\alpha}z}^{2}}\cdot\partial_{\alpha}BR(z,\varpi)\partial_{\alpha}z}_{L^{\infty}(S)}+\norm{\frac{\partial_{\alpha}z}{\abs{\partial_{\alpha}z}^{2}}\cdot\partial^{2}_{\alpha}BR(z,\varpi)\partial_{\alpha}z}_{L^{\infty}(S)}\\
&\le C\norm{\F}^{\frac{1}{2}}_{L^{\infty}(S)}\norm{z}_{\mathcal{C}^{2}(S)}(\norm{\partial_{\alpha}BR(z,\varpi)}_{L^{\infty}(S)}+\norm{\partial^{2}_{\alpha}BR(z,\varpi)}_{L^{\infty}(S)})\\
&\le\est{4},\\
&2\norm{\partial_{\alpha}c\partial^{2}_{\alpha}z}_{L^{\infty}(S)}\le 4\norm{\F}^{\frac{1}{2}}_{L^{\infty}(S)}\norm{\partial_{\alpha}BR(z,\varpi)}_{L^{\infty}(S)}\norm{\partial^{2}_{\alpha}z}_{L^{\infty}(S)}\\
&\le\est{4},
\end{align*}

then
$$\norm{\partial^{2}_{\alpha}z_{t}}_{L^{\infty}(S)}\le\est{4}.$$
Thus,
\begin{align*}
K_{2}&=\frac{1}{2\pi}\int_{\T}\frac{\int _
{0}^{1}\partial_{\alpha}z_{t}^{\bot}(\phi)ds}{A(t)\beta}\varpi(\gamma-\beta)d\beta\\
&=\frac{1}{2\pi}\int_{\T}\frac{\int _
{0}^{1}\partial_{\alpha}z_{t}^{\bot}(\phi)-\partial_{\alpha}z_{t}^{\bot}(\gamma)ds}{A(t)\beta}\varpi(\gamma-\beta)d\beta+\frac{1}{2\pi}\int_{\T}\frac{\partial_{\alpha}z_{t}^{\bot}(\gamma)}{A(t)\beta}\varpi(\gamma-\beta)d\beta\\
&=\frac{1}{2\pi}\int_{\T}\frac{\int _
{0}^{1}\int _
{0}^{1}\partial^{2}_{\alpha}z_{t}^{\bot}(\psi)(s-1)dtds}{A(t)}\varpi(\gamma-\beta)d\beta+\frac{1}{2}\frac{\partial_{\alpha}z_{t}^{\bot}(\gamma)}{A(t)}H(\varpi)(\gamma)\\
&\le C\norm{\F}_{L^{\infty}(S)}\norm{\partial^{2}_{\alpha}z_{t}}_{L^{\infty}(S)}\norm{\varpi}_{L^{\infty}(S)}+K\norm{\F}_{L^{\infty}(S)}\norm{\partial_{\alpha}z_{t}}_{L^{\infty}(S)}\norm{\varpi}_{\mathcal{C}^{\delta}(S)}.
\end{align*}
Therefore,
$$J_{1}\le\est{4}.$$
In the same way, it is easy to see that
$$J_{2}\le\est{4}.$$

Finally, since
\begin{align*}
&\frac{\Delta z^{\bot}}{\abs{\Delta z}^{2}}-\frac{\partial_{\alpha}z^{\bot}(\gamma)}{A(t)\beta}=\frac{\beta^{2}\int _
{0}^{1}\int _
{0}^{1}\partial^{2}_{\alpha}z(\psi)(s-1)dtds}{\abs{\Delta z}^{2}}\\
&+\frac{\beta^{2}\partial_{\alpha}z(\gamma)\int _
{0}^{1}\int _
{0}^{1}\partial^{2}_{\alpha}z(\psi)(s-1)dtds\cdot\int _
{0}^{1}[\partial_{\alpha}z(\gamma)+\partial_{\alpha}z(\phi)]ds}{A(t)\abs{\Delta z}^{2}},
\end{align*}
\begin{align*}
J_{3}&=\frac{1}{2\pi}\int_{\T}(\frac{\Delta z^{\bot}}{\abs{\Delta z}^{2}}-\frac{\partial_{\alpha}z^{\bot}(\gamma)}{A(t)\beta})\varpi_{t}(\gamma-\beta)d\beta+\frac{1}{2\pi}\int_{\T}\frac{\partial_{\alpha}z^{\bot}(\gamma)}{A(t)\beta}\varpi_{t}(\gamma-\beta)d\beta\\
&\equiv K_{5}+K_{6}
\end{align*}
where
\begin{align*}
&K_{5}\le C\norm{\F}_{L^{\infty}(S)}\norm{z}_{\mathcal{C}^{2}(S)}\norm{\varpi_{t}}_{\s},\\
&K_{6}=\frac{1}{2}\frac{\partial_{\alpha}z^{\bot}(\gamma)}{A(t)}H(\varpi_{t})(\gamma)\le C\norm{\F}^{\frac{1}{2}}_{L^{\infty}(S)}\norm{\varpi_{t}}_{\mathcal{C}^{\delta}(S)}.
\end{align*}

In order to control $\norm{\varpi_{t}}_{\mathcal{C}^{\delta}(S)}$ we proceed as in section $9$ in \cite{hele}.

Therefore,
$$\abs{\sigma_{t}(\gamma,t)}\le\est{4}$$

given us,
$$m'(t)\ge-\est{4}$$
for almost every t. And a further integration yields
$$m(t)\ge m(0)-\int^{t}_{0}\est{4}ds.$$
\end{proof}


\section{Instant Analyticity}
\label{inst}
\begin{thm}
\label{thinst}
Let $z(\alpha,0)=z_{0}(\alpha)\in H^{4}$, $\F(z_{0})(\alpha,\beta)\in L^{\infty}$. Then there exists a solution of the Muskat problem $z(\alpha,t)$ defined for $0<t\le T$ that continues analytically into the strip $S(t)=\{\alpha\pm i\varsigma :\abs{\varsigma}<\lambda t\}$ for each $t$. Here, $\lambda$ and $T$ are determined by upper bounds of the $H^{4}$ norm and the arc-chord constant of the initial data and a positive lower bound of the $\sigma(\alpha,0)$. Moreover, for $0<t\le T$, the quantity
$$\sum_{\pm}\int_{\T}(\abs{z(\alpha\pm i\lambda t)-(\alpha\pm i\lambda t)}^{2}+\abs{\de(\alpha\pm i\lambda t)}^{2})d\alpha$$
is bounded by a constant determinate by upper bounds for the $H^{4}$ norm and the arc-chord constant of the initial data and a positive lower bound of $\sigma(\alpha,0)$.
\end{thm}
\begin{proof}
For all estimates in above sections we have finally
\begin{align*}
\frac{1}{2}\frac{d}{dt}\int_{\T}\abs{\de(\alpha\pm i\lambda t)}^{2}&\le\est{4}\\
&+(2\lambda+C\norm{f}(t)-m(t))\norm{\Lambda^{\frac{1}{2}}\de}^{2}_{\s}(t)
\end{align*}
where
$$\norm{f}(t)=\norm{\im(\frac{\varpi}{A(t)})}_{H^{2}(S)}+\norm{\im(\partial_{\alpha}z)\re(\partial_{\alpha}z\varpi)}_{H^{2}(S)}+\norm{\im(\partial_{\alpha}z)\im(\partial_{\alpha}z\varpi)}_{H^{2}(S)}+\norm{\im(c)}_{H^{2}(S)}.$$

Note that $\norm{f}(0)=0$. If $2\lambda -m(0)<0$ we will show that 
$$2\lambda+K\norm{f}(t)-m(t)<0$$
for short time. It yields
$$\frac{1}{2}\frac{d}{dt}\int_{\T}\abs{\de(\alpha\pm i\lambda t)}^{2}d\alpha\le\est{4}$$
as long as $2\lambda+K\norm{f}(t)-m(t)<0$. We proceed as in section $8$ in \cite{hele} to show that
$$\frac{d}{dt}\norm{\F}_{L^{\infty}(S)}\le\est{4}.$$
From the two inequalities above and (\ref{l2}) it is easy to obtain a priori energy estimates that depend upon the negativity of $2\lambda+K\norm{f}(t)-m(t)$.
We denote
\begin{align*}
\norm{z}_{RT}(t)\equiv\norm{\F}_{L^{\infty}(S)}^{2}(t)+\norm{z}_{\s}^{2}+\frac{1}{m(t)-2\lambda-C\norm{f}}.
\end{align*}

At this point it is easy to find
\begin{displaymath}
\norm{f}\le\est{4}
\end{displaymath}
and
\begin{align*}
\frac{d}{dt}(\frac{1}{m(t)-2\lambda-C\norm{f}})\le\frac{\est{4}}{(m(t)-2\lambda-C\norm{f})^{2}}
\end{align*}
then,
$$\frac{d}{dt}\norm{z}_{RT}(t)\le \exp C(\norm{z}_{RT}(t))$$
and therefore,
$$\norm{z}_{RT}\le -\log(\exp(-C\norm{z}_{RT}(0)-C^{2}t)).$$
Now we approximate the problem as follows,
\begin{displaymath}
\left\{ \begin{array}{ll}
z^{\epsilon}_{t}(\alpha,t)=BR(z^{\epsilon},\varpi^{\epsilon})(\alpha,t)+c^{\epsilon}(\alpha,t)\partial_{\alpha}z^{\epsilon}(\alpha,t)\\
z^{\epsilon}(\alpha,0)=\phi_{\epsilon}*z_{0}(\alpha)
\end{array} \right.
\end{displaymath}

where
\begin{align*}
c^{\epsilon}(\alpha,t)&=\frac{\alpha+\pi}{2\pi}\int_{\T}\frac{\partial_{\alpha}z^{\epsilon}(\alpha,t)}{\abs{\partial_{\alpha}z^{\epsilon}(\alpha,t)}^{2}}\cdot\partial_{\alpha}BR(z^{\epsilon},\varpi^{\epsilon})(\alpha,t)d\alpha\\
&-\int_{-\pi}^{\alpha}\frac{\partial_{\alpha}z^{\epsilon}(\beta,t)}{\abs{\partial_{\alpha}z^{\epsilon}(\beta,t)}^{2}}\cdot\partial_{\alpha}BR(z^{\epsilon},\varpi^{\epsilon})(\beta,t)d\beta,
\end{align*}
$$\varpi^{\epsilon}(\alpha,t)=-\phi_{\epsilon}*\phi_{\epsilon}*(2BR(z^{\epsilon},\varpi^{\epsilon})\cdot\partial_{\alpha}z^{\epsilon})(\alpha)-2\kappa\frac{\rho^{2}}{\mu^{2}}\phi_{\epsilon}*\phi_{\epsilon}*(\partial_{\alpha}\zz^{\epsilon})(\alpha)$$
where $\phi_{\epsilon}(\alpha)=\phi(\frac{\alpha}{\epsilon})/\epsilon$ for $\epsilon>0$ and $\phi$ the heat kernel.

 Picard's Theorem yields the existence of a solution $z^{\epsilon}(\alpha)$ in $\mathcal{C}([0,T^{\epsilon});H^{4})$ which is analytic in the whole space for $z_{0}$ satisfying the arc-chord condition and $\epsilon$ small enough. Using the same techniques we have devoted above we obtain a bound for $z^{\epsilon}(\alpha,t)$ in $H^{4}$ in the strip $S(t)$ for a small enough $T$ which is independent of $\epsilon$. We need arc-chord condition, $z_{0}\in H^{4}$ and $2\lambda-m(0)<0$. Then we pass to the limit.
\end{proof}
\section{Decay estimates on the strip of analyticity}
\label{decai}
\begin{thm}
\label{decaimiento}
Let $z(\alpha,0)=z^{0}(\alpha)$ be an analytic curve in the strip
$$S=\{\alpha+i\varsigma\in\C:\abs{\varsigma}<h(0)\},$$
with $h(0)>0$ and satisfying:
\begin{enumerate}
\item[*]The arc-chord condition, $\mathcal{F}(z^{0})(\alpha+i\varsigma,\beta)\in L^{\infty}(S\times\R)$
\item[*]The curve $z^{0}(\alpha)$ is real for real $\alpha$
\item[*]The functions $z_{1}^{0}(\alpha)-\alpha$ and $\zz^{0}(\alpha)$ are periodic with period $2\pi$
\item[*]The functions $\z^{0}(\alpha)-\alpha$ and $\zz^{0}(\alpha)$ belong to $H^{4}(S)$
\end{enumerate}
Then there exist a time T and a solution of the Muskat problem $z(\alpha,t)$ defined for $0<t\le T$
that continues analytically into some complex strip for each fixed $t\in\lbrack 0,T\rbrack$. Here $T$ is either
a small constant depending only on $\exp C(\norm{\mathcal{F}(z^{0})}_{L^{\infty}(S)}^{2}+\norm{z^{0}}^{2}_{\s})$.
\end{thm}

We will use the following:
\begin{lem}
\label{corol}
Let $\psi(\alpha\pm i\xi)=\sum_{k=-N}^{N}A_{k}(t)e^{ik\alpha}e^{\pm k\xi}$ and $h(t)>0$ be a decreasing function of $t$. Then
\begin{align*}
&\frac{\partial}{\partial t}\sum_{\pm}\int_{\T}\abs{\psi(\alpha\pm ih(t))}^{2}d\alpha\le\frac{h'(t)}{10}\sum_{\pm}\int_{\T}\Lambda\psi(\alpha\pm ih(t))\overline{\psi(\alpha\pm ih(t))}d\alpha\\
&-10h'(t)\int_{\T}\Lambda\psi(\alpha)\overline{\psi}(\alpha)d\alpha+2\re\sum_{\pm}\int_{\T}\psi_{t}(\alpha\pm ih(t))\overline{\psi(\alpha\pm ih(t))}d\alpha.
\end{align*}
\end{lem}

This lemma is a corollary of the lemma $4.2$ in \cite{raytay} and it allows us to prove the Theorem \ref{decaimiento}.
\begin{proof}[Proof of Theorem \ref{decaimiento}]

The norms $\norm{z}_{\s}$ and $\norm{z}_{H^{k}(S)}$ are defined as before using the new strip $S(t)$ defined by
\begin{displaymath}
S(t)=\{\alpha+i\varsigma\in\C:\abs{\varsigma}<h(t)\}
\end{displaymath}
where $h(t)$ is a positive decreasing function of $t$.

	We use the Galerkin approximation of equation (\ref{zt}), i.e.
	\begin{displaymath}
	z^{[N]}_{t}(\gamma,t)=\Pi_{N}[J[z^{[N]}]](\gamma,t)
	\end{displaymath}
where $\gamma\in\overline{S(t)}$, $\Pi_{N}$ will be defined below, and
\begin{displaymath}
J[z](\alpha,t)=BR(z,\varpi)(\alpha,t)+c(\alpha,t)\partial_{\alpha}z(\alpha,t).
\end{displaymath}

We impose the initial condition
\begin{displaymath}
z^{[N]}(\alpha,0)=z^{[N]}(\alpha).
\end{displaymath}

Here, for a large enough positive integer $N$, we define $z^{[N]}(\alpha,0)$ from $z^{0}(\alpha)$ by using the projection
\begin{displaymath}
\Pi_{N}:\sum_{-\infty}^{\infty}A_{k}e^{ik\alpha}\to\sum_{-N}^{N}A_{k}e^{ik\alpha}.
\end{displaymath}

We defined $z^{[N]}(\alpha)$ by stipulating that
\begin{displaymath}
\z^{[N]}(\alpha)-\alpha=\Pi_{N}[\z^{0}(\alpha)-\alpha]
\end{displaymath}
and
\begin{displaymath}
\zz^{[N]}(\alpha)=\Pi_{N}[\zz^{0}(\alpha)].
\end{displaymath}

For $N$ large enough, the functions $z^{[N]}(\alpha,0)$ satisfy the arc-chord and Rayleigh-Taylor conditions.

We shall consider the evolution of the most singular quantity
$$\sum_{\pm}\int_{\T}\abs{\de^{\lbrack N\rbrack}(\alpha\pm ih_{N}(t),t)}^{2}d\alpha$$
where $h_{N}(t)$ is a smooth positive decreasing function on t, with $h_{N}(0)=h(0)$, which will be given below. Also we denote
$$S_{N}(t)=\{\alpha+i\varsigma\in\C:\abs{\varsigma}<h_{N}(t)\}.$$

We will drop the dependency on $N$ from $z^{\lbrack N\rbrack}$ and $h_{N}(t)$ in our notation.
Using lemma above,
\begin{align*}
&\frac{d}{dt}\sum_{\pm}\int_{\T}\abs{\de_{j}(\alpha\pm ih(t))}^{2}d\alpha\le\frac{h'(t)}{10}\sum_{\pm}\int_{\T}\Lambda(\de_{j})(\alpha\pm ih(t))\overline{\de_{j}(\alpha\pm ih(t))}\\
&-10h'(t)\int_{\T}\Lambda(\de_{j})(\alpha)\overline{\de_{j}(\alpha)}d\alpha+2\sum_{\pm}\re\int_{\T}\Pi_{N}[\partial_{\alpha}^{4}J_{j}[z]](\alpha\pm ih(t))\overline{\de_{j}(\alpha\pm ih(t))}.\\
\end{align*} 
Since $\de_{j}(\alpha\pm ih(t))$ is a trigonometric polynomial in the range of $\Pi_{N}$
\begin{align*}
&2\sum_{\pm}\re\int_{\T}\Pi_{N}[\partial_{\alpha}^{4}J_{j}[z]](\alpha\pm ih(t))\overline{\de_{j}(\alpha\pm ih(t))}\\
&=2\sum_{\pm}\re\int_{\T}\partial_{\alpha}^{4}J_{j}[z](\alpha\pm ih(t))\Pi_{N}[\overline{\de_{j}(\alpha\pm ih(t))]}\\
&=2\sum_{\pm}\re\int_{\T}\partial_{\alpha}^{4}J_{j}[z](\alpha\pm ih(t))\overline{\de_{j}(\alpha\pm ih(t))}
\end{align*}
then,
\begin{align*}
&\frac{d}{dt}\sum_{\pm}\int_{\T}\abs{\de_{j}(\alpha\pm ih(t))}^{2}d\alpha\le\frac{h'(t)}{10}\sum_{\pm}\int_{\T}\Lambda(\de_{j})(\alpha\pm ih(t))\overline{\de_{j}(\alpha\pm ih(t))}\\
&-10h'(t)\int_{\T}\Lambda(\de_{j})(\alpha)\overline{\de_{j}(\alpha)}d\alpha+2\sum_{\pm}\re\int_{\T}\partial^{4}_{\alpha}BR(z_{j},\varpi)(\alpha\pm ih(t))\overline{\de_{j}(\alpha\pm ih(t))}\\
&+2\sum_{\pm}\re\int_{\T}\partial^{4}_{\alpha}(c(\alpha\pm ih(t))\partial_{\alpha}z_{j}(\alpha\pm ih(t)))\overline{\de_{j}(\alpha\pm ih(t))}\\
&\equiv M_{1}+M_{2}+M_{3}+M_{4}.
\end{align*} 
To estimate $M_{3}$ and $M_{4}$ we have to repeat the arguments in sections \ref{esti}, with the exception of the term $R_{20}+P_{7}$.

Following the same way, we will get that
\begin{align*}
M_{3}\le&\est{4}+C\norm{\im(\frac{\varpi}{A(t)})}_{H^{2}(S)}\norm{\Lambda^{\frac{1}{2}}\de}^{2}_{\s}\\
&-2\re\int_{\T}\frac{\sigma(\gamma)}{A(t)}\de_{j}(\gamma)\cdot\Lambda(\overline{\de_{j}})(\gamma)d\alpha
\end{align*}
where $\gamma=\alpha\pm ih(t)$.

In order to avoid problems we write,
\begin{displaymath}
\sigma(\gamma)=\sigma(\alpha)+h(t)g_{\pm}(\alpha)
\end{displaymath}
where $g_{\pm}=\frac{1}{h(t)}(\sigma(\gamma)-\sigma(\alpha)).$

Since $$\sigma(\alpha)=\frac{\mu^{2}}{\kappa}BR(z,\varpi)(\alpha)\cdot\partial^{\bot}_{\alpha}z(\alpha)+g\rho^{2}\partial_{\alpha}z_{1}(\alpha),$$
we can write,
\begin{displaymath}
g_{\pm}=\pm \frac{i\mu^{2}}{\kappa}\int _
{0}^{1}\partial_{\alpha}(BR(z,\varpi)\cdot\partial^{\bot}_{\alpha}z)(\gamma t+(t-1)\alpha)dt\pm ig\rho^{2}\int _
{0}^{1}\partial^{2}_{\alpha}z_{1}(\gamma t+(t-1)\alpha)dt
\end{displaymath}

then
\begin{displaymath}
\norm{g_{\pm}}_{H^{2}(S)}\le\est{4}.
\end{displaymath}
Thus, we get
\begin{align*}
&-2\re\int_{\T}\frac{\sigma(\gamma)}{A(t)}\de_{j}(\gamma)\Lambda(\overline{\de_{j}})(\gamma)d\alpha=-2\re\int_{\T}\frac{\sigma(\alpha)}{A(t)}\de_{j}(\gamma)\Lambda(\overline{\de_{j}})(\gamma)d\alpha\\
&-2h(t)\re\int_{\T}\frac{g_{\pm}(\alpha)}{A(t)}\de_{j}(\gamma)\Lambda(\overline{\de_{j}})(\gamma)d\alpha\equiv M_{3}^{1}+M_{3}^{2}.
\end{align*}
On the one hand, since $\re(\frac{\sigma}{A(t)})>0$ and $2g\Lambda(g)-\Lambda(g^{2})\ge 0$
\begin{align*}
M_{3}^{1}&=-2\int_{\T}\re(\frac{\sigma}{A(t)})(\re(\de_{j})\re(\Lambda(\de_{j}))+\im(\de_{j})\im(\Lambda(\de_{j})))d\alpha\\
&\le\norm{\Lambda(\frac{\sigma}{A(t)})}_{L^{\infty}(S)}\norm{\de}^{2}_{\s}\le\est{4}.
\end{align*}
On the other hand, like in the term $N_{5}$ in section \ref{estbr}
\begin{align*}
M_{3}^{2}&=-2h(t)\re\int_{\T}\Lambda^{\frac{1}{2}}(\frac{g_{\pm}(\alpha)}{A(t)}\de_{j}(\gamma))\Lambda^{\frac{1}{2}}(\overline{\de_{j}})(\gamma)d\alpha\\
&\le Ch(t)\norm{\frac{g_{\pm}}{A(t)}}_{H^{2}(S)}(\norm{\de}_{\s}+\norm{\Lambda^{\frac{1}{2}}\de}_{\s})\norm{\Lambda^{\frac{1}{2}}\de}_{\s}\\
&\le\est{4}+Ch(t)\est{4}\norm{\Lambda^{\frac{1}{2}}\de}^{2}_{\s}.
\end{align*}

For $M_{1}$,
\begin{align*}
&M_{1}\le\frac{h'(t)}{10}\norm{\Lambda^{\frac{1}{2}}\de}^{2}_{\s}
\end{align*}
and $M_{4}$,
\begin{align*}
M_{4}&\le\est{4}\\
&+C(\norm{\im(\partial_{\alpha}z)\re(\partial_{\alpha}z\frac{\varpi}{A^{2}(t)})}_{H^{2}(S)}+\norm{\im(\partial_{\alpha}z)\im(\partial_{\alpha}z\frac{\varpi}{A^{2}(t)})}_{H^{2}(S)}+\norm{\im(c)}_{H^{2}(S)})\norm{\Lambda^{\frac{1}{2}}\de}^{2}_{\s}\\
&\le\est{4}\norm{\Lambda^{\frac{1}{2}}\de}^{2}_{\s}.
\end{align*}
Then,
\begin{align*}
&\frac{d}{dt}\sum_{\pm}\int_{\T}\abs{\de_{j}(\alpha\pm ih(t))}^{2}d\alpha\le\est{4}\\
&-10h'(t)\int_{\T}\Lambda(\de_{j})(\alpha)\overline{\de_{j}(\alpha)}d\alpha\\
&+(\est{4}h(t)+\frac{h'(t)}{10}+\est{4})\norm{\Lambda^{\frac{1}{2}}\de}^{2}_{\s}.
\end{align*}

Choosing,
\begin{align*}
h(t)=&\exp(-10\int_{0}^{t}G(r)dr)[\int_{0}^{t}-10G(r)\exp(10\int^{r}_{0}G(s)ds)dr+h(0)]
\end{align*}
where $G(t)=\est{4}(t)$, we eliminate the most dangerous term. The other term in the expression above,
\begin{displaymath}
\int_{\T}\Lambda(\de_{j})(\alpha)\de_{j}(\alpha)d\alpha\le\frac{C}{h(t)}\sum_{\pm}\int_{\T}\abs{\de_{j}}^{2}d\alpha
\end{displaymath}
as one sees by examining the Fourier expansion of $\de_{j}(\alpha,t)$.
Thus,
\begin{align*}
&\abs{-10h'(t)\int_{\T}\Lambda(\de_{j})(\alpha)\de_{j}(\alpha)d\alpha}\le C\frac{\abs{h'(t)}}{h(t)}(\norm{z}^{2}_{H^{4}(S)}+\norm{\F}^{2}_{L^{\infty}(S)})\\
&\le\est{4}.
\end{align*}
And we obtain finally,
\begin{displaymath}
\frac{d}{dt}\sum_{\pm}\int_{\T}\abs{\de(\alpha\pm ih(t))}^{2}d\alpha\le\est{4}
\end{displaymath}

Recovering the dependency on $N$ in our notation we have that
\begin{equation}
\label{eqn}
\frac{d}{dt}\sum_{\pm}\int_{\T}\abs{\de^{\lbrack N\rbrack}(\alpha\pm ih(t))}^{2}d\alpha\le\exp C(\norm{\mathcal{F}(z^{\lbrack N\rbrack})}^{2}_{L^{\infty}(S_{N})}+\norm{z^{\lbrack N\rbrack}}^{2}_{L^{2}(S_{N})})
\end{equation}

This estimate is true wherever $t\in\lbrack 0,T_{N}\rbrack$, where $T_{N}$ is the maximal time of existence of the solution $z^{\lbrack N\rbrack}$. In addition inequality (\ref{eqn}) shows that we can extend these solutions in $H^{4}(S)$ up to a small enough time $T$ independent of $N$ and dependent on the initial data.
\end{proof}
\section{Non-splat singularity}
\label{nosplat}

	As we have said in the introduction, it is necessary to consider a transformed Muskat problem and we need to prove instant analyticity and decay estimates in $\tilde{\Omega}$. We will prove that the energy estimates of the Theorems \ref{thinst} and \ref{decaimiento} holds in $\tilde{\Omega}$ for solutions $\tilde{z}$ of equations:
\begin{equation}
\label{ztil}
\tilde{z}_{t}(\alpha,t)=Q^{2}(\alpha,t)BR(\tilde{z},\tilde{\varpi})(\alpha,t)+\tilde{c}(\alpha)\partial_{\alpha}\tilde{z}(\alpha,t)
\end{equation}
where
\begin{equation}
\label{q}
Q^{2}(\alpha,t)=\abs{\frac{dP}{dw}(z(\alpha,t))}^{2}=\abs{\frac{dP}{dw}(P^{-1}(\tilde{z}(\alpha,t)))}^{2},
\end{equation}
\begin{equation}
\label{wtil}
\tilde{\varpi}(\alpha,t)=-2BR(\tilde{z},\tilde{\varpi})(\alpha,t)\cdot\partial_{\alpha}\tilde{z}(\alpha,t)-2\frac{\rho^{2}}{\mu^{2}}\partial_{\alpha}(P_{2}^{-1}(\tilde{z}(\alpha,t)))
\end{equation}
and
\begin{align}
\label{ctil}
&\tilde{c}(\alpha,t)=\frac{\alpha+\pi}{2\pi}\int_{\T}{\frac{\partial_{\beta}\tilde{z}(\beta,t)}{\abs{\partial_{\beta}\tilde{z}(\beta,t)}^{2}}\cdot\partial_{\beta}BR(\tilde{z},\tilde{\varpi})(\beta,t)d\beta}\nonumber\\
&-\int_{-\pi}^{\alpha}\frac{\partial_{\beta}\tilde{z}(\beta,t)}{\abs{\partial_{\beta}\tilde{z}(\beta,t)}^{2}}\cdot\partial_{\beta}BR(\tilde{z},\tilde{\varpi})(\beta,t)d\beta
\end{align}

	with  $\tilde{z}\in\mathcal{C}([0,T],H^{k})$ for $k\ge 4$, 
	
\subsection{Instant analyticity in $\tilde{\Omega}$ domain}
	
	We define
	
	\begin{align*}
	&q^{0}=(0,0),\quad q^{1}=(\frac{1}{\sqrt{2}},\frac{1}{\sqrt{2}}),\quad q^{2}=(-\frac{1}{\sqrt{2}},\frac{1}{\sqrt{2}}),\\
	&q^{3}=(-\frac{1}{\sqrt{2}},-\frac{1}{\sqrt{2}}),\quad q^{4}=(\frac{1}{\sqrt{2}},-\frac{1}{\sqrt{2}})
	\end{align*}

	which are the singular points of the $P^{-1}$ conformal map. We set $z(\alpha,t)$ to hold $\tilde{z}(\alpha,t)\neq q^{l}$ for $l=0,1,2,3,4$. In order to get this we fix $\overline{\Omega(0)}$ so that $\frac{dP}{dw}(w)\neq 0$ for any $w\in\overline{\Omega(0)}$ without loss of generality.
	
	We define the energy
	\begin{displaymath}
	\norm{\tilde{z}}_{RT}\equiv\norm{\tilde{z}}^{2}_{H^{k}(S)}+\norm{\mathcal{F}(\tilde{z})}^{2}_{L^{\infty}(S)}+\frac{1}{m(Q^{2}\tilde{\sigma})(t)-2\lambda-\norm{g}(t)}+\sum_{l=0}^{4}\frac{1}{m(q^{l})(t)}
	\end{displaymath}
	where
	\begin{align*}
\norm{g}(t)=&C(\norm{\im(\partial_{\alpha}\tilde{z})\re(\partial_{\alpha}\tilde{z}\frac{\tilde{\varpi}Q^{2}}{A^{2}(t)})}_{H^{2}(S)}+\norm{\im(\partial_{\alpha}\tilde{z})\im(\partial_{\alpha}\tilde{z}\frac{\tilde{\varpi}Q^{2}}{A^{2}(t)})}_{H^{2}(S)}\\
&+\norm{\im(\frac{\tilde{\varpi}Q^{2}}{A(t)})}_{H^{2}(S)}+\norm{\im(\tilde{c})}_{H^{2}(S)})
\end{align*}
	and
	\begin{displaymath}
	m(Q^{2}\tilde{\sigma})(t)=\min_{\alpha}Q^{2}(\alpha,t)\sigma(\alpha,t),\quad m(q^{l})(t)=\min_{\alpha}\abs{\tilde{z}(\alpha,t)-q^{l}}.
	\end{displaymath}
	\begin{thm}
	\label{prop}
	Let $\tilde{z}(\alpha,t)$ be a solution of (\ref{ztil}-\ref{ctil}). Then, the following estimate holds:
	\begin{displaymath}
	\frac{d}{dt}\norm{\tilde{z}}_{RT}\le\exp C(\norm{\tilde{z}}_{RT})
	\end{displaymath}
	for C constant.
	\end{thm}
	\begin{nota}
	We will show the proof for $k=4$, being the rest of the cases analogous.
	\end{nota}
\begin{proof}
 We have to estimate 
 \begin{displaymath}
 \frac{d}{dt}\norm{\detil}^{2}_{\s}.
 \end{displaymath}
 We quote \cite{finite} for dealing with the $Q^{2}$ term. This factor do not introduce a high order term 
 \begin{displaymath}
 \norm{Q^{2}}_{H^{k}(S)}\le\exp C(\norm{\tilde{z}}_{RT}).
 \end{displaymath}
 
Then we have to repeat all estimates in section \ref{esti}, in which $Q^{2}$ is involved. We will show below how to deal with them.

We find 
\begin{align*}
&\frac{d}{dt}\norm{\detil}^{2}_{\s}\le\esttil{4}+2\lambda\norm{\Lambda^{\frac{1}{2}}\detil{}}^{2}_{\s}\\
&+J_{1}+J_{2}
\end{align*}
where
\begin{align*}
&J_{1}=\re{\int_{\T}\overline{\detil(\gamma)}}\cdot\partial_{\alpha}^{4}(Q^{2}(\gamma)BR(\tilde{z},\tilde{\varpi})(\gamma))d\alpha,\\
&J_{2}=\re{\int_{\T}\overline{\detil(\gamma)}}\cdot\partial^{4}_{\alpha}(\tilde{c}(\gamma)\partial_{\alpha}\tilde{z}(\gamma))d\alpha.
\end{align*}

We get $J_{1}\le\esttil{4}+I_{7}$ where
\begin{displaymath}
I_{7}=\re\int_{\T}\overline{\detil(\gamma)}\cdot Q^{2}(\gamma)\partial^{4}_{\alpha}BR(\tilde{z},\tilde{\varpi})(\gamma)d\alpha.
\end{displaymath}
As in \ref{estbr} we split $I_{7}=\tilde{I_{3}}+\tilde{I_{4}}+\tilde{I_{5}}+\tilde{I_{6}}+\tilde{I_{7}}$
in the same way we have
\begin{displaymath}
\tilde{I_{3}}+\tilde{I_{4}}+\tilde{I_{5}}+\tilde{I_{6}}\le\esttil{4}+C\norm{\im(\frac{\tilde{\varpi}}{A(t)}Q^{2})}_{H^{2}(S)}\norm{\Lambda^{\frac{1}{2}}\detil{}}^{2}_{\s}
\end{displaymath}
and $\tilde{I_{7}}\le\esttil{4}+\tilde{K}_{9}$ being
\begin{displaymath}
\tilde{K}_{9}=\frac{1}{2}\re\int_{\T}\overline{\detil(\gamma)}\cdot\frac{\partial^{\bot}_{\alpha}\tilde{z}(\gamma)}{\abs{\partial_{\alpha}\tilde{z}}^{2}}H(\partial^{4}_{\alpha}\tilde{\varpi})(\gamma)Q^{2}(\gamma)d\alpha.
\end{displaymath}
Identity $H(\partial_{\alpha})=\Lambda$ allows us to rewrite $\tilde{K}_{9}$ as follows
\begin{displaymath}
\tilde{K}_{8}=\frac{1}{2}\re\int_{\T}\Lambda(\overline{\detil{}}\cdot\frac{\partial^{\bot}_{\alpha}\tilde{z}}{\abs{\partial_{\alpha}\tilde{z}}^{2}}Q^{2})(\gamma)\partial^{3}_{\alpha}\tilde{\varpi}(\gamma)d\alpha.
\end{displaymath}
Using the formula (\ref{wtil}), we decompose $\tilde{K}_{9}=\tilde{P}_{7}+\tilde{P}_{8}$
\begin{align*}
&\tilde{P}_{7}=-\kappa g\frac{\rho^{2}}{\mu^{2}}\re\int_{\T}\frac{\Lambda(\overline{\detil{}}\cdot Q^{2}\partial^{\bot}_{\alpha}\tilde{z})(\gamma)}{A(t)}\partial^{4}_{\alpha}(P_{2}^{-1}(\tilde{z}(\gamma)))d\alpha\\
&\tilde{P}_{8}=-\frac{1}{2}\re\int_{\T}\frac{\Lambda(\overline{\detil{}}\cdot Q^{2}\partial^{\bot}_{\alpha}\tilde{z})(\gamma)}{A(t)}\partial^{3}_{\alpha}\tilde{T}(\tilde{\varpi})(\gamma)d\alpha
\end{align*}
where $\tilde{T}(\tilde{\varpi})=-2BR(\tilde{z},\tilde{\varpi})\cdot\partial_{\alpha}\tilde{z}.$

The term $\tilde{P}_{8}$ can be estimate as the term $P_{8}$ in subsection \ref{look}. An analogous approach provides
\begin{align}
\label{p8}
&\tilde{P}_{8}\le\esttil{4}\nonumber\\
&-\re\int_{\T}\frac{Q^{2}(\gamma)BR(\tilde{z},\tilde{\varpi})(\gamma)\cdot\partial^{\bot}_{\alpha}\tilde{z}(\gamma)}{A(t)}\detil(\gamma)\cdot\Lambda^{\frac{1}{2}}(\overline{\detil{}})(\gamma)d\alpha.
\end{align}

For $\tilde{P}_{7}$ we consider the most singular terms: $\tilde{P}_{7}\le\esttil{4}+\tilde{P}_{7}^{1}$ where
\begin{displaymath}
\tilde{P}_{7}^{1}=-\kappa g\frac{\rho^{2}}{\mu^{2}}\re\int_{\T}\frac{\Lambda(\overline{\detil{}}\cdot Q^{2}\partial^{\bot}_{\alpha}\tilde{z})(\gamma)}{A(t)}\nabla P_{2}^{-1}(\tilde{z}(\gamma))\cdot\detil{(\gamma)}d\alpha.
\end{displaymath}

Then we split $\tilde{P}_{7}^{1}=\tilde{P}_{7}^{11}+\tilde{P}_{7}^{12}+\tilde{P}_{7}^{13}+\tilde{P}_{7}^{14}$ by writing the component of the curve:
\begin{align*}
&\tilde{P}_{7}^{11}=\kappa g\frac{\rho^{2}}{\mu^{2}}\re\int_{\T}\frac{\Lambda(\overline{\detil_{1}} Q^{2}\partial_{\alpha}\tilde{\zz})(\gamma)}{A(t)}\partial_{\tilde{x_{1}}} P_{2}^{-1}(\tilde{z}(\gamma))\detil_{1}{(\gamma)}d\alpha,\\
&\tilde{P}_{7}^{12}=\kappa g\frac{\rho^{2}}{\mu^{2}}\re\int_{\T}\frac{\Lambda(\overline{\detil_{1}} Q^{2}\partial_{\alpha}\tilde{\zz})(\gamma)}{A(t)}\partial_{\tilde{x_{2}}} P_{2}^{-1}(\tilde{z}(\gamma))\detil_{2}{(\gamma)}d\alpha,\\
&\tilde{P}_{7}^{13}=-\kappa g\frac{\rho^{2}}{\mu^{2}}\re\int_{\T}\frac{\Lambda(\overline{\detil_{2}} Q^{2}\partial_{\alpha}\tilde{\z})(\gamma)}{A(t)}\partial_{\tilde{x_{1}}} P_{2}^{-1}(\tilde{z}(\gamma))\detil_{1}{(\gamma)}d\alpha,\\
&\tilde{P}_{7}^{14}=-\kappa g\frac{\rho^{2}}{\mu^{2}}\re\int_{\T}\frac{\Lambda(\overline{\detil_{2}} Q^{2}\partial_{\alpha}\tilde{\z})(\gamma)}{A(t)}\partial_{\tilde{x_{2}}} P_{2}^{-1}(\tilde{z}(\gamma))\detil_{2}{(\gamma)}d\alpha.\\
\end{align*}
The commutator estimate yields
\begin{align}
\label{p1}
&\tilde{P}_{7}^{11}\le\esttil{4}\nonumber\\
&+\kappa g\frac{\rho^{2}}{\mu^{2}}\re\int_{\T}\frac{Q^{2}(\gamma)\partial_{\tilde{x_{1}}} P_{2}^{-1}(\tilde{z}(\gamma))}{A(t)}\partial_{\alpha}\tilde{\zz}(\gamma)\detil_{1}{(\gamma)}\Lambda(\overline{\detil_{1}})(\gamma)d\alpha,\\
&\tilde{P}_{7}^{12}\le\esttil{4}\nonumber\\
&+\kappa g\frac{\rho^{2}}{\mu^{2}}\re\int_{\T}\frac{Q^{2}(\gamma)\partial_{\tilde{x_{2}}} P_{2}^{-1}(\tilde{z}(\gamma))}{A(t)}\partial_{\alpha}\tilde{\zz}(\gamma)\detil_{2}{(\gamma)}\Lambda(\overline{\detil_{1}})(\gamma)d\alpha,\nonumber\\
&\tilde{P}_{7}^{13}\le\esttil{4}\nonumber\\
&-\kappa g\frac{\rho^{2}}{\mu^{2}}\re\int_{\T}\frac{Q^{2}(\gamma)\partial_{\tilde{x_{1}}} P_{2}^{-1}(\tilde{z}(\gamma))}{A(t)}\partial_{\alpha}\tilde{\z}(\gamma)\detil_{1}{(\gamma)}\Lambda(\overline{\detil_{2}})(\gamma)d\alpha,\nonumber\\
&\tilde{P}_{7}^{14}\le\esttil{4}\nonumber\\
\label{p4}
&-\kappa g\frac{\rho^{2}}{\mu^{2}}\re\int_{\T}\frac{Q^{2}(\gamma)\partial_{\tilde{x_{2}}} P_{2}^{-1}(\tilde{z}(\gamma))}{A(t)}\partial_{\alpha}\tilde{\z}(\gamma)\detil_{2}{(\gamma)}\Lambda(\overline{\detil_{2}})(\gamma)d\alpha.
\end{align}

Using that
\begin{displaymath}
\partial_{\alpha}\tilde{\zz}\detil_{2}=-3\partial^{2}_{\alpha}\tilde{z}\cdot\partial^{3}_{\alpha}\tilde{z}-\partial_{\alpha}\tilde{\z}\detil_{1},
\end{displaymath}

we get
\begin{align}
\label{p2}
&\tilde{P}_{7}^{12}\le\esttil{4}\nonumber\\
&-\kappa g\frac{\rho^{2}}{\mu^{2}}\re\int_{\T}\frac{Q^{2}(\gamma)\partial_{\tilde{x_{2}}} P_{2}^{-1}(\tilde{z}(\gamma))}{A(t)}\partial_{\alpha}\tilde{\z}(\gamma)\detil_{1}{(\gamma)}\Lambda(\overline{\detil_{1}})(\gamma)d\alpha,\\
&\tilde{P}_{7}^{13}\le\esttil{4}\nonumber\\
\label{p3}
&+\kappa g\frac{\rho^{2}}{\mu^{2}}\re\int_{\T}\frac{Q^{2}(\gamma)\partial_{\tilde{x_{1}}} P_{2}^{-1}(\tilde{z}(\gamma))}{A(t)}\partial_{\alpha}\tilde{\zz}(\gamma)\detil_{2}{(\gamma)}\Lambda(\overline{\detil_{2}})(\gamma)d\alpha.
\end{align}

Adding the inequalities (\ref{p1}),(\ref{p2}),(\ref{p3}) and (\ref{p4}) it is easy to check 
\begin{align*}
&\tilde{P}_{7}\le\esttil{4}\\
&-\kappa g\frac{\rho^{2}}{\mu^{2}}\re\int_{\T}\frac{Q^{2}(\gamma)\nabla P_{2}^{-1}(\tilde{z}(\gamma))}{A(t)}\cdot\partial^{\bot}_{\alpha}\tilde{z}(\gamma)\detil{(\gamma)}\cdot\Lambda(\detil)(\gamma)d\alpha.
\end{align*}
Above inequality together with (\ref{p8}) let us obtain
\begin{align*}
&\tilde{K}_{9}\le\esttil{4}\\
&-\re\int_{\T}\frac{Q^{2}(\gamma)\tilde{\sigma}(\gamma)}{A(t)}\detil{(\gamma)}\cdot\Lambda(\detil)(\gamma)d\alpha
\end{align*}
with $\tilde{\sigma}$ given in (\ref{rttild}).

Considering $m(Q^{2}\tilde{\sigma})(t)$ and the pointwise inequality $2f\Lambda(f)\ge\Lambda(f^{2})$ we check
\begin{align*}
&\tilde{I}_{7}\le\esttil{4}-m(Q^{2}\tilde{\sigma})(t)\norm{\detil}^{2}_{\s}.
\end{align*}

For $J_{2}$ it is easy to deal with $\partial^{4}_{\alpha}\tilde{c}$ in the same way as in section \ref{cest}. The analogous approach provides
\begin{align*}
&J_{2}\le\esttil{4}\\
&+C(\norm{\im(\partial_{\alpha}\tilde{z})\re(\partial_{\alpha}\tilde{z}\frac{\tilde{\varpi}Q^{2}}{A^{2}(t)})}_{H^{2}(S)}+\norm{\im(\partial_{\alpha}\tilde{z})\im(\partial_{\alpha}\tilde{z}\frac{\tilde{\varpi}Q^{2}}{A^{2}(t)})}_{H^{2}(S)}\\
&+\norm{\im(\tilde{c})}_{H^{2}(S)})\norm{\Lambda^{\frac{1}{2}}\detil}^{2}_{\s}.
\end{align*}

Finally we obtain,
\begin{align*}
&\frac{d}{dt}\norm{\detil}^{2}_{\s}\le\esttil{4}\\
&+(2\lambda+\norm{g}-m(Q^{2}\tilde{\sigma}))\norm{\Lambda^{\frac{1}{2}}\detil}^{2}_{\s}.
\end{align*}

Bearing in mind the singular points of the $P^{-1}$ together with the estimation for $m(Q^{2}\tilde{\sigma})(t)$, which we can obtain in analogous way as in section \ref{min}, we have the desired estimate.
\end{proof}




	

\subsection{Decay of the strip of analyticity in the $\tilde{\Omega}$ domain}
\begin{thm}
\label{decaitil}
Let $\tilde{z}(\alpha,0)=\tilde{z}^{0}(\alpha)$ be an analytic curve in the strip
$$S=\{\alpha+i\varsigma\in\C:\abs{\varsigma}<h(0)\},$$
with $h(0)>0$ and satisfying:
\begin{enumerate}
\item[*]The arc-chord condition, $\mathcal{F}(\tilde{z}^{0})(\alpha+i\varsigma,\beta)\in L^{\infty}(S\times\R)$
\item[*]The curve $\tilde{z}^{0}(\alpha)$ is real for real $\alpha$
\item[*]The functions $\tilde{z}_{1}^{0}(\alpha)-\alpha$ and $\tilde{\zz}^{0}(\alpha)$ are periodic with period $2\pi$
\item[*]The functions $\tilde{\z}^{0}(\alpha)-\alpha$ and $\tilde{\zz}^{0}(\alpha)$ belong to $H^{4}(S)$
\end{enumerate}
Then there exist a time T and a solution of the Muskat problem in $\tilde{\Omega}$, $\tilde{z}(\alpha,t)$ defined for $0<t\le T$
that continues analytically into some complex strip for each fixed $t\in\lbrack 0,T\rbrack$. Here $T$ is either
a small constant depending only on $\exp C(\norm{\mathcal{F}(\tilde{z}^{0})}_{L^{\infty}(S)}^{2}+\norm{\tilde{z}^{0}}^{2}_{\s})$.
\end{thm}
\begin{proof}
Here we proceed in the same way that in the proof of the Theorem \ref{decaimiento}.

After we use the Galerkin approximation, by Lemma \ref{corol} we get
\begin{align*}
&\frac{d}{dt}\sum_{\pm}\int_{\T}\abs{\detil_{j}(\alpha\pm ih(t))}^{2}d\alpha\le\frac{h'(t)}{10}\sum_{\pm}\int_{\T}\Lambda(\detil_{j})(\alpha\pm ih(t))\overline{\detil_{j}(\alpha\pm ih(t))}\\
&-10h'(t)\int_{\T}\Lambda(\detil_{j})(\alpha)\overline{\detil_{j}(\alpha)}d\alpha+2\sum_{\pm}\re\int_{\T}\partial^{4}_{\alpha}(Q^{2}BR(\tilde{z}_{j},\tilde{\varpi})(\alpha\pm ih(t))\overline{\de_{j}(\alpha\pm ih(t))}\\
&+2\sum_{\pm}\re\int_{\T}\partial^{4}_{\alpha}(\tilde{c}(\alpha\pm ih(t))\partial_{\alpha}\tilde{z}_{j}(\alpha\pm ih(t)))\overline{\detil_{j}(\alpha\pm ih(t))}.
\end{align*} 

We write,
\begin{displaymath}
Q^{2}(\gamma)\tilde{\sigma}(\gamma)=Q^{2}(\alpha)\tilde{\sigma}(\alpha)+h(t)\tilde{g_{\pm}}(\alpha)
\end{displaymath}
and we have
\begin{align*}
&\frac{d}{dt}\sum_{\pm}\int_{\T}\abs{\detil_{j}(\alpha\pm ih(t))}^{2}d\alpha\le\esttil{4}\\
&-10h'(t)\int_{\T}\Lambda(\detil_{j})(\alpha)\overline{\detil_{j}(\alpha)}d\alpha\\
&+(\esttil{4}h(t)+\frac{h'(t)}{10}+\esttil{4})\norm{\Lambda^{\frac{1}{2}}\detil}^{2}_{\s}.
\end{align*}
Choosing,
\begin{align}
\label{ht}
h(t)=&\exp(-10\int_{0}^{t}G(r)dr)[\int_{0}^{t}-10G(r)\exp(10\int^{r}_{0}G(s)ds)dr+h(0)]
\end{align}
where $G(t)=\est{4}(t)$ we get the desired estimation.
\end{proof}
\subsection{Proof of Theorem \ref{main}}






Let $z_{0}(\alpha)\in H^{4}$, from Theorem \ref{inst} there exists a local solution $z$ that becomes real-analytic in the complex strip S(t).

Suppose that there exists a time $T$ where we have a splat singularity, i.e., the smooth interface collapses along an arc at time $T$.

From Theorem \ref{decaimiento}, our strip of analyticity is nonzero as long as the regularity of the curve and the arc-chord condition do not fail. But at splat time $T$, the arc-chord condition blows-up, and we cannot guarantee analyticity at that time.

At this point, we transform the system to the tilde domain $\tilde{\Omega}$.

As long as the regularity of the curve and the arc-chord condition do not fail, from Theorem \ref{prop} we have
\begin{displaymath}
	\frac{d}{dt}\norm{\tilde{z}}_{RT}\le\exp C(\norm{\tilde{z}}_{RT})
	\end{displaymath}
where the constant C only depends on the initial data and 
\begin{displaymath}
	\norm{\tilde{z}}_{RT}\equiv\norm{\tilde{z}}^{2}_{H^{k}(S)}+\norm{\mathcal{F}(\tilde{z})}^{2}_{L^{\infty}(S)}+\frac{1}{m(Q^{2}\tilde{\sigma})(t)-2\lambda-\norm{g}(t)}+\sum_{l=0}^{4}\frac{1}{m(q^{l})(t)}.
\end{displaymath}
	
	Hence, we can conclude that our transformed curve $\tilde{z}$ is real-analytic into the strip $S(t)$. From the proof of Theorem \ref{decaitil}, this complex strip decays exponentially until a time that depends on the regularity of the curve and the arc-chord condition too [see equation (\ref{ht})]. 

Since in $\tilde{\Omega}$ the arc-chord condition and the regularity of the curve are bounded, the strip of analyticity is nonzero and therefore we can guarantee the analyticity at time $T$.

Thus, applying $P^{-1}$, we have that the analytic curve self-intersects along an arc, therefore we get a contradiction and hence Theorem \ref{main} is proved.
	
\bibliography{mibiblioteca}

\begin{thebibliography}{10}

\bibitem{Breakdown}
{\'A}ngel Castro, Diego C{\'o}rdoba, Charles Fefferman, and Francisco Gancedo.
\newblock Breakdown of smoothness for the muskat problem.
\newblock {\em Arch. Rational. Mech. Anal.}, 208(3):805--909, 2013.

\bibitem{splash}
{\'A}ngel Castro, Diego C{\'o}rdoba, Charles Fefferman, and Francisco Gancedo.
\newblock Splash singularities for the one-phase {M}uskat problem in stable
  regimes.
\newblock {\em arXiv:1311.7653v1}, 2013.

\bibitem{finite}
Angel Castro, Diego C{\'o}rdoba, Charles Fefferman, Francisco Gancedo, and
  Javier G{\'o}mez-Serrano.
\newblock Finite time singularities for the free boundary incompressible
  {E}uler equations.
\newblock {\em Ann. of Math. (2)}, 178(3):1061--1134, 2013.

\bibitem{raytay}
{\'A}ngel Castro, Diego C{\'o}rdoba, Charles Fefferman, Francisco Gancedo, and
  Mar{\'{\i}}a L{\'o}pez-Fern{\'a}ndez.
\newblock Rayleigh-{T}aylor breakdown for the {M}uskat problem with
  applications to water waves.
\newblock {\em Ann. of Math. (2)}, 175(2):909--948, 2012.

\bibitem{graner}
C.H.Arthur Cheng, Rafael Granero-Belinch{\'o}n, and Steve Shkoller.
\newblock Well-posedness of the muskat problem with ${H}^{2}$ initial data.
\newblock {\em arXiv:14127737v1}, 2014.

\bibitem{global}
Peter Constantin, Diego C{\'o}rdoba, Francisco Gancedo, and Robert~M. Strain.
\newblock On the global existence for the {M}uskat problem.
\newblock {\em J. Eur. Math. Soc. (JEMS)}, 15(1):201--227, 2013.

\bibitem{point}
Antonio C{\'o}rdoba and Diego C{\'o}rdoba.
\newblock A pointwise estimate for fractionary derivatives with applications to
  partial differential equations.
\newblock {\em Proc. Natl. Acad. Sci. USA}, 100(26):15316--15317, 2003.

\bibitem{hele}
Antonio C{\'o}rdoba, Diego C{\'o}rdoba, and Francisco Gancedo.
\newblock Interface evolution: the {H}ele-{S}haw and {M}uskat problems.
\newblock {\em Ann. of Math. (2)}, 173(1):477--542, 2011.

\bibitem{porous}
Antonio C{\'o}rdoba, Diego C{\'o}rdoba, and Francisco Gancedo.
\newblock Porous media: the {M}uskat problem in three dimensions.
\newblock {\em Anal. PDE}, 6(2):447--497, 2013.

\bibitem{contour}
Diego C{\'o}rdoba and Francisco Gancedo.
\newblock Contour dynamics of incompressible 3-{D} fluids in a porous medium
  with different densities.
\newblock {\em Comm. Math. Phys.}, 273(2):445--471, 2007.

\bibitem{3d}
Daniel Coutand and Steve Shkoller.
\newblock On the finite-time splash and splat singularities for the 3-{D}
  free-surface {E}uler equations.
\newblock {\em Comm. Math. Phys.}, 325(1):143--183, 2014.

\bibitem{impossi}
Daniel Coutand and Steve Shkoller.
\newblock On the impossibility of finite-time splash singularities for vortex
  sheets.
\newblock {\em arXiv:1407.1479v1}, 2014.

\bibitem{para}
Joachim Escher and Bogdan-Vasile Matioc.
\newblock On the parabolicity of the {M}uskat problem: well-posedness,
  fingering, and stability results.
\newblock {\em Z. Anal. Anwend.}, 30(2):193--218, 2011.

\bibitem{nosplash}
Charles Fefferman, Alexandru~D. Ionescu, and Victor Lie.
\newblock On the absence of "splash" singularities in the case of two-fluids
  interfaces.
\newblock {\em arXiv:1312.2917v2}, 2013.

\bibitem{absence}
Francisco Gancedo and Robert~M. Strain.
\newblock Absence of splash singularities for surface quasi-geostrophic sharp
  fronts and the muskat problem.
\newblock {\em Proc. Natl. Acad. Sci. USA}, 111(2):635--639, 2014.

\bibitem{Granero-Gomez}
Javier G\'{o}mez-Serrano and Rafael Granero-Belinch\'{o}n.
\newblock On turning waves for the inhomogeneous muskat problem: a
  computer-assisted proof.
\newblock {\em Nonlinearity}, 27(6):1471--1498, 2014.

\bibitem{Granero}
Rafael Granero-Belinch\'{o}n.
\newblock Global existence for the confined muskat problem.
\newblock {\em {SIAM} J. Math. Analysis}, 46(2):1651--1680, 2014.

\bibitem{commutator}
Tosio Kato and Gustavo Ponce.
\newblock Commutator estimates and the {E}uler and {N}avier-{S}tokes equations.
\newblock {\em Comm. Pure Appl. Math.}, 41(7):891--907, 1988.

\bibitem{muskat}
Morris Muskat.
\newblock Two fluid systems in porous media. {T}he encroachment of water into
  an oil sand.
\newblock {\em J. Appl. Phys.}, (5):250--264, 1934.

\bibitem{ray}
Lord Rayleigh.
\newblock On {T}he {I}nstability {O}f {J}ets.
\newblock {\em Proc. London Math. Soc.}, S1-10(1):4.

\bibitem{tay}
P.~G. Saffman and Geoffrey Taylor.
\newblock The penetration of a fluid into a porous medium or {H}ele-{S}haw cell
  containing a more viscous liquid.
\newblock {\em Proc. Roy. Soc. London. Ser. A}, 245:312--329. (2 plates), 1958.

\bibitem{global2}
Michael Siegel, Russel~E. Caflisch, and Sam Howison.
\newblock Global existence, singular solutions, and ill-posedness for the
  {M}uskat problem.
\newblock {\em Comm. Pure Appl. Math.}, 57(10):1374--1411, 2004.

\end{thebibliography}
\bibliographystyle{plain}

\end{document}